\documentclass[american]{article}
\usepackage[T1]{fontenc}
\usepackage[latin9]{inputenc}
\usepackage{color}
\usepackage{array}
\usepackage{float}
\usepackage{mathtools}
\usepackage{bm}
\usepackage{amsmath}
\usepackage{amsthm}
\usepackage{amssymb}
\usepackage[margin=0.9in]{geometry}
\usepackage[nice]{nicefrac}

\makeatletter

\providecommand{\tabularnewline}{\\}
\floatstyle{ruled}
\newfloat{algorithm}{tbp}{loa}
\providecommand{\algorithmname}{Algorithm}
\floatname{algorithm}{\protect\algorithmname}

\numberwithin{equation}{section}
\numberwithin{figure}{section}
\theoremstyle{plain}
\newtheorem{thm}{\protect\theoremname}
\theoremstyle{definition}
\newtheorem{defn}[thm]{\protect\definitionname}
\theoremstyle{plain}
\newtheorem{assumption}[thm]{\protect\assumptionname}
\theoremstyle{remark}
\newtheorem{rem}[thm]{\protect\remarkname}
\theoremstyle{plain}
\newtheorem{lemma}[thm]{\protect\lemmaname}
\newtheorem{corollary}[thm]{\protect\corollaryname}

\usepackage{algorithmic}
\usepackage[backref=page]{hyperref}

\makeatother

\usepackage{babel}
\addto\captionsamerican{\renewcommand{\lemmaname}{Lemma}}
\addto\captionsamerican{\renewcommand{\assumptionname}{Assumption}}
\addto\captionsamerican{\renewcommand{\definitionname}{Definition}}
\addto\captionsamerican{\renewcommand{\remarkname}{Remark}}
\addto\captionsamerican{\renewcommand{\theoremname}{Theorem}}
\addto\captionsenglish{\renewcommand{\lemmaname}{Lemma}}
\addto\captionsenglish{\renewcommand{\algorithmname}{Algorithm}}
\addto\captionsenglish{\renewcommand{\assumptionname}{Assumption}}
\addto\captionsenglish{\renewcommand{\definitionname}{Definition}}
\addto\captionsenglish{\renewcommand{\remarkname}{Remark}}
\addto\captionsenglish{\renewcommand{\theoremname}{Theorem}}
\providecommand{\lemmaname}{Lemma}
\providecommand{\assumptionname}{Assumption}
\providecommand{\definitionname}{Definition}
\providecommand{\remarkname}{Remark}
\providecommand{\theoremname}{Theorem}
\providecommand{\corollaryname}{Corollary}

\begin{document}
\selectlanguage{english}%

\newcommand{\barak}[1]{\textcolor{blue}{[Barak: #1]}}
\newcommand{\haim}[1]{\textcolor{red}{[Haim: #1]}}
\newcommand{\boris}[1]{\textcolor{red}{[Boris: #1]}}

\global\long\def\R{\mathbb{R}}%

\global\long\def\e{{\mathbf{e}}}%

\global\long\def\et#1{{\e(#1)}}%

\global\long\def\ef{{\mathbf{\et{\cdot}}}}%

\global\long\def\x{{\mathbf{x}}}%

\global\long\def\w{{\mathbf{w}}}%

\global\long\def\m{{\mathbf{m}}}%

\global\long\def\xt#1{{\x(#1)}}%

\global\long\def\xf{{\mathbf{\xt{\cdot}}}}%

\global\long\def\d{{\mathbf{d}}}%

\global\long\def\b{{\mathbf{b}}}%

\global\long\def\k{{\mathbf{k}}}%

\global\long\def\a{{\mathbf{a}}}%

\global\long\def\q{{\mathbf{q}}}%

\global\long\def\u{{\mathbf{u}}}%

\global\long\def\y{{\mathbf{y}}}%

\global\long\def\0{{\mathbf{0}}}%

\global\long\def\yt#1{{\y(#1)}}%

\global\long\def\yf{{\mathbf{\yt{\cdot}}}}%

\global\long\def\z{{\mathbf{z}}}%

\global\long\def\w{{\mathbf{w}}}%

\global\long\def\v{{\mathbf{v}}}%

\global\long\def\h{{\mathbf{h}}}%

\global\long\def\s{{\mathbf{s}}}%

\global\long\def\c{{\mathbf{c}}}%

\global\long\def\p{{\mathbf{p}}}%

\global\long\def\f{{\mathbf{f}}}%

\global\long\def\rb{{\mathbf{r}}}%

\global\long\def\rt#1{{\rb(#1)}}%

\global\long\def\rf{{\mathbf{\rt{\cdot}}}}%

\global\long\def\mat#1{{\ensuremath{\bm{\mathrm{#1}}}}}%

\global\long\def\matN{\ensuremath{{\bm{\mathrm{N}}}}}%

\global\long\def\matA{\ensuremath{{\bm{\mathrm{A}}}}}%

\global\long\def\matB{\ensuremath{{\bm{\mathrm{B}}}}}%

\global\long\def\matC{\ensuremath{{\bm{\mathrm{C}}}}}%

\global\long\def\matD{\ensuremath{{\bm{\mathrm{D}}}}}%

\global\long\def\matG{\ensuremath{{\bm{\mathrm{G}}}}}%

\global\long\def\matP{\ensuremath{{\bm{\mathrm{P}}}}}%

\global\long\def\matU{\ensuremath{{\bm{\mathrm{U}}}}}%

\global\long\def\matV{\ensuremath{{\bm{\mathrm{V}}}}}%

\global\long\def\matW{\ensuremath{{\bm{\mathrm{W}}}}}%

\global\long\def\matM{\ensuremath{{\bm{\mathrm{M}}}}}%

\global\long\def\matZ{\ensuremath{{\bm{\mathrm{Z}}}}}%

\global\long\def\matR{\mat R}%

\global\long\def\matQ{\mat Q}%

\global\long\def\matS{\mat S}%

\global\long\def\matSb{\mat{S_{B}}}%

\global\long\def\matSw{\mat{S_{w}}}%

\global\long\def\matY{\mat Y}%

\global\long\def\matX{\mat X}%

\global\long\def\matYhat{\hat{\mat Y}}%

\global\long\def\matXhat{\hat{\mat X}}%

\global\long\def\matI{\mat I}%

\global\long\def\matzero{\mat 0}%

\global\long\def\matJ{\mat J}%

\global\long\def\matZ{\mat Z}%

\global\long\def\matL{\mat L}%

\global\long\def\S#1{{\mathbb{S}_{N}[#1]}}%

\global\long\def\IS#1{{\mathbb{S}_{N}^{-1}[#1]}}%

\global\long\def\PN{\mathbb{P}_{N}}%

\global\long\def\TNormS#1{\|#1\|_{2}^{2}}%

\global\long\def\TNorm#1{\|#1\|_{2}}%

\global\long\def\InfNorm#1{\|#1\|_{\infty}}%

\global\long\def\FNorm#1{\|#1\|_{F}}%

\global\long\def\FNormS#1{\|#1\|_{F}^{2}}%

\global\long\def\UNorm#1{\|#1\|_{\matU}}%

\global\long\def\UNormS#1{\|#1\|_{\matU}^{2}}%

\global\long\def\UINormS#1{\|#1\|_{\matU^{-1}}^{2}}%

\global\long\def\ANorm#1{\|#1\|_{\matA}}%

\global\long\def\BNorm#1{\|#1\|_{\mat B}}%

\global\long\def\ANormS#1{\|#1\|_{\matA}^{2}}%

\global\long\def\AINormS#1{\|#1\|_{\matA^{-1}}^{2}}%

\global\long\def\BINormS#1{\|#1\|_{\matB^{-1}}^{2}}%

\global\long\def\BINorm#1{\|#1\|_{\matB^{-1}}}%

\global\long\def\ONorm#1#2{\|#1\|_{#2}}%

\global\long\def\T{\textsc{T}}%

\global\long\def\pinv{\textsc{+}}%

\global\long\def\Expect#1{{\mathbb{E}}\left[#1\right]}%

\global\long\def\ExpectC#1#2{{\mathbb{E}}_{#1}\left[#2\right]}%

\global\long\def\dotprod#1#2{\left\langle #1,#2 \right\rangle}%

\global\long\def\dotprodX#1#2#3{(#1,#2)_{#3}}%

\global\long\def\dotprodM#1#2{(#1,#2)_{\matM}}%

\global\long\def\dotprodsqr#1#2{(#1,#2)^{2}}%

\global\long\def\Trace#1{{\bf Tr}\left(#1\right)}%

\global\long\def\Range#1{{\bf Range}\left(#1\right)}%

\global\long\def\dist#1{{\bf dist}\left(#1\right)}%

\global\long\def\vectorization#1{{\bf vec}\left(#1\right)}%

\global\long\def\vecskew#1{{\bf vec_{skew}}\left(#1\right)}%

\global\long\def\nnz#1{{\bf nnz}\left(#1\right)}%

\global\long\def\blockdiag#1{{\bf blkdiag}\left(#1\right)}%

\global\long\def\vol#1{{\bf vol}\left(#1\right)}%

\global\long\def\rank#1{{\bf rank}\left(#1\right)}%

\global\long\def\reach#1{{\bf rch}\left(#1\right)}

\global\long\def\span#1{{\bf Span}#1}%

\global\long\def\diag#1{{\bf diag}\left(#1\right)}%

\global\long\def\grad#1{{\bf grad}#1}%
\global\long\def\gradRD#1{{\bf grad_{\R^{D}}}#1}%
\global\long\def\gradM#1{{\bf grad_{\mathcal{M}}}#1}%
\global\long\def\gradtil#1{{\bf \widetilde{grad}}#1}%
\global\long\def\gradMtil#1{{\bf grad_{\widetilde{\mathcal{M}}}}#1}%
\global\long\def\gradhat#1{{\bf \widehat{grad}}#1}%
\global\long\def\gradMhat#1{{\bf grad_{\widehat{\mathcal{M}}}}#1}%

\global\long\def\hess#1{{\bf Hess}#1}%

\global\long\def\sym#1{{\bf sym}#1}%

\global\long\def\skew#1{{\bf skew}#1}%

\global\long\def\st{\,\,\,\text{s.t.}\,\,\,}%

\global\long\def\elp{\mathbb{S}^{\matB}}%

\global\long\def\elpCCA{\mathbb{S}_{\x\y}}%

\global\long\def\elpa{\mathbb{S}^{\mat A}}%

\global\long\def\elplad{\mathbb{S}^{\mat{\mat{S_{w}}+\lambda\matI_{d}}}}%

\global\long\def\Stiefellda{{\bf St}_{(\mat{S_{w}}+\lambda\matI_{d})}(p,d)}%

\global\long\def\ldaB{\mat{S_{w}}+\lambda\matI_{d}}%

\global\long\def\elpsigx{\mathbb{S}^{\mat{\Sigma_{\x\x}}}}%

\global\long\def\elpsigy{\mathbb{S}^{\mat{\Sigma_{\y\y}}}}%

\global\long\def\elpparam#1{\mathbb{S}^{#1}}%

\global\long\def\poly#1{{\bf poly}\left(#1\right)}%

\global\long\def\id{{\bf id}}%

\global\long\def\stiefel{{\bf St}}%

\global\long\def\stiefelB{{\bf St}_{\matB}}%

\global\long\def\qf#1{{\bf qf}\left(#1\right)}%

\global\long\def\qfm#1#2{{\bf qf}_{#2}\left(#1\right)}%

\global\long\def\qfmsmall#1#2{{\bf qf}_{#2}(#1)}%

\global\long\def\fcca{f_{{\bf CCA}}(\matZ)}%

\global\long\def\justfcca{f_{{\bf CCA}}}%

\global\long\def\sigmacca{\Sigma_{\nabla^{2}f_{{\bf CCA}}}}%

\global\long\def\flda{f_{{\bf LDA}}(\matW)}%

\global\long\def\justflda{f_{{\bf LDA}}}%

\global\long\def\nicehalf{\nicefrac{1}{2}}%

\global\long\def\tmmls{T_{\text{MMLS}}}%

\global\long\def\mmls{\text{MMLS}}%

\selectlanguage{american}%

\title{Manifold Free Riemannian Optimization}
\author{Boris Shustin\thanks{Tel-Aviv University (borisshy@mail.tau.ac.il, haimav@tauex.tau.ac.il).}, Haim Avron\footnotemark[1], and Barak Sober\thanks{The Hebrew University of Jerusalem (barak.sober@mail.huji.ac.il)}}
\maketitle
\begin{abstract}
Riemannian optimization is a principled framework for solving optimization problems where the desired optimum is constrained to a smooth manifold $\mathcal{M}$. Algorithms designed in this framework usually require some geometrical description of the manifold, which typically includes tangent spaces, retractions, and gradients of the cost function. However, in many cases, only a subset (or none at all) of these elements can be accessed due to lack of information or intractability. In this paper, we propose a novel approach that can perform approximate Riemannian optimization in such cases, where the constraining manifold is a submanifold of $\R^{D}$. At the bare minimum, our method requires only a noiseless sample set of the cost function $(\x_{i}, y_{i})\in {\mathcal{M}} \times \mathbb{R}$ and the intrinsic dimension of the manifold $\mathcal{M}$. Using the samples, and utilizing the Manifold-MLS framework \cite{sober2020manifold}, we construct approximations of the missing components entertaining provable guarantees and analyze their computational costs. In case some of the components are given analytically (e.g., if the cost function and its gradient are given explicitly, or if the tangent spaces can be computed), the algorithm can be easily adapted to use the accurate expressions instead of the approximations. We analyze the global convergence of Riemannian gradient-based methods using our approach, and we demonstrate empirically the strength of this method, together with a conjugate-gradients type method based upon similar principles.
\end{abstract}

\section{Introduction}

Non-convex constrained optimization problems are prevalent across multiple areas in science, physics, economics, climate modeling and many other fields. Through history, various methods and vast literature dealt with proposing algorithm for solving constrained optimization problems, e.g., projected gradient method; sequential quadratic programming; proximal point method; penalty, barrier, and
augmented Lagrangian
methods, and many others \cite{NoceWrig06}. In many applications, the constraint set is a low dimensional manifold, e.g., eigenvalue problems, principal component analysis, low-rank matrix completion, pose estimation and motion recovery. Explicitly, consider an optimization problem of the form 
\begin{equation}
\min_{\x\in{\mathcal{M}}}f(\x)\label{eq:general_problem}
\end{equation}
where ${\mathcal{M}}$ is a compact and boundaryless smooth $d$-dimensional
submanifold of $\R^{D}$, and $f:\R^{D}\to\R$ is a sufficiently
smooth cost function (i.e., gradient Lipschitz). Note that, essentially the minimization is of the restriction of $f$ on $\mathcal{M}$, i.e., $f|_{\mathcal{M}}$. Several approaches were developed over
the years to solve Problem~\eqref{eq:general_problem}. Most relevant to our work is the extensive literature, e.g.~\cite{luenberger1972gradient,smith1994optimization,EAS98}, which developed the so-called Riemannian optimization framework for solving Problem\eqref{eq:general_problem} for various constraining manifolds.
For recent surveys, see \cite{AMS09,boumal2022intromanifolds}.
Key to the success of this framework are advances in numerical linear algebra (e.g., matrix factorization),
and the introduction of tractable geometric components, 
which enabled progress in this field and allowed the development of effective algorithms.
The Riemannian optimization framework proved especially effective for matrix manifolds (i.e., manifolds constructed from $\R^{D_{1}\times D_{2}}$ as embedded submanifolds or quotient manifolds), such as the Stiefel manifold, the Grassmann
manifold of subspaces, the cone of positive definite matrices, and even
the Euclidean space.

The main idea in the Riemannian optimization framework is as follows. Given a
constraining manifold, use Riemannian geometry to develop geometrical
components which allow modifying iterative methods for solving unconstrained
optimization problems to solve the constrained case by viewing it as an unconstrained problem with a manifold geometry. Much effort
was invested in developing these components for frequently arising
manifolds and finding ways to efficiently compute them (see Subsection \ref{subsec:Riemannian-Optimization}).

In the aforementioned line of research either the manifold ${\mathcal{M}}$
is explicitly available, or the computation of the components
is possible directly from the given constraints. However, a much less investigated problem setting is in which the given constraints form a manifold ${\mathcal{M}}$ for which the geometric components are intractable or cannot be formed explicitly, then the manifold ${\mathcal{M}}$, or at the very least the relevant components, need to be approximated. 
In this paper, we tackle three specific scenarios for Problem \eqref{eq:general_problem} in which the manifold ${\mathcal{M}}$ is unknown, in the sense that it is given only implicitly. The first two are the main focus of this paper and we demonstrate both of them empirically. The last scenario is mentioned as a remark of a possible application of our method, and we do not demonstrate it, but our analysis is valid for it as well. 

In the strictest scenario, at the bare minimum, we require only a noiseless {\em quasi-uniform sample set} (see Definition \ref{def:Quasi-uniform-sample-set}) with respect to the domain $\mathcal{M}$ of the cost function and $\mathcal{M}$ (see assumption \ref{assu:MMLS_Samples} and \ref{assu:MMLS_func_Samples}), and knowledge of the intrinsic dimension of the constraining manifold. In a simpler scenario, we assume in addition to the previous requirements that we have access to the cost function and its gradient in $\R^{D}$. Finally, in the simplest scenario we also assume that we have access to the tangent spaces of the constraining manifold at least at the points of the sample set. Note that the last scenario exists in real world applications such as in digital imaging (e.g., \cite{hug2017voronoi,klette2004digital}), where for a finite sample set it is possible to determine precisely tangent spaces using high resolution images or $3$D-scans.

Our method builds upon a recently developed method
for approximating manifolds, ``Manifold Moving Least-Squares (MMLS) Projection''
\cite{sober2020manifold}. Given a set of samples of ${\mathcal{M}}$ which forms a quasi-uniform sample set with respect to the domain $\mathcal{M}$, the dimension
of the manifold $d$, and a point $\rb\in\R^{D}$ close enough to the
manifold (see Assumption \ref{assu:Uniqueness-domain}), MMLS
algorithm performs two stages of approximation resulting in an approximate
``projection'' of $\rb$ to the manifold ${\mathcal{M}}$ (see Subsection
\ref{subsec:Manifold-Moving-Least-Squares}). It is shown in \cite[Theorem 4.21]{sober2020manifold}
that the set of all MMLS projections of the points of ${\mathcal{M}}$,
denoted by $\widetilde{{\mathcal{M}}}$, is in itself almost everywhere\footnote{We say $\widetilde{{\mathcal{M}}}$ is almost everywhere a smooth $d$-dimensional
manifold, when for all $\rb\in \widetilde{{\mathcal{M}}}$ but a set of measure $0$ there exists a local diffeomorphism around $\rb$ to another smooth $d$-dimensional
manifold.} a smooth $d$-dimensional
manifold approximating ${\mathcal{M}}$ and MMLS projection of $\rb$
belongs to $\widetilde{{\mathcal{M}}}$. 

MMLS provides a local $d$-dimensional coordinate system and an
origin, from which a polynomial approximation of a parametrization
of ${\mathcal{M}}$ is constructed, viewing the manifold locally as a graph
of a function from a $d$-dimensional space to a $(D-d)$-dimensional
space. MMLS can be utilized to define an approximate tangent space and some of
the geometric components required for optimization. We present
our proposed approximations in Section \ref{sec:The-Proposed-Algorithms}. Thus, our proposed algorithm can be utilized for previously unsolved problems, zeroth-order optimization in the sense of constraint set accessibility.

Furthermore, we also provide a fully zeroth-order optimization method with respect to the cost function accessibility. Suppose that the cost function $f$ is only given
via samples of the values of $f$ on ${\mathcal{M}}$ (or its gradient is unknown, while samples of $f$ are available), the extension of
MMLS for approximating functions over manifolds \cite{sober2021approximation}
is utilized for approximating $f$ (if required) and its Euclidean gradient\footnote{Note that the gradient approximation is an extension we present here, see Lemma \ref{lem:fderiv_approx} in Appendix \ref{subsec:background_MMLS}.}.

We theoretically analyze the global convergence of a Riemannian gradient
method based on MMLS in Section \ref{sec:Convergence-Analysis-of}, following a similar analysis as in \cite{boumal2019global}
and using the results presented in \cite{sober2020approximating,aizenbud2021non}. In our analysis we
do not assume convexity in the usual or the Riemannian sense \cite[Section 11]{boumal2022intromanifolds}
in order to not restrict our analysis, and also since continuous and
convex (in the Riemannian sense) functions over connected and compact manifolds (as we assume for $\mathcal{M}$ in this paper) are constant functions \cite[Corollary 11.10]{boumal2022intromanifolds}.
Finally, we demonstrate empirically our proposed MMLS based geometrical components
both for a Riemannian gradient method based on MMLS, and also for a
Riemannian conjugate-gradients (CG) method based on MMLS.

\begin{rem}[Clean samples]\label{rem:clean_samples}
Our analysis assumes a noiseless (i.e., clean) sample set of the cost function and the constraining manifold. In practice, this assumption can be relaxed since MMLS \cite{sober2020manifold} and its function approximation extention \cite{sober2021approximation} can also be applied on a noisy sample set, but some results in \cite{sober2020manifold,sober2021approximation,sober2020approximating}, e.g., on the approximation order, require clean samples. Thus for the sake of theoretical analysis, we restrict ourselves in this paper to the case in which we have clean samples. 
\end{rem}
\begin{rem}[MMLS alternatives]
Our proposed method builds upon MMLS. Note that other methods for manifold learning which provide similar tools as MMLS (see Assumption \ref{assu:list_of_req}) can also be used to produce approximations of the various geometric components we propose here (see Table \ref{tab:riemannian-approx}). Moreover, the analysis we perform in Subsection \ref{subsec:Convergence-Analysis-of1} is general and can be extended for the use of other manifold learning methods satisfying Assumption \ref{assu:list_of_req}. 
\end{rem}

\subsection{Contributions}
The main contributions in this paper are:
\begin{itemize}
    \item For the case where only a quasi-uniform sample set a with respect to the domain $\mathcal{M}$ of the cost function and $\mathcal{M}$ is given, and the intrinsic dimension of the constraining manifold is known, i.e., zeroth-order optimization where both the $f$ and $\mathcal{M}$ are accessed only via samples, we provide approximations using MMLS for the following geometric components (see Section \ref{sec:The-Proposed-Algorithms}): tangent spaces, retraction map, orthogonal projection on the tangent spaces, Riemannian gradient (including an approximation of the Euclidean gradient, see Lemma \ref{lem:fderiv_approx} in Appendix \ref{subsec:background_MMLS}), and vector transport. The aforementioned components allow us to apply any standard first-order Riemannian optimization algorithm. 
    \item For the case where a quasi-uniform sample set with respect to the domain $\mathcal{M}$ is given, its intrinsic dimension is known, and $f$ and its Euclidean gradient are explicitly known, i.e., zeroth-order optimization in the sense that the constraint $\mathcal{M}$ is only accessed via samples but first-order with respect to the cost function, we provide approximations using MMLS for the following geometric components (see Section \ref{sec:The-Proposed-Algorithms}): tangent spaces, retraction map, orthogonal projection on the tangent spaces, Riemannian gradient, and vector transport. As in the previous case, the aforementioned components allow us to apply any standard first-order Riemannian optimization algorithm.
    \item For the case where in addition the tangent spaces are given at the sample set points of $\mathcal{M}$ we provide a simpler algorithm, based on the fact that at these points the first step of MMLS is unnecessary, since these tangent spaces provide a local coordinate system over which the second step of MMLS can be performed. Thus, the second step of MMLS projection can be used as an approximation of a retraction map. 
\end{itemize}
For all the cases above, we provide computational costs of the approximated components. In addition, our analysis in Subsection \ref{subsec:Convergence-Analysis-of2} of a Riemannian gradient algorithm based on MMLS is valid. Moreover, we exemplify a gradient algorithm based on MMLS and a Riemannian CG algorithm based on MMLS empirically for the first two cases. 

\subsection{Related Work}
Our work combines Riemannian optimization and MMLS, and more generally methods for manifold learning which are similar to MMLS (in the sense of Assumption \ref{assu:list_of_req}). Here, we briefly summarize the most relevant
prior work on each of these subjects. 

\paragraph{Riemannian Optimization.} Riemannian optimization is a framework aimed at solving problems of the form of Problem~\eqref{eq:general_problem}. There exists an extensive literature on Riemannian optimization, starting from the early works in \cite{luenberger1972gradient,smith1994optimization,EAS98}, and more recently the surveys \cite{AMS09,boumal2022intromanifolds}. In particular, in this work we focus on optimization problems where there is limited information both on the constraining manifold $\mathcal{M}$ and possibly on the cost function as well, i.e., accessed only via samples. Most works in the field of Riemannian optimization related to zeroth-order optimization are generalizations of unconstrained zeroth-order optimization methods, i.e., the manifold is explicitly available but the cost function is accessed only via samples, see for example \cite{li2020zeroth,chattopadhyay2015derivative,maass2020online,fong2019framework,fong2019extended,yao2021riemannian}. Unlike the aforementioned works, our main contribution is in tackling problems where the constraining manifold $\mathcal{M}$ and its geometric components cannot be explicitly formed, forcing their approximation via samples of $\mathcal{M}$. To the best of our knowledge our work is the first which tackles such case, with the aim of applying the framework of Riemannian optimization. 

\paragraph{Manifold Learning and MMLS.} Manifold learning is a thoroughly studied problem, with widespread applications. The goal of manifold learning is to find an embedding of high dimensional data in a low dimensional space, where it is assumed that the data resides in an underlying low dimensional manifold. Some well-known algorithms include: Isomap \cite{tenenbaum2000global}, Local Linear Embedding \cite{roweis2000nonlinear,saul2003think}, Laplacian Eigenmaps \cite{belkin2003laplacian}, Diffusion maps \cite{coifman2006diffusion}, and $t$-distributed stochastic neighbor embedding \cite{van2008visualizing}. Unlike the aforementioned methods which aim at finding a global embedding of the data, in this paper we utilize MMLS algorithm \cite{sober2020manifold} which provides local approximation of the manifold in the ambient space. This property is important in our proposed algorithm since it allows the approximation of local geometrical structure of the underlying manifold such as the tangent spaces. MMLS method is a generalization of \cite{levin2004mesh}, where surfaces approximation using moving least-squares (MLS) was presented. MMLS was extended to approximation of functions over manifolds \cite{sober2021approximation}. Also, the properties of the differential of the approximation were studied in \cite{sober2020approximating}. An extension to MMLS was presented in \cite{aizenbud2021non}, where full analysis of the method was performed under noisy data assumption.

Other works close in spirit to MMLS, i.e., form a local coordinate system over which some regression over some unknown underlying manifold is solved locally, are \cite{bickel2007local,cheng2013local,lin2021functional}, where regression problems under manifold constraints are solved in a finite ambient dimension \cite{bickel2007local,cheng2013local}, and in an infinite ambient dimension \cite{lin2021functional}. Another recent work is \cite{dunson2021inferring}, which aims at preprocessing noisy data via a denoising process before forming a local coordinates system, over which a local Gaussian process regression is performed to approximate the manifold. 

\section{Preliminaries}
In this section we recall some relevant basic notions from Riemannian optimization, and from MMLS method for approximating manifolds and functions over manifolds. Throughout the paper, we denote the standard Euclidean norm and the corresponding induced matrix norm (the spectral norm), i.e., $\left\|\cdot\right\|_{2}$, by $\left\|\cdot\right\|$ for short. For any other norm, we explicitly write its type, e.g., $\left\|\cdot\right\|_{\infty}$. We denote an open ball with a radius $\delta>0$ around $\x$ by $B_{\delta}(\x)$. Given some matrix $\matA\in\R^{n\times d}$, we denote its range (column space) by $\Range{\matA}$, and we denote its corresponding Gram matrix by $\matG_{\matA} \coloneqq \matA^{\T} \matA \in \R^{d\times d}$.

\subsection{\label{subsec:Riemannian-Optimization}Riemannian Optimization}

In this subsection we recall some basic definitions of Riemannian
geometry and Riemannian optimization. A \emph{Riemannian manifold} ${\mathcal{M}}$
is a real differentiable manifold ${\mathcal{M}}$ with a smoothly varying
inner product $g_{\x}(\cdot,\cdot):T_{\x}{\mathcal{M}} \times T_{\x}{\mathcal{M}} \to \R$ on its tangent spaces $T_{\x}{\mathcal{M}}$ where
$\x\in{\mathcal{M}}$, denoted by $(\mathcal{M},g)$. The tangent bundle is defined by \cite[Definition 3.42]{boumal2022intromanifolds}:
\begin{equation}\label{eq:tan_bun}
    T{\mathcal{M}} \coloneqq \left\{(\x, \v)\ :\ \x\in {\mathcal{M}}\ \cap \ \v \in T_{\x}{\mathcal{M}}\right\}.
\end{equation}

Next, we recall the definition of a Riemannian submanifold. ${\mathcal{M}}$ is a \emph{Riemannian submanifold} of $\bar{\mathcal{M}}$ (called the \emph{embedding manifold} or the \emph{ambient space}) if it is an embedded submanifold of $\bar{{\mathcal{M}}}$, $\bar{{\mathcal{M}}}$ is a Riemannian manifold $(\bar{{\mathcal{M}}},\bar{g})$, and the
Riemannian metric $g$ on ${\mathcal{M}}$ is induces by the Riemannian metric $\bar{g}$ on $\bar{{\mathcal{M}}}$: $$g_{\x}(\eta_{\x},\xi_{\x})\coloneqq\bar{g}_{\x}(\eta_{\x},\xi_{\x}),$$
for $\eta_{\x},\xi_{\x}\in T_{\x}{\mathcal{M}}$ where in the right-side
$\eta_{\x}$ and $\xi_{\x}$ are viewed as elements in $T_{\x}\bar{{\mathcal{M}}}$,
and $T_{\x}{\mathcal{M}}$ is viewed as a subspace of $T_{\x}\bar{{\mathcal{M}}}$ \cite[Section 3.6.1]{AMS09}. The notion of Riemannian submanifolds
is central in this paper since our proposed method is aimed at Riemannian submanifolds of Euclidean spaces.

A general form of an iterative Riemannian optimization algorithm on a Riemannian submanifold
${\mathcal{M}}$ of a Euclidean space appears in Algorithm \ref{alg:general_R_opt}  (see \cite{AMS09,boumal2022intromanifolds}). Algorithm \ref{alg:general_R_opt} fits the form of various iterative Riemannian optimization methods such as Riemannian gradient
methods (illustrated in Fig. \ref{fig:illustration}) and Riemannian Newton methods.

\begin{algorithm}[tb]
\caption{\label{alg:general_R_opt}General form of an iterative Riemannian optimization method.}

\begin{algorithmic}[1]

\STATE\textbf{Input: }$f:\mathcal{M}\to\R$ a smooth function and a Riemannian manifold $({\mathcal{M}},g)$.

\STATE Choose an initial point $\x_{0}\in{\mathcal{M}}$.

\STATE Locally approximate the manifold ${\mathcal{M}}$ by its tangent space
at $\x_{0}$, i.e., $T_{\x_{0}}{\mathcal{M}}$.

\STATE Pick a search direction on the tangent space $\xi_{\x_{0}}\in T_{\x_{0}}{\mathcal{M}}$,
and a step-size, $\alpha_{0}>0$, satisfying some criteria depending
on the algorithm.

\STATE Retract $\alpha_{0}\xi_{\x_{0}}\in T_{\x_{0}}{\mathcal{M}}$ to the manifold,
via a \emph{retraction mapping}, $R_{\x_{0}}(\cdot):T_{\x_{0}}{\mathcal{M}}\to{\mathcal{M}}$,
i.e., $\x_{1}=R_{\x_{0}}(\alpha_{0}\xi_{\x_{0}})$ (see Definition \ref{def:retraction}). 

\STATE Repeat for $\x_{i}$, $i=1,2,...$ until some stopping criteria is satisfied.

\end{algorithmic}
\end{algorithm}
\begin{figure}[t]
\centering{}\includegraphics[scale=0.35]{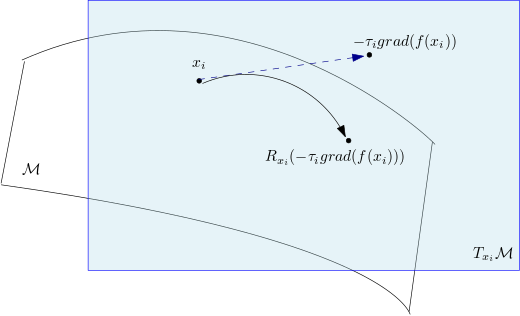}\caption{Illustration of Riemannian gradient-descent.\label{fig:illustration}}
\end{figure}

Next, recall the definition of a retraction map \cite[Definition 4.1.1]{AMS09}:
\begin{defn}[Retraction]\label{def:retraction}
A \textbf{retraction}
$R_{(\cdot)}(\cdot):T{\mathcal{M}}\to {\mathcal{M}}$ on a manifold $\mathcal{M}$ is a smooth mapping from the tangent bundle $T{\mathcal{M}}$ onto
$\mathcal{M}$, such that its restriction $R_{\x}$ to $T_{\x}{\mathcal{M}}$
satisfies the following conditions:
\begin{enumerate}
    \item $R_{\x}(\0_{\x})=\x$ where $\0_{\x}\in T_{\x}{\mathcal{M}}$
is the zero element of $T_{\x}{\mathcal{M}}$.
    \item $\text{D}R_{\x}(\0_{\x})=\text{Id}_{T_{\x}{\mathcal{M}}}$ where $\text{Id}_{T_{\x}{\mathcal{M}}}$
is the identity mapping on $T_{\x}{\mathcal{M}}$. 
\end{enumerate}
\end{defn}
The conditions in Definition \ref{def:retraction} ensure that a retraction is at least a first-order approximation of the exponential
mapping \cite[Section 5.4]{AMS09}, which is in itself also a retraction moving along geodesics (see \cite[Proposition 10.17]{boumal2022intromanifolds}). 

Additional important components for Riemannian optimization are: the Riemannian gradient\footnote{In the paper we use additional notations for the Riemannian gradient on different manifolds.} $\gradM{f(\x)}\in T_{\x}{\mathcal{M}}$
\cite[Section 3.6]{AMS09}, the Riemannian connection \cite[Section 5.3]{AMS09}
and the Riemannian Hessian $\hess{f(\x)}:T_{\x}{\mathcal{M}}\to T_{\x}{\mathcal{M}}$
\cite[Section 5.5]{AMS09}. These notions can be made explicit in a simple way if $\mathcal{M}$ is a submanifold of a Euclidean space (e.g., $\R^{D}$) and $f$ is given
in ambient coordinates. In this paper, we focus on the Riemannian gradient, since we present first-order algorithms. Thus, we show its derivation explicitly. Suppose we define the Riemannian metric via
the standard inner product on $\R^{D}$, i.e., 
\[
\bar{g}_{\x}(\eta_{\x},\xi_{\x})=\eta_{\x}^{\T}\xi_{\x}\ .
\]
Given a smooth function $f:\R^{D}\to\R$, its Euclidean gradient, $\nabla f(\x)$, and its Riemannian gradient on $\R^{D}$, i.e., $\gradRD{f(\x)}$, are equal. Then, the Riemannian gradient of $f$ on ${\mathcal{M}}$ is simply the orthogonal
projection of $\gradRD{f(\x)}$ on $T_{\x}{\mathcal{M}}$, i.e., 
\begin{equation}
\gradM{f(\x)} = \Pi_{\x}(\gradRD{f(\x)}) = \Pi_{\x}(\nabla f(\x)),\label{eq:R_grad}
\end{equation}
where $\Pi_{\x}(\cdot)$ denotes the orthogonal projection on
$T_{\x}{\mathcal{M}}$ \cite[Eq. 3.37]{AMS09}. 

Some optimization algorithms require manipulating tangent vectors
from different tangent spaces, e.g., finite difference approximations
and Riemannian CG. To that purpose, the notion of vector transport
$\tau_{\eta_{\x}}\xi_{\x}\in T_{R_{\x}(\eta_{\x})}{\mathcal{M}}$ \cite[Section 8.1]{AMS09}
is used. The notion of vector transport is a relaxation of the noition of parallel transport \cite[Section 5.4]{AMS09}, which is based on movement along geodesics. For a
Riemannian submanifold, a common vector transport is  simply to take
the orthogonal projection on the tangent space with the foot at the
retraction of the first tangent vector of the desirable tangent vector
(see \cite[Section 8.1.3]{AMS09}).

One particularly useful property of the retraction
mapping is that the Euclidean gradient of the pull-back function ($f\circ R_{\x}$)  at the origin of $T_{\x}{\mathcal{M}}$ equals to the
Riemannian gradient of $f$ at $\x$~\cite[Proposition 3.59]{boumal2022intromanifolds}. The aforementioned property plays a crucial role in the convergence analysis of Riemannian optimization algorithms. For example, in the context of this paper, using this property it is shown in \cite{boumal2019global} that if $f$ is bounded below on $\mathcal{M}$ and $f\circ R_{\x}$ has Lipschitz gradient $L_{g}$, then Riemannian gradient-descent with a constant step-size $1/L_{g}$, or with backtracking Armijo line-search, returns a point $\x$ such that $\left\|\gradM{f(\x)}\right\|\leq \varepsilon$ for some $\varepsilon>0$ in $O(1/\varepsilon^{2})$ or $O(1/\varepsilon)$ iterations, depending on the size of the domain of the retraction mapping. In this work, we show a parallel property (Lemma \ref{lem:retraction_2_property}) which leads to a similar analysis (Section \ref{sec:Convergence-Analysis-of}).    

\subsection{\label{subsec:Manifold-Moving-Least-Squares}Manifold Moving Least-Squares
(MMLS) Projection}
In this subsection we recall MMLS algorithm \cite{sober2020manifold} and some results from \cite{sober2020manifold,sober2020approximating}. We also recall the extension of MMLS for function approximations \cite{sober2021approximation}. More details can be found in Appendix \ref{subsec:background_MMLS}.

\subsubsection{\label{subsec:MMLS_the_algorithm}MMLS - The Algorithm}
We begin
with the following definition of the conditions on the sample set of ${\mathcal{M}}$ and possibly $f:\R^{D}\to\R$ on $\mathcal{M}$ \cite[Definition 1]{sober2020approximating}:
\begin{defn}[Quasi-uniform sample
set]\label{def:Quasi-uniform-sample-set} A set of data sites $X=\{\x_{1},...,\x_{n}\}$ is said to be {\em quasi-uniform}
with respect to a domain $\Omega$ and a constant $c_{qu}>0$ if 
\[
\delta_{X}\leq h_{X,\Omega}\leq c_{qu}\delta_{X}\ ,
\]
where $h_{X,\Omega}$ is the {\em fill distance} defined by
\[
h_{X,\Omega}\coloneqq\sup_{\x\in\Omega}\min_{\x_{i}\in X}\|\x-\x_{i}\|\ ,
\]
and $\delta_{X}$ is the {\em separation radius} defined by
\[
\delta_{X}\coloneqq\frac{1}{2}\min_{i\neq j}\|\x_{i}-\x_{j}\|\ .
\]
To keep notation concise, we omit the subscripts from $h$, i.e., $h\coloneqq h_{X,\Omega}$.
\end{defn}

Next we recall the required assumptions for MMLS algorithm, when applied to noiseless data:
\begin{assumption}[Manifold approximation assumptions]
\label{assu:MMLS_Samples}${\mathcal{M}}\in C^{2}$ is a closed (i.e.,
compact and boundaryless) $d$-dimensional submanifold of $\R^{D}$.
The sample set $S=\{\rb_{i}\}_{i=1}^{n}\subset{\mathcal{M}}$ is a quasi-unifom
sample set with respect to the domain ${\mathcal{M}}$, with fill distance
$h$.
\end{assumption}

With these assumption, it is possible to perform MMLS approximation
of ${\mathcal{M}}$. In \cite{sober2020manifold}, Sober and Levin proposed
using the technique of MLS for approximating submanifolds
in $\R^{D}$. Given a point $\rb\in\R^{D}$ close enough to ${\mathcal{M}}$
(see Assumption \ref{assu:Uniqueness-domain}), MMLS projection
is performed in two steps: 
\begin{enumerate}
\item Approximate the sampled points via a local $d$-dimensional affine
space $(\q(\rb),H(\rb))$, where $H(\rb)$ is a linear space and the origin
is set to $\q(\rb)$. Explicitly, $H(\rb)=\span\{\e_{k}\}_{k=1}^{d}$
(an element in the $d$-dimensional Grassmanian of $\R^{D}$, denoted
by $\text{Gr}(d,D)$) where $\{\e_{k}\}_{k=1}^{d}$ is an orthonormal
basis of $H(\rb)$, and the affine space is $\{\q(\rb)+\h\,|\,\h\in H(\rb)\}$. The
affine space is used as a local coordinates system for the second
step.
\item Define the projection $\rb$ via a local polynomial approximation $g:H(\rb)\backsimeq\R^{d}\to\R^{D}$
(of total degree $m$) of ${\mathcal{M}}$ over the new coordinate system
spanning $H(\rb)$, i.e., the projection is defined by ${\mathcal P}_{m}^{h}(\rb)\coloneqq g(\0)$.
The approximation is achieved as follows: denote by $\q_{i}$ the orthogonal
projections of $\rb_{i}-\q(\rb)$ onto $H(\rb)$, then $g$ is a polynomial
approximation of the vector valued function $\varphi_{\rb}(\cdot):H(\rb)\to{\mathcal{M}}$
with the samples $\varphi_{\rb}(\q_{i})=\rb_{i}$, which is found by solving a weighted
least-squares problem. Note that $\varphi_{\rb}(\cdot)$ is a function which takes as an input orthogonal projections of points on ${\mathcal{M}}$ deflected by $\q(\rb)$ on $H(\rb)$, and returns the pre-projected points. Thus, in a small neighborhood of $\q(\rb)$, the function $\varphi_{\rb}(\cdot)$ is unique. To keep notation concise, we omit the subscripts from $\varphi_{\rb}(\cdot)$, i.e., $\varphi(\cdot)$, where the subscript can be concluded from the domain of $\varphi(\cdot)$. 
\end{enumerate}
This approach approximates the manifold ${\mathcal{M}}$ via an approximation
of some local parametrization $\varphi$ of it, i.e., an inverse map of a
coordinate chart. Another view of the aforementioned approach is as follows. A smooth $d$-dimensional manifold
${\mathcal{M}}$ can be viewed locally as a graph of a function from a
$d$-dimensional space to a $(D-d)$-dimensional space; see \cite[Lemma A.12]{aizenbud2021non}
for the local existence of such a representation of ${\mathcal{M}}$ as
a graph of a function from $H(\rb)$ to $H^{\perp}(\rb)$. In that view, as mentioned in
\cite[Remark 3.8]{sober2020manifold} and explicitly performed in
\cite[Algorithm 2]{aizenbud2021non}, the aforementioned approximation is equivalent to finding a local polynomial approximation $g:H(\rb)\backsimeq\R^{d}\to H^{\perp}(\rb)\backsimeq\R^{D-d}$
approximating a representation of ${\mathcal{M}}$ as a graph of a function. 

We now present, the explicit steps of MMLS projection (see
\cite[Section 3.2]{sober2020manifold} for the implementation details
of MMLS projection):

\paragraph{Step 1 - the local coordinate system.}

\begin{equation}\label{eq:MMLS_step1}
\q(\rb),H(\rb)=\arg\min_{\q\in\R^{D},H}\sum_{i=1}^{n}d(\rb_{i}-\q,H)^{2}\theta_{1}(\|\rb_{i}-\q\|),
\end{equation}
where $d(\rb_{i}-\q, H)$ is the Euclidean distance between the point
$\rb_{i}-\q$ and the linear subspace $H$, $\theta_{1}(t)$ is a non-negative
weight function (locally supported or rapidly decreasing
as $t\to\infty$, e.g., a Gaussian). The goal is to find a $d$-dimensional linear subspace
$H(\rb)$, and a point $\q(\rb)\in\R^{d}$
that satisfies Eq. \eqref{eq:MMLS_step1}
under the constraints 
\begin{enumerate}
\item \label{item:MMLS_step1_1}$\rb-\q(\rb)\perp H(\rb)$,
\item \label{item:MMLS_step1_2}$\q\in B_{\mu}(\rb)$,
\item \label{item:MMLS_step1_3}$\#(S\cap B_{h}(\q))\neq0$,
\end{enumerate}
where $B_{h}(\q)$ is an open ball of radius $h$ around $q$, and $h$ is the
fill distance, and $B_{\mu}(\rb)$ from Constraint \ref{item:MMLS_step1_2} is an open ball of radius $\mu$ (later
defined) around $\rb$ limiting the Region Of Interest (ROI). In practice, Constraint \ref{item:MMLS_step1_2} is fulfilled heuristically (see Subsection \ref{subsec:implement_details}).  Constraint \ref{item:MMLS_step1_3} makes sure that there are sample points in the support of $\theta_{1}$, thus avoiding trivial zero solution to Problem \eqref{eq:MMLS_step1}. In practice, Constraint \ref{item:MMLS_step1_3} can be manually checked (see the end of Subsection \ref{subsubsec:manopt_experiments}). According to \cite[Section 3.2]{sober2020manifold}, the cost of this step is $O(Dd^{m})$. The radius
$\mu$ must be limited by the manifold's reach: 
\begin{defn}[Reach]
The {\em reach} of a subset $A$ of $\R^{D}$, is the largest
$\tau$ (possibly $\infty$) such that if $\x\in\R^{D}$ and the distance,
$\dist{A,\x}$, from $\x$ to $A$ is smaller than $\tau$, then $A$
contains a unique point, $P_{A}(\x)\in A$, nearest to $\x$. Then the
reach of $A$ is denoted by $\reach{A}$.
\end{defn}

\begin{defn}[Reach neighborhood of a manifold ${\mathcal{M}}$] The {\em reach neighborhood}
of a manifold ${\mathcal{M}}$ is defined by 
\[
U_{\text{reach}}\coloneqq\{\x\in\R^{D}\,|\,\dist{\x,{\mathcal{M}}}<\reach{\mathcal{M}}\}.
\]
\end{defn}

We assume that ${\mathcal{M}}$ is a manifold with non-zero (positive) reach. Recall
that a closed manifold has a positive reach if and only if it is
differentiable and locally Lipschitz \cite{scholtes2013hypersurfaces},
which falls under our assumptions. Moreover, the region in $H(\rb)$ where ${\mathcal{M}}$ can be viewed as
a graph of a function from $H(\rb)$ to $H^{\perp}(\rb)$ depends on $\reach{\mathcal{M}}$ (see \cite[Lemma A.12]{aizenbud2021non}). To generalize the concept of a
reach neighborhood for a domain where MMLS approximation is unique,
the following assumption is made in \cite[Assumption 3.6]{sober2020manifold}:
\begin{assumption}[Uniqueness domain - Assumption 3.6 from \cite{sober2020manifold}] \label{assu:Uniqueness-domain}
It is assumed that there exists a uniqueness domain, $U_{\mathrm{unique}}$, defined by the largest $\epsilon$-neighborhood
of the manifold ${\mathcal{M}}$ (possibly $\infty$) such that if $\x\in\R^{D}$ and the distance,
$\dist{\x,{\mathcal{M}}}$, from $\x$ to ${\mathcal{M}}$ is smaller than $\epsilon$, then the minimization problem
\eqref{eq:MMLS_step1} has a unique local minimum $\q(\x)\in B_{\mu}(\x)$,
for some constant $\mu<\reach{\mathcal{M}}/2$ which does not depend on $\x$.
\end{assumption}

Indeed, in the limit case where $h\to0$ it is shown in \cite[Lemma 4.4]{sober2020manifold}
that such a uniqueness domain exists for all closed manifolds, where $\epsilon < \reach{\mathcal{M}}/4$ and $\mu = \reach{\mathcal{M}}/2$. Thus, for $h$ small enough, $\mu$ can be approximately $\reach{\mathcal{M}}/2$. Moreover,
for any $\widetilde{\rb}\in U_{\mathrm{unique}}$ that is also $\|\widetilde{\rb}-\q(\rb)\|<\mu$
and $\widetilde{\rb}-q(\rb)\perp H(\rb)$ the problem \eqref{eq:MMLS_step1} for
$\widetilde{\rb}$ yields the same local minimum \cite[Lemma 4.7]{sober2020manifold},
i.e., $\q(\widetilde{\rb})=\q(\rb)$ and $H(\widetilde{\rb})=H(\rb)$. In addition,
$\q(\rb)$ and $H(\rb)$ are smoothly varying functions of $\rb$ when $\theta_{1}(\cdot)\in C^{\infty}$ \cite[Theorem 4.12]{sober2020manifold},
making them a smoothly varying coordinate system.

\paragraph{Step 2 - the MLS projection ${\mathcal P}_{m}^{h}$.}

Now the manifold ${\mathcal{M}}$ is approximated using the local coordinate
system in $B_{\mu}(\rb)$, by approximating the function $\varphi:H(\rb)\to{\mathcal{M}}\subseteq\R^{D}$ (a parametrization of $\mathcal{M}$).
Let $\{\e_{k}\}_{k=1}^{d}$ be an orthogonal basis of $H(\rb)$. Let
$\q_{i}$ be the orthogonal projections of $\rb_{i}-\q(\rb)$ onto $H(\rb)$,
i.e., $\q_{i}=\sum_{k=1}^{d}\left\langle \rb_{i}-\q(\rb),\e_{k}\right\rangle \e_{k}$.
Note that $\rb-\q(\rb)$ is projected to the origin of $H(\rb)$. Next, $\varphi$
is approximated via a polynomial $g:H(\rb)\backsimeq\R^{d}\to\R^{D}$ of
total degree $m$ for $1\leq k\leq D$ (denoted by $g\in\Pi_{m}^{d}$),
using the data points $\varphi_{i}=\varphi(\q_{i})=\rb_{i}$ and MLS.
Explicitly, we have
\begin{equation}\label{eq:MMLS_step2}
g^{\star}(\cdot\ |\ \rb)=\arg\min_{g\in\Pi_{m}^{d}}\sum_{i=1}^{n}\|g(\q_{i})-\varphi_{i}\|^{2}\theta_{2}(\rb_{i}-\q(\rb)),
\end{equation}
where $\theta_{2}(\cdot)$ is
a fast decaying radial weight function consistent across scales, i.e., $\theta_{2}(\rb_{i}-\q(\rb)) = \theta_{h}(\|\rb_{i}-\q(\rb)\|)$ and $\theta_{h}(th)=\Phi(t)$. The projection ${\mathcal P}_{m}^{h}(\rb)$ is then defined as:
\[
{\mathcal P}_{m}^{h}(\rb)\coloneqq g^{\star}(\0\ |\ \rb).
\]
According to \cite[Section 3.2]{sober2020manifold}, the cost of this step is $O(Dd^{m} + d^{3m})$. Thus, the total cost of performing MMLS projection on a given point $\rb$ is $O(Dd^{m} + d^{3m})$ \cite[Corollary 3.11]{sober2020manifold}. For convenience we denote it by $\tmmls \coloneqq O(Dd^{m} + d^{3m})$.

Using MMLS, the approximating manifold of ${\mathcal{M}}$ denoted by $\widetilde{{\mathcal{M}}}$
is then defined by 
\begin{equation}\label{eq:deftildeM}
    \widetilde{{\mathcal{M}}}=\{{\mathcal P}_{m}^{h}(\p)\ |\ \p\in{\mathcal{M}}\}.
\end{equation}

In \cite[theorems 4.21 and 4.22]{sober2020manifold}, it is shown that if $\theta_{1}$ and $\theta_{2}$ are monotonically decaying and compactly supported, for $\rb\in U_{\mathrm{unique}}$, and small enough $h$,
we have that ${\mathcal P}_{m}^{h}(\rb)\in\widetilde{{\mathcal{M}}}$ and $\widetilde{{\mathcal{M}}}$
is almost everywhere a $C^{\infty}$ $d$-dimensional manifold that approximates ${\mathcal{M}}$
with the approximation order $O(h^{m+1})$ in Hausdorff norm, i.e.
\begin{equation}\label{eq:Hausdorffdistapprox}
    \|\widetilde{{\mathcal{M}}}-{\mathcal{M}}\|_{\text{Hausdorff}}\coloneqq\max\{\max_{\s\in\widetilde{{\mathcal{M}}}}d(\s,{\mathcal{M}}),\!\max_{\x\in{\mathcal{M}}}d(\x,\widetilde{{\mathcal{M}}})\}\leq M\cdot h^{m+1},
\end{equation}
for some constant $M>0$ independent of $h$. Together with \cite[Eq. (24)]{sober2020approximating}, we have that given $\p\in\mathcal{M}$ such that ${\mathcal P}_{m}^{h}(\p) = \rb$, then $\rb\in \widetilde{\mathcal{M}}$ and 
\begin{equation}\label{eq:distbetween_p_and_r}
    \|\rb - \p\| \leq c_{\mmls}\sqrt{D}h^{m+1}, 
\end{equation}
for some constant $c_{\mmls}>0$ independent of $\p$ and $\rb$. In addition, in \cite[Theorem 4.21]{sober2020manifold} it is shown that MMLS projection from $\mathcal{M}$ to $\widetilde{\mathcal{M}}$ is smooth, thus $\widetilde{\mathcal{M}}$ is a compact set as the image of the compact manifold $\mathcal{M}$. Furthermore, \cite[Theorem 4.19]{sober2020manifold} states that the aforementioned projection is an injective mapping. Thus, together with the definition of $\widetilde{\mathcal{M}}$ from Eq. \eqref{eq:deftildeM}, we have that MMLS projection from $\mathcal{M}$ to $\widetilde{\mathcal{M}}$ is bijective. 

Moreover, under the
same conditions, in \cite[Lemma 3]{sober2020approximating} (see Lemma \ref{lem:Sober_app_lemma3} in Appendix \ref{subsec:background_MMLS}), it is shown that
a parametrization $\varphi:H(\p)\to{\mathcal{M}}\subseteq\R^{D}$ of ${\mathcal{M}}$
exists such that $\p\in{\mathcal{M}}$, $\varphi(\0)=\p$, and ${\mathcal P}_{m}^{h}(\p)\in\widetilde{{\mathcal{M}}}$, and also a parametrization $\widetilde{\varphi}:H(\p)\to\widetilde{{\mathcal{M}}}\subseteq\R^{D}$ of $\widetilde{{\mathcal{M}}}$
exists such that $\widetilde{\varphi}(\0)={\mathcal P}_{m}^{h}(\p)$. It is further shown in \cite[Lemma 4]{sober2020approximating} (see Lemma \ref{lem:Sober_app_lemma4} in Appendix \ref{subsec:background_MMLS}),
that if in addition $\lim_{t \to 0}\theta_{h}(t)=\infty$, i.e., $\widetilde{{\mathcal{M}}}$ interpolates ${\mathcal{M}}$ at
$\rb_{i}$, then in a small vicinity of $\0\in H(\p)$
in $H(\p)$ the directional derivatives of $\varphi(\x)$
and $g^{\star}(\x\ |\ \p)$ (with respect to the first input) at any direction $\v\in\R^{d}$ satisfy
\begin{equation}\label{eq:MMLS_der}\frac{1}{\sqrt{D}}\|\text{D}g^{\star}(\x\ |\ \p)[\v]-\text{D}\varphi(\x)[\v]\|\leq
    \|\text{D}g^{\star}(\x\ |\ \p)[\v]-\text{D}\varphi(\x)[\v]\|_{\infty}\leq c_{\mathcal{M}} h^{m},
\end{equation}
for some constant $c_{\mathcal{M}} > 0$ independent of $\v$ or $\p$ (see Appendix \ref{subsec:background_MMLS}). Note that Eq. \eqref{eq:MMLS_der} is also true for the directional derivatives of $g^{\star}$ (with respect to the first input) and of $\widetilde{\varphi}(\cdot)$, with another constant $c_{\widetilde{\mathcal{M}}} > 0$, and also for the directional derivatives of $\varphi$ and $\widetilde{\varphi}$ with $c_{\mathcal{M}, \widetilde{\mathcal{M}}} > 0$, both constants are independent of $\v$ or $\p$, i.e.,
\begin{equation}\label{eq:MMLS_der1}\frac{1}{\sqrt{D}}\|\text{D}g^{\star}(\x\ |\ \p)[\v]-\text{D}\widetilde{\varphi}(\x)[\v]\|\leq
    \|\text{D}g^{\star}(\x\ |\ \p)[\v]-\text{D}\widetilde{\varphi}(\x)[\v]\|_{\infty}\leq c_{\widetilde{\mathcal{M}}} h^{m},
\end{equation}
and 
\begin{equation}\label{eq:MMLS_der2}\frac{1}{\sqrt{D}}\|\text{D}\varphi(\x)[\v]-\text{D}\widetilde{\varphi}(\x)[\v]\|\leq
    \|\text{D}\varphi(\x)[\v]-\text{D}\widetilde{\varphi}(\x)[\v]\|_{\infty}\leq c_{\mathcal{M}, \widetilde{\mathcal{M}}} h^{m}.
\end{equation}

Note that \cite[Lemma 3]{sober2020approximating} and \cite[Lemma 4]{sober2020approximating} can be extended in the following way. Given $\rb\in\widetilde{\mathcal{M}}$, there exists a $\p\in \mathcal{M}$ such that ${\mathcal P}_{m}^{h}(\p)=\rb$ since MMLS projection from $\mathcal{M}$ to $\widetilde{\mathcal{M}}$ is bijective. Thus, from the uniqueness property \cite[Lemma 4.7]{sober2020manifold}, we have that $\q(\rb) = \q (\p)$, $H(\rb) = H(\p)$, and ${\mathcal P}_{m}^{h}(\p)={\mathcal P}_{m}^{h}(\rb)=\rb$. Thus, for such $\rb\in\widetilde{\mathcal{M}}$ the same parametrizations of $\widetilde{\mathcal{M}}$ and $\mathcal{M}$ from $H(\p)$ which are guaranteed by \cite[Lemma 3]{sober2020approximating}, i.e., $\varphi:H(\p)\to{\mathcal{M}}$ and $\widetilde{\varphi}:H(\p)\to\widetilde{{\mathcal{M}}}$ correspondingly, are also parametrizations from $H(\rb)$. Moreover, Eq. \eqref{eq:MMLS_der}, Eq. \eqref{eq:MMLS_der1}, and Eq. \eqref{eq:MMLS_der2}, which hold due to \cite[Lemma 4]{sober2020approximating}, hold also for such $\rb\in\widetilde{\mathcal{M}}$ and its corresponding MMLS polynomial $g^{\star}(\x\ |\ \rb)$ (since $g^{\star}(\x\ |\ \rb) = g^{\star}(\x\ |\ \p)$). 

Finally, it can be shown from Eq. \eqref{eq:MMLS_der}, Eq. \eqref{eq:MMLS_der1}, and Eq. \eqref{eq:MMLS_der2}, and the structure of MMLS projection that the orthogonal projections operators on the ranges of $\text{D}g^{\star}(\x\ |\ \p)$, $\text{D}\varphi(\x)$, and $\text{D}\widetilde{\varphi}(\x)$, differ in $O(\sqrt{D} h^{m})$ from each other in $L_{2}$ norm (see Lemma \ref{lem:our_proj_prop} in Appendix \ref{subsec:background_MMLS}).
\begin{rem}[Changing the inner product]
\label{rem:other_metrics}MMLS projection is presented here using
the standard inner product in $\R^{D}$, but it is possible to generalize
this procedure for any other inner product of the form $\dotprodM{\u}{\v}=\sqrt{\u^{\T}\matM\v}$
for $\matM$ symmetric positive-definite (SPD) matrix \cite[Remark 4.23]{sober2020manifold}.
\end{rem}

\subsubsection{\label{subsec:MMLS_func}Extension of MMLS for Function Approximations}

In this subsection we recall the extension of MMLS to function approximation \cite{sober2021approximation}. The assumption of the algorithm for a noiseless data are:
\begin{assumption}[Function approximation assumptions]
\label{assu:MMLS_func_Samples}${\mathcal{M}}\in C^{2}$ is a closed (i.e.,
compact and boundaryless) $d$-dimensional submanifold of $\R^{D}$. $f$ is a function from $\R^{D}$ to $\R$, and we look at its restriction on $\mathcal{M}$.
The sample set $S=\left\{\rb_{i}\right\}_{i=1}^{n}\subset{\mathcal{M}}$ is a quasi-uniform
sample set with respect to the domain ${\mathcal{M}}$, with fill distance
$h$. At each point of $S$ we also have a sample of $f$, i.e., $f_{i}=f(\rb_{i})$ for $1\leq i\leq n$. Thus, the sample-set at hand is $S_{f}=\left\{\left(\rb_{i},f(\rb_{i})\right)\right\}_{i=1}^{n}\subset {\mathcal{M}} \times \mathbb{R}$
\end{assumption}

The goal is given a
point $\rb$ close to ${\mathcal{M}}$, i.e., $\rb=\widehat{\rb}+\epsilon$ such
that $\widehat{\rb}\in{\mathcal{M}}$ and $\epsilon\in\R^{D}$, approximate $f(\widehat{\rb})$.
Using the moving coordinate system obtained in first step of MMLS
projection, the goal in the second step is modified so that in essence the function 
\begin{equation}\label{eq:MMLS_func_to approx}
   \widehat{f}\coloneqq f\circ\varphi 
\end{equation}
is approximated (where $\varphi:H(\rb)\to \mathcal{M}$ is some local parametrization of $\mathcal{M}$ as in the previous subsection), by modifying the second step of MMLS to be 
\begin{equation}\label{eq:step2-func_approx}
p_{\rb}^{f}(\cdot)=\arg\min_{p\in\Pi_{m}^{d}}\sum_{i=1}^{n}\|p(\x_{i})-\widehat{f}(\x_{i})\|^{2}\theta_{3}(\|\rb_{i}-q(\rb)\|),
\end{equation}
where $\theta_{3}(\cdot)$ is a
non-negative weight function (locally supported or rapidly decreasing as $t \to \infty$) consistent across scales, i.e., $\theta_{3}(th)=\Psi(t)$. We take $p_{\rb}^{f}(\0)$ to be the approximation of $f(\widehat{\rb})$,
i.e.,
\[
f(\widehat{\rb})\approx\widetilde{f}(\rb)\coloneqq p_{\rb}^{f}(\0),
\]
where $\widetilde{f}(\rb)$ denotes MMLS approximation of $f$ on $\mathcal{M}$. Similarly to the reasoning in \cite[Section 3.2]{sober2020manifold}, solving problem \eqref{eq:step2-func_approx} costs $O(Dd^{m} + d^{3m})$, and together with the first step of MMLS algorithm, the cost stays $O(Dd^{m} + d^{3m})$, i.e., $O(\tmmls)$.

In \cite[Theorem 3.1]{sober2021approximation}, it
is shown that if $\theta_{1}(\cdot), \theta_{2}(\cdot)\in C^{\infty}$ then the resulting approximation $\widetilde{f}(\rb)$ is a $C^{\infty}$
function for $\rb\in U_{\mathrm{unique}}$, and for any $\rb_{0},\rb_{1}\in U_{\mathrm{unique}}$
such that $\rb_{1}-q(\rb_{0})\perp H(\rb_{0})$ we have $\widetilde{f}(\rb_{0})=\widetilde{f}(\rb_{1})$
(uniqueness). Moreover, for $h$ small enough and $\rb_{2}\in{\mathcal{M}}$ we have that $|f(\rb_{2})-\widetilde{f}(\rb_{2})|=O(h^{m+1})$ \cite[Theorem 3.2]{sober2021approximation}. Thus, for $\rb_{0}\in U_{\mathrm{unique}}$
such that $\rb_{2}-q(\rb_{0})\perp H(\rb_{0})$ and ${\mathcal P}_{m}^{h}(\rb_{0})={\mathcal P}_{m}^{h}(\rb_{2}) \in \widetilde{{\mathcal{M}}}$ we have that $|f(\rb_{2})-\widetilde{f}(\rb_{0})|=O(h^{m+1})$. Moreover, using \cite[Lemma 1]{sober2020approximating} it is possible
to show that the order of approximation of the first derivative (gradient) in infinity norm
of $\widehat{f}(\x)$ by $p_{\rb}^{f}(\x)$ around $0\in H(\rb)$ is $O(h^{m})$ (see Lemma \ref{lem:fderiv_approx} in Appendix \ref{subsec:background_MMLS}).

\section{\label{sec:The-Proposed-Algorithms}MMLS Riemannian Optimization
(MMLS-RO)}

In this section, we present our proposed methods for for approximating the solution of Problem \eqref{eq:general_problem}. Our method is based on performing Riemannian
optimization where the various geometric components are implemented using MMLS projection, which was presented in Subsection \ref{subsec:Manifold-Moving-Least-Squares}.

The problem setting is as follows. The constraint set is a $d$-dimensional
submanifold of $\R^{D}$, denoted by ${\mathcal{M}}$. Its dimension is known in advance or estimated, so we have $d$ at our hand. We assume that some of ${\mathcal{M}}$'s geometric components which are required for Riemannian optimization (e.g., tangent spaces, retraction) are intractable explicitly, but it is possible to sample ${\mathcal{M}}$ according to Assumption
\ref{assu:MMLS_Samples}. Note that if the tangent spaces at the sampled point are given, then at least at these points this information can be utilized to achieve better approximations of the other components (e.g., retraction). We assume that the cost function $f:\R^{D}\to\R$ is at least a Lipschitz-gradient function, with a bounded Lipschitz constant, as to (informally) ensure that our approximate solution has a close cost value to the cost value on the closest points on $\mathcal{M}$. We assume that either $f$ and its Euclidean gradient are known, either only $f$ is known, or that at most, only samples of $f$ are available according to Assumption
\ref{assu:MMLS_func_Samples}. 

We propose to perform the optimization on the approximating
manifold of ${\mathcal{M}}$ obtained from MMLS procedure, i.e., $\widetilde{{\mathcal{M}}}$. In order to do so, we use MMLS algorithm (Section \ref{subsec:MMLS_the_algorithm}) to approximate the tangent spaces of $\mathcal{M}$ (and simultaneously the tangent spaces of $\widetilde{{\mathcal{M}}}$), and define an approximate retraction on $\widetilde{{\mathcal{M}}}$. In addition, orthogonal projections on the approximated tangent spaces allow us to
approximate the Riemannian gradient of $f$ and a vector transport. Note that for the case where $f$ (if required) is also approximated (Section \ref{subsec:MMLS_func}), we first approximate the Euclidean gradient of $\widehat{f}$ before turning it into a Riemannian gradient (Subsection \ref{subsec:Approximate-ROMMLS}).

Explicitly, Problem \eqref{eq:general_problem} is reformulated
in the following way: given an initial point $\p_0 \in U_{\mathrm{unique}}$
our goal is to build iterative methods to solve
\begin{equation}
\min_{\x\in\widetilde{{\mathcal{M}}}}f(\x). \label{eq:general_problem_for_ROMMLS}
\end{equation}
Our proposed Algorithms \ref{alg:MMLS-for-Riemannian} and \ref{alg:MMLS-RO_CG}, are first-order algorithms, thus potentially achieve points on $\widetilde{\mathcal{M}}$ which satisfy first-order criteria (bounded Riemannian gradient norm). We relate the achieved points on $\widetilde{\mathcal{M}}$ to their corresponding points on $\mathcal{M}$, and we bound their Riemannian gradient norm (Subsection \ref{subsec:Convergence-Analysis-of2}).

\subsection{\label{subsec:Approximate-ROMMLS}The Riemannian Components and Example Algorithms}

In this subsection we elaborate on the construction of the various approximate Riemannian components. First, we propose to use an alternative for the unknown or intractable tangent space (at the points where it is unknown), based on the polynomial constructed at the second step of MMLS projection. Suppose we are at a point
$\rb\in\widetilde{{\mathcal{M}}}$ (otherwise, project it on $\widetilde{{\mathcal{M}}}$ using MMLS). Our proposed approximation of the tangent space is the range of the differential of the polynomial approximation $g^{\star}(\cdot\ |\ \rb):H(\rb)\to\R^{D}$ at $\x = \0$, i.e., $\Range{\text{D}g^{\star}(\0\ |\ \rb)}$. Recall that $g^{\star}(\cdot\ |\ \rb)$ can be viewed as an approximation of a local parametrization $\varphi:H(\rb)\to\mathcal{M}$ of $\mathcal{M}$. Thus, we can view $\Range{\text{D}g^{\star}(\0\ |\ \rb)}$ as an approximation of the tangent space of $\mathcal{M}$ (at the corresponding point on $\mathcal{M}$ to $\rb$ via MMLS approximation, i.e., $\p\in\mathcal{M}$ such that ${\mathcal P}_{m}^{h}(\p) = \rb$, and also of the tangent space of $\widetilde{{\mathcal{M}}}$ at the point $\rb$. We denote the \emph{approximate-tangent space} by
\begin{equation}\label{eq:aprox_tan_def}
    \widetilde{T}_{\rb}\widetilde{\mathcal{M}} \coloneqq \Range{\text{D}g^{\star}(\0\ |\ \rb)}.
\end{equation} 
Note that the output of the second step of MMLS, i.e., $g^{\star}(\x\ |\ \rb):H(\rb)\to\R^{D}$, can be equivalently viewed as $g^{\star}(\x\ |\ \rb):H(\rb)\to H^{\perp}(\rb)$, i.e., an approximation of $\mathcal{M}$ as a graph of a function (see Subsection \ref{subsec:MMLS_the_algorithm}). Thus, $\text{D}g^{\star}(\0\ |\ \rb)$ is full-rank. Moreover, even if $\rb\in\widetilde{{\mathcal{M}}}$ is a point where $\widetilde{\mathcal{M}}$ is not a manifold, $\widetilde{T}_{\rb}\widetilde{\mathcal{M}}$ is still defined. 

Here we have an abuse of notation, we denote by $\text{D}g^{\star}(\0\ |\ \rb)$ the matrix which represent the linear transformation $\text{D}g^{\star}(\0\ |\ \rb)$, such that its columns (assumed independent) are a basis for $\Range{\text{D}g^{\star}(\0\ |\ \rb)}$. The cost of finding the basis for $\Range{\text{D}g^{\star}(\0\ |\ \rb)}$, is equal to the cost of forming the matrix $\text{D}g^{\star}(\0\ |\ \rb)$. This amounts to first performing MMLS projection of $\rb\in\widetilde{{\mathcal{M}}}$, i.e., $\tmmls$, and then find the coefficients of each of the $d$ first-order monomials at each of the $D$ coordinates of MMLS projection which takes $O(Dd)$ operations. Thus, the total cost is $O(\tmmls + Dd)$.  
Next, we define the approximate-tangent bundle in a manner similar to Eq.~\eqref{eq:tan_bun}:
\begin{equation}\label{eq:approx_tan_bun}
     \widetilde{T}\widetilde{\mathcal{M}} \coloneqq \left\{(\rb, \xi)\ :\ \rb\in \widetilde{\mathcal{M}}\ \cap \ \xi \in \widetilde{T}_{\rb}\widetilde{\mathcal{M}}\right\}.
\end{equation}

Recall from \cite[Lemma 4]{sober2020approximating}, that if $h$ is sufficiently small and $\lim_{t\to0}\theta_{h}(t)=\infty$, then $\Range{\text{D}g^{\star}(\0\ |\ \rb)}$, $\Range{\text{D}\varphi(\0)}$, and $\Range{\text{D}\widetilde{\varphi}(\0)}$ are $O(h^{m})$ approximations at infinity norm of each other. Thus, the spaces $\widetilde{T}_{\rb}\widetilde{\mathcal{M}}$, $T_{\varphi(\0) = \p}{\mathcal{M}}$, and $T_{\widetilde{\varphi}(\0)=\rb}\widetilde{\mathcal{M}}$ are also $O(h^{m})$ approximations at infinity norm of each other, i.e., the orthogonal projections operators on each of these spaces differ in $O(\sqrt{D} h^{m})$ (in $L_{2}$ norm, see Lemma \ref{lem:our_proj_prop} in Appendix \ref{subsec:background_MMLS}). 

The Riemannian metric we define on $\widetilde{T}_{\rb}\widetilde{\mathcal{M}}$ is the standard inner product
on the ambient space $\R^{D}$ restricted on $\Range{\text{D}g^{\star}(\0\ |\ \rb)}$
(for different choices of Riemannian metric see Remark \ref{rem:other_metrics}).
The orthogonal projection on $\widetilde{T}_{\rb}\widetilde{\mathcal{M}}$ with respect to the Riemannian metric can be defined via the Moore-Penrose inverse 
\cite[Chapter 5.5.2]{golub2013matrix} of $\text{D}g^{\star}(\0\ |\ \rb)$, i.e.,
\begin{equation*}
    \text{D}g^{\star}(\0\ |\ \rb)^{\pinv} = \matG_{\text{D}g^{\star}(\0\ |\ \rb)}^{-1} \text{D}g^{\star}(\0\ |\ \rb)^{\T},
\end{equation*}
where $\matG_{\text{D}g^{\star}(\0\ |\ \rb)}$ denotes the Gram matrix of $\text{D}g^{\star}(\0\ |\ \rb)\in \R^{D \times d}$, in the following way 
\begin{equation*}
    \forall\xi\in\R^{D},\ \widetilde{\Pi}_{\rb}(\xi)\coloneqq\text{D}g^{\star}(\0\ |\ \rb)\text{D}g^{\star}(\0\ |\ \rb)^{\pinv}\xi.
\end{equation*}
The cost of applying the orthogonal projection consists of computing $\text{D}g^{\star}(\0\ |\ \rb)$, which takes $O(\tmmls + Dd)$, computing $\text{D}g^{\star}(\0\ |\ \rb)^{\pinv}$ - inverting the Gram matrix and applying it on $\text{D}g^{\star}(\0\ |\ \rb)^{\T}$ applied on some $\xi\in \R^{D}$, and applying $\text{D}g^{\star}(\0\ |\ \rb)$ on the result, which together take $O(Dd^{2})$. Thus, the total cost is $O(\tmmls + Dd^{2})$.

The orthogonal projection on $\widetilde{T}_{\rb}\widetilde{\mathcal{M}}$ allows us to define both an approximate-Riemannian gradient, denoted by ${\gradtil{f(\rb)}}$,
and an approximate-vector transport denoted by $\widetilde{\tau}_{\eta}\xi$.
We begin with an approximate-Riemannian gradient. There are two different cases we deal with, depending on the available information on $f$, thus, we define a different approximation for each case, but to keep the notation simple, we use the same notation for both cases. 

In the first case, $f$ and its Euclidean gradient are known. Thus, we define an approximate-Riemannian gradient  in the following way:
\begin{equation}\label{eq:f_Rgrad}
    {\gradtil{f(\rb)}}\coloneqq\widetilde{\Pi}_{\rb}(\nabla{f}(\rb)),
\end{equation}
i.e., an orthogonal projection of the Euclidean gradient on the corresponding approximate-tangent space. Note, that this definition is equivalent to defining the gradient via some local parametrization. The cost of computing the approximate-Riemannian gradient in this case amounts to $O(\tmmls +Dd^{2} + T_{\nabla f})$, where $T_{\nabla f}$ denotes the maximal cost of computing $\nabla f$.

In the second case, the Euclidean gradient of $f$ (and possibly $f$ itself) is approximated. Recall that $\nabla p_{\rb}^{f}(\0)$ is an $O(h^{m})$ approximation of $\nabla \widehat{f}(\0)$ in infinity norm (Lemma \ref{lem:fderiv_approx} in Appendix \ref{subsec:background_MMLS}). From Eq. \eqref{eq:MMLS_func_to approx}, we have that
\begin{equation}\label{eq:approx_grad_form}
    \nabla \widehat{f}(\0) = \text{D}\varphi(\0)^{\T} \nabla f(\p),
\end{equation}
where  $\rb\in \widetilde{{\mathcal{M}}}$ and $\varphi (\0) = \p\in{\mathcal{M}}$ such that ${\mathcal P}_{m}^{h}(\p) = \rb$. Recall that the exact Riemannian gradient on $\mathcal{M}$, which is denoted by $\gradM$, is
\begin{equation}\label{eq:exact_Rgrad}
    \gradM{f(\p)} = \text{D}\varphi(\0)\text{D}\varphi(\0)^{\pinv} \nabla f(\p) = \text{D}\varphi(\0)\matG_{\text{D}\varphi(\0)}^{-1} \text{D}\varphi(\0)^{\T}\nabla f(\p) = \text{D}\varphi(\0)\matG_{\text{D}\varphi(\0)}^{-1}\nabla \widehat{f}(\0).
\end{equation}
Thus, to approximate Eq. \eqref{eq:exact_Rgrad} using $\nabla p_{\rb}^{f}(\0)$, we define
\begin{equation}\label{eq:approximate_f_Rgrad}
    \gradM{f(\p)} \approx {\gradtil{f(\rb)}} \coloneqq \text{D}g^{\star}(\0\ |\ \rb)\matG_{\text{D}g^{\star}(\0\ |\ \rb)}^{-1}\nabla p_{\rb}^{f}(\0),
\end{equation}
as the approximate-Riemannian gradient. The cost of computing the approximate-Riemannian gradient in this case consists of computing $\text{D}g^{\star}(\0\ |\ \rb)$, computing the inverse of its Gram matrix, compute $\nabla p_{\rb}^{f}(\x)$ at $\x = \0$ (which requires performing MMLS algorithm for function approximation, $O(\tmmls)$, and finding the coefficients of each of the $d$ first-order monomials, $O(d)$), and performing the matrix-vector multiplications. Thus, the total cost is $O(\tmmls + Dd^{2} + d)$. 

Denote by $\gradMtil$ the Riemannian gradient on ${\widetilde{\mathcal{M}}}$. The following lemma establishes a relation between $\gradtil{f(\rb)}$, $\gradMtil{f(\rb)}$, and $\gradM{f(\p)}$ for a given $\rb\in\widetilde{\mathcal{M}}$ in all cases. However, we first need the following standard (in the analysis of optimization methods) assumption.

\begin{assumption}[$f$ is gradient $L$-Lipschitz in the convex hull of $\widetilde{\mathcal{M}} \cup {\mathcal{M}}$]
\label{assu:Lipschitz-gradient} Denote by $\text{Conv}(\widetilde{\mathcal{M}} \cup {\mathcal{M}})$ the convex hull of $\widetilde{\mathcal{M}} \cup {\mathcal{M}}$. There exists $L\geq0$ such that for all
$\x,\y\in\text{Conv}(\widetilde{\mathcal{M}} \cup {\mathcal{M}})$, we have that $f$ is (Euclidean) gradient $L$-Lipschitz, i.e.,  
\[
\|\nabla f(\y) - \nabla f(\x)\|\leq L \|\y-\x\|,
\]
or equivalently
\[
|f(\y)-[f(\x)+\dotprod{\nabla f(\x)}{\y-\x}]|\leq\frac{L}{2}\|\y-\x\|^{2}.
\]
\end{assumption}

\begin{lemma}[Riemannian gradient approximation order]\label{lem:approx_Rim_grad}
If $f$ and its Euclidean gradient are known, then
\begin{equation}\label{eq:approx_Rim_grad1}
    \left\|{\gradtil{f(\rb)}}-{\gradMtil{f(\rb)}}\right\| \leq c_{\widetilde{\mathcal{M}}} \sqrt{D}\left\| \nabla f(\rb) \right\|h^{m}.
\end{equation}
If in addition, $f$ is gradient $L$-Lipschitz (Assumption \ref{assu:Lipschitz-gradient}), $h$ is small enough such that $h\leq L^{-1}$, and let $\p\in{\mathcal{M}}$ satisfy ${\mathcal P}_{m}^{h}(\p)=\rb\in\widetilde{\mathcal{M}}$, then
\begin{equation}\label{eq:approx_Rim_grad2}
    \left\|{\gradtil{f(\rb)}}-{\gradM{f(\p)}}\right\| \leq \left(c_{\mathcal{M}} \left\| \nabla f(\rb) \right\| + c_{\mmls}\right)\sqrt{D}h^{m}.
\end{equation}

If the Euclidean gradient of $\widehat{f}$ (and possibly $f$ itself) is approximated via MMLS procedure extension (Subsection \ref{subsec:MMLS_func}), let $\p\in{\mathcal{M}}$ satisfy ${\mathcal P}_{m}^{h}(\p)=\rb\in\widetilde{\mathcal{M}}$, and define the approximate-Riemannian gradient according to Eq. \eqref{eq:approximate_f_Rgrad}, then
\begin{equation}\label{eq:approx_Rim_grad3}
    \left\|{\gradtil{f(\rb)}}-{\gradM{f(\p)}}\right\| \leq \left( c_{f} + 2 c_{\mathcal{M}}  \left\| \nabla f(\p) \right\| \right) \sqrt{D} h^{m}.
\end{equation}
If in addition, $f$ is gradient $L$-Lipschitz (Assumption \ref{assu:Lipschitz-gradient}), $h$ is small enough such that $h\leq L^{-1}$, then
\begin{equation}\label{eq:approx_Rim_grad4}
    \left\|{\gradtil{f(\rb)}}-{\gradMtil{f(\rb)}}\right\| \leq \left( 2 c_{\mathcal{M}} \left\| \nabla f(\p) \right\| + c_{\widetilde{\mathcal{M}}} \left\| \nabla f(\p) \right\| + c_{f} + c_{\mmls} \right)\sqrt{D}h^{m}.
\end{equation}
\end{lemma}
\begin{proof}
We begin with the case where $f$ and its Euclidean gradient are known. To show Eq.~\eqref{eq:approx_Rim_grad1}, we use the fact that Lemma \ref{lem:our_proj_prop} (Appendix \ref{subsec:background_MMLS}) shows that the orthogonal projection of $\nabla f(\rb)$ on $\widetilde{T}_{\rb}\widetilde{\mathcal{M}}$, i.e., ${\gradtil{f(\rb)}}$, and the orthogonal projection of $\nabla f(\rb)$ on $T_{\rb}{\widetilde{\mathcal{M}}}$, i.e., ${\gradMtil{f(\rb)}}$, satisfy 
\begin{eqnarray*}
    \left\|{\gradtil{f(\rb)}}-{\gradMtil{f(\rb)}}\right\|&=&\left\|\Pi_{\widetilde{T}_{\rb}\widetilde{\mathcal{M}}}(\nabla f(\rb)) - \Pi_{T_{\rb}\widetilde{{\mathcal{M}}}}(\nabla f(\rb))\right\|\\&\leq& c_{\widetilde{\mathcal{M}}} \sqrt{D} \left\| \nabla f(\rb) \right\|h^{m}.
\end{eqnarray*} 
Next, to show Eq. \eqref{eq:approx_Rim_grad2}, let $\p\in{\mathcal{M}}$ such that ${\mathcal P}_{m}^{h}(\p)=\rb$, then for sufficiently small $h$, Eq. \eqref{eq:distbetween_p_and_r} holds. Thus, from Assumption \ref{assu:Lipschitz-gradient} and $h$ small enough such that $h\leq L^{-1}$, we get
\begin{equation}\label{eq:Lip_outer_grad}
    \left\| \nabla f(\rb)-\nabla f(\p) \right\| \leq L\|\rb - \p\| \leq L c_{\mmls}\sqrt{D}h^{m+1}\leq c_{\mmls}\sqrt{D}h^{m}.
\end{equation} 
From Lemma \ref{lem:our_proj_prop} (Appendix \ref{subsec:background_MMLS}), for the orthogonal projection of $\nabla f(\rb)$ on $\widetilde{T}_{\rb}\widetilde{\mathcal{M}}$, i.e., ${\gradtil{f(\rb)}}$, and the orthogonal projection of $\nabla f(\rb)$ on $T_{\p}{\mathcal{M}}$, i.e., $\Pi_{T_{\p}{\mathcal{M}}}(\nabla f(\rb))$, we have
\begin{eqnarray}\label{eq:infinity_grad}
    \|{\gradtil{f(\rb)}}-\Pi_{T_{\p}{\mathcal{M}}}(\nabla f(\rb))\|&=& \|\Pi_{\widetilde{T}_{\rb}\widetilde{\mathcal{M}}}(\nabla f(\rb)) - \Pi_{T_{\p}{\mathcal{M}}}(\nabla f(\rb))\|\nonumber\\&\leq&  c_{\mathcal{M}} \sqrt{D} \left\| \nabla f(\rb) \right\|h^{m}.
\end{eqnarray} Finally using Eq. \eqref{eq:Lip_outer_grad} and Eq. \eqref{eq:infinity_grad}, we get
\begin{eqnarray*}
    \|{\gradtil{f(\rb)}}-{\gradM{f(\p)}}\|&\leq&\\ \|{\gradtil{f(\rb)}}-\Pi_{T_{\p}{\mathcal{M}}}(\nabla f(\rb))\|+ \|\Pi_{T_{\p}{\mathcal{M}}}(\nabla f(\rb))-{\gradM{f(\p)}}\| &\leq&\\ c_{\mathcal{M}} \sqrt{D}\left\| \nabla f(\rb) \right\|h^{m} + \|\Pi_{T_{\p}{\mathcal{M}}}(\nabla f(\rb))-\Pi_{T_{\p}{\mathcal{M}}}(\nabla f(\p))\| &\leq&\\ c_{\mathcal{M}} \sqrt{D}\left\| \nabla f(\rb) \right\|h^{m} + \|\nabla f(\rb)-\nabla f(\p)\| &\leq& \left(c_{\mathcal{M}} \left\| \nabla f(\rb) \right\| + c_{\mmls}\right)\sqrt{D}h^{m}.
\end{eqnarray*}

Now, we address the case where the Euclidean gradient of $\widehat{f}$ (and possibly $f$ itself) is approximated via MMLS procedure extension. To show Eq. \eqref{eq:approx_Rim_grad3}, recall that $\nabla p_{\rb}^{f}(\0)$ is an $O(h^{m})$ approximation of $\nabla \widehat{f}(\0)$ in infinity norm (Lemma \ref{lem:fderiv_approx} in Appendix \ref{subsec:background_MMLS}), i.e.,
\begin{eqnarray*}\label{eq:infinity_grad_approx}
    \frac{1}{\sqrt{D}}\|\nabla p_{\rb}^{f}(\0)-\nabla \widehat{f}(\0)\|&\leq&\|\nabla p_{\rb}^{f}(\0)-\nabla \widehat{f}(\0)\|_{\infty}\\&\leq& c_{f} h^{m}.
\end{eqnarray*}
We write explicitly the left-hand side of Eq. \eqref{eq:approx_Rim_grad3} using Eq. \eqref{eq:exact_Rgrad}
\begin{eqnarray}\label{eq:approx_Rim_grad3_1}
\left\|{\gradtil{f(\rb)}}-{\gradM{f(\p)}}\right\| &=&  \left\|\text{D}g^{\star}(\0\ |\ \rb)\matG_{\text{D}g^{\star}(\0\ |\ \rb)}^{-1}\nabla p_{\rb}^{f}(\0)-\text{D}\varphi(\0)\matG_{\text{D}\varphi(\0)}^{-1}\nabla \widehat{f}(\0)\right\|\\
&\leq& \left\|\text{D}g^{\star}(\0\ |\ \rb)\matG_{\text{D}g^{\star}(\0\ |\ \rb)}^{-1}\nabla p_{\rb}^{f}(\0)-\text{D}g^{\star}(\0\ |\ \rb)\matG_{\text{D}g^{\star}(\0\ |\ \rb)}^{-1}\nabla \widehat{f}(\0)\right\|\nonumber\\
&+& \left\|\text{D}g^{\star}(\0\ |\ \rb)\matG_{\text{D}g^{\star}(\0\ |\ \rb)}^{-1}\nabla \widehat{f}(\0)-\text{D}\varphi(\0)\matG_{\text{D}\varphi(\0)}^{-1}\nabla \widehat{f}(\0)\right\|.\nonumber
\end{eqnarray}
To bound Eq. \eqref{eq:approx_Rim_grad3_1}, we bound each of the two terms separately. For the first term we have,
\begin{eqnarray}
\left\|\text{D}g^{\star}(\0\ |\ \rb)\matG_{\text{D}g^{\star}(\0\ |\ \rb)}^{-1}\nabla p_{\rb}^{f}(\0)-\text{D}g^{\star}(\0\ |\ \rb)\matG_{\text{D}g^{\star}(\0\ |\ \rb)}^{-1}\nabla \widehat{f}(\0)\right\|\nonumber\\
\leq \left\| \text{D}g^{\star}(\0\ |\ \rb)\matG_{\text{D}g^{\star}(\0\ |\ \rb)}^{-\nicehalf} \matG_{\text{D}g^{\star}(\0\ |\ \rb)}^{-\nicehalf} \left(\nabla p_{\rb}^{f}(\0)-\nabla \widehat{f}(\0) \right) \right\|\nonumber\\
\leq \left\| \text{D}g^{\star}(\0\ |\ \rb)\matG_{\text{D}g^{\star}(\0\ |\ \rb)}^{-\nicehalf} \right\| \cdot \left\| \matG_{\text{D}g^{\star}(\0\ |\ \rb)}^{-\nicehalf} \right\| \cdot \left\| \nabla p_{\rb}^{f}(\0)-\nabla \widehat{f}(\0) \right\|\nonumber\\
\leq \left\| \text{D}g^{\star}(\0\ |\ \rb)\matG_{\text{D}g^{\star}(\0\ |\ \rb)}^{-\nicehalf} \right\| \cdot \left\| \matG_{\text{D}g^{\star}(\0\ |\ \rb)}^{-\nicehalf} \right\| c_{f} \sqrt{D} h^{m}. \label{eq:approx_Rim_grad3_2}
\end{eqnarray}
For the second term
\begin{eqnarray}
\left\|\text{D}g^{\star}(\0\ |\ \rb)\matG_{\text{D}g^{\star}(\0\ |\ \rb)}^{-1}\nabla \widehat{f}(\0)-\text{D}\varphi(\0)\matG_{\text{D}\varphi(\0)}^{-1}\nabla \widehat{f}(\0)\right\|\nonumber\\
= \left\| \text{D}g^{\star}(\0\ |\ \rb)\matG_{\text{D}g^{\star}(\0\ |\ \rb)}^{-1}\text{D}\varphi(\0)^{\T} \nabla f(\p) - \text{D}\varphi(\0)\matG_{\text{D}\varphi(\0)}^{-1}\text{D}\varphi(\0)^{\T} \nabla f(\p) \right\| \nonumber\\
\leq \left\| \text{D}g^{\star}(\0\ |\ \rb)\matG_{\text{D}g^{\star}(\0\ |\ \rb)}^{-1}\text{D}\varphi(\0)^{\T} \nabla f(\p) - \text{D}g^{\star}(\0\ |\ \rb)\matG_{\text{D}g^{\star}(\0\ |\ \rb)}^{-1}\text{D}g^{\star}(\0\ |\ \rb)^{\T} \nabla f(\p) \right\|\nonumber\\
+ \left\| \text{D}g^{\star}(\0\ |\ \rb)\matG_{\text{D}g^{\star}(\0\ |\ \rb)}^{-1}\text{D}g^{\star}(\0\ |\ \rb)^{\T} \nabla f(\p) - \text{D}\varphi(\0)\matG_{\text{D}\varphi(\0)}^{-1}\text{D}\varphi(\0)^{\T} \nabla f(\p) \right\| \nonumber\\
\leq \left\| \text{D}g^{\star}(\0\ |\ \rb)\matG_{\text{D}g^{\star}(\0\ |\ \rb)}^{-\nicehalf} \right\| \cdot \left\| \matG_{\text{D}g^{\star}(\0\ |\ \rb)}^{-\nicehalf} \right\| \cdot \left\| \text{D}\varphi(\0)^{\T} - \text{D}g^{\star}(\0\ |\ \rb)^{\T} \right\| \cdot \left\| \nabla f(\p) \right\|\nonumber\\
+ \left\| \Pi_{\Range{\text{D}g^{\star}(\0\ |\ \rb)}}(\nabla f(\p)) - \Pi_{\Range{\text{D}\varphi(\0)}}(\nabla f(\p)) \right\| \nonumber\\
\leq \left\| \text{D}g^{\star}(\0\ |\ \rb)\matG_{\text{D}g^{\star}(\0\ |\ \rb)}^{-\nicehalf} \right\| \cdot \left\| \matG_{\text{D}g^{\star}(\0\ |\ \rb)}^{-\nicehalf} \right\| c_{\mathcal{M}} \sqrt{D} h^{m} \left\| \nabla f(\p) \right\| \label{eq:approx_Rim_grad3_3}\\
+ c_{\mathcal{M}} \sqrt{D} \left\| \nabla f(\p) \right\| h^{m},\nonumber
\end{eqnarray}
where the last inequality above arise from Lemma \ref{lem:our_proj_prop} (Appendix \ref{subsec:background_MMLS}), the following equality
$$ \left\| \text{D}\varphi(\0)^{\T} - \text{D}g^{\star}(\0\ |\ \rb)^{\T}\right\|  = \left\| \text{D}\varphi(\0) - \text{D}g^{\star}(\0\ |\ \rb) \right\|, $$
and from \cite[Lemma 4]{sober2020approximating} followed by the definition of the spectral matrix norm, i.e., taking the maximum over $\v\in \R^{d}$ such that $\left\| \v  \right\| = 1$. 

Finally, to bound Eq. \eqref{eq:approx_Rim_grad3_2} and Eq. \eqref{eq:approx_Rim_grad3_3}, we use a similar reasoning as in Lemma \ref{lem:our_proj_prop} (Appendix \ref{subsec:background_MMLS}). Recall that the output of the second step of MMLS, i.e., $g^{\star}(\x\ |\ \rb):H(\rb)\to\R^{D}$, can be equivalently viewed as $g^{\star}(\x\ |\ \rb):H(\rb)\to H^{\perp}(\rb)$, i.e., an approximation of $\mathcal{M}$ as a graph of a function (see Subsection \ref{subsec:MMLS_the_algorithm}). Now, take a basis of $\R^{D}$ to be a union of some orthogonal bases of $H(\rb)$ and $H^{\perp}(\rb)$, then the differential of $g^{\star}(\cdot\ |\ \rb)$ is of the form of Eq. \eqref{eq:our_proj_prop_mat} from Appendix \ref{subsec:background_MMLS}. In particular, $\text{D}g^{\star}(\0\ |\ \rb)$ is of the form of Eq. \eqref{eq:our_proj_prop_mat}, making the eigenvalues of $\matG_{\text{D}g^{\star}(\0\ |\ \rb)}$ be larger than $1$, leading to 
\begin{equation}\label{eq:approx_Rim_grad3_4}
    \left\| \matG_{\text{D}g^{\star}(\0\ |\ \rb)}^{-\nicehalf} \right\| \leq 1.
\end{equation}
In addition, the matrix  
$$ \text{D}g^{\star}(\0\ |\ \rb)\matG_{\text{D}g^{\star}(\0\ |\ \rb)}^{-1}\text{D}g^{\star}(\0\ |\ \rb)^{\T}, $$
is an orthogonal projection matrix. Its eigenvalues are either $0$
or $1$, bounding the following spectral norm
\begin{equation}\label{eq:approx_Rim_grad3_5}
    \left\| \text{D}g^{\star}(\0\ |\ \rb)\matG_{\text{D}g^{\star}(\0\ |\ \rb)}^{-\nicehalf} \right\| \leq 1.
\end{equation}
Plugging Eq. \eqref{eq:approx_Rim_grad3_4} and Eq. \eqref{eq:approx_Rim_grad3_5} into Eq. \eqref{eq:approx_Rim_grad3_2} and Eq. \eqref{eq:approx_Rim_grad3_3} leads to Eq. \eqref{eq:approx_Rim_grad3}.

To show Eq. \eqref{eq:approx_Rim_grad4}, we write explicitly the left-hand side of Eq. \eqref{eq:approx_Rim_grad4} using Eq. \eqref{eq:approx_grad_form} and the orthogonal projection on $T_{\rb}\widetilde{\mathcal{M}}$ defined via $\text{D}\widetilde{\varphi}(\0)$
\begin{eqnarray}\label{eq:approx_Rim_grad4_1}
\left\|{\gradtil{f(\rb)}}-{\gradMtil{f(\rb)}}\right\| = \left\| \text{D}g^{\star}(\0\ |\ \rb)\matG_{\text{D}g^{\star}(\0\ |\ \rb)}^{-1}\nabla p_{\rb}^{f}(\0) - \text{D}\widetilde{\varphi}(\0)\matG_{\text{D}\widetilde{\varphi}(\0)}^{-1}\text{D}\widetilde{\varphi}(\0)^{\T} \nabla f (\rb) \right\| \\
\leq \left\| \text{D}g^{\star}(\0\ |\ \rb)\matG_{\text{D}g^{\star}(\0\ |\ \rb)}^{-1}\nabla p_{\rb}^{f}(\0) - \text{D}g^{\star}(\0\ |\ \rb)\matG_{\text{D}g^{\star}(\0\ |\ \rb)}^{-1} \nabla \widehat{f} (\0) \right\| \nonumber\\
+ \left\| \text{D}g^{\star}(\0\ |\ \rb)\matG_{\text{D}g^{\star}(\0\ |\ \rb)}^{-1} \nabla \widehat{f} (\0) - \text{D}\widetilde{\varphi}(\0)\matG_{\text{D}\widetilde{\varphi}(\0)}^{-1}\text{D}\widetilde{\varphi}(\0)^{\T} \nabla f (\rb) \right\|. \nonumber
\end{eqnarray}
As before, to bound Eq. \eqref{eq:approx_Rim_grad4_1}, we bound each of the two terms in the inequality above. For the first term, we have already seen in Eq. \eqref{eq:approx_Rim_grad3_2} (plugging in Eq. \eqref{eq:approx_Rim_grad3_4} and Eq. \eqref{eq:approx_Rim_grad3_5}) that it is bounded by $c_{f} \sqrt{D} h^{m}$. For the second term, we use Eq. \eqref{eq:Lip_outer_grad}, Eq. \eqref{eq:approx_Rim_grad3_4}, and Eq. \eqref{eq:approx_Rim_grad3_5}, to have
\begin{eqnarray}
\left\| \text{D}g^{\star}(\0\ |\ \rb)\matG_{\text{D}g^{\star}(\0\ |\ \rb)}^{-1} \nabla \widehat{f} (\0) - \text{D}\widetilde{\varphi}(\0)\matG_{\text{D}\widetilde{\varphi}(\0)}^{-1}\text{D}\widetilde{\varphi}(\0)^{\T} \nabla f (\rb) \right\| \nonumber\\
= \left\| \text{D}g^{\star}(\0\ |\ \rb)\matG_{\text{D}g^{\star}(\0\ |\ \rb)}^{-1} \text{D}\varphi(\0)^{\T} \nabla f (\p) - \text{D}\widetilde{\varphi}(\0)\matG_{\text{D}\widetilde{\varphi}(\0)}^{-1}\text{D}\widetilde{\varphi}(\0)^{\T} \nabla f (\rb) \right\| \nonumber\\
\leq \left\| \text{D}g^{\star}(\0\ |\ \rb)\matG_{\text{D}g^{\star}(\0\ |\ \rb)}^{-1} \text{D}\varphi(\0)^{\T} \nabla f (\p) - \text{D}g^{\star}(\0\ |\ \rb)\matG_{\text{D}g^{\star}(\0\ |\ \rb)}^{-1} \text{D}g^{\star}(\0\ |\ \rb)^{\T} \nabla f (\p) \right\| \nonumber\\
+ \left\| \text{D}g^{\star}(\0\ |\ \rb)\matG_{\text{D}g^{\star}(\0\ |\ \rb)}^{-1} \text{D}g^{\star}(\0\ |\ \rb)^{\T} \nabla f (\p) - \text{D}{\varphi}(\0)\matG_{\text{D}{\varphi}(\0)}^{-1}\text{D}{\varphi}(\0)^{\T} \nabla f (\p) \right\| \nonumber\\
+ \left\| \text{D}{\varphi}(\0)\matG_{\text{D}{\varphi}(\0)}^{-1}\text{D}{\varphi}(\0)^{\T} \nabla f (\p) - \text{D}\widetilde{\varphi}(\0)\matG_{\text{D}\widetilde{\varphi}(\0)}^{-1}\text{D}\widetilde{\varphi}(\0)^{\T} \nabla f (\p) \right\|\nonumber\\
+ \left\| \text{D}\widetilde{\varphi}(\0)\matG_{\text{D}\widetilde{\varphi}(\0)}^{-1}\text{D}\widetilde{\varphi}(\0)^{\T} \nabla f (\p) - \text{D}\widetilde{\varphi}(\0)\matG_{\text{D}\widetilde{\varphi}(\0)}^{-1}\text{D}\widetilde{\varphi}(\0)^{\T} \nabla f (\rb) \right\|\nonumber\\
\leq \left\| \text{D}g^{\star}(\0\ |\ \rb) \matG_{\text{D}g^{\star}(\0\ |\ \rb)}^{-\nicehalf} \right\| \cdot \left\| \matG_{\text{D}g^{\star}(\0\ |\ \rb)}^{-\nicehalf} \right\| \cdot \left\| \text{D}\varphi(\0)^{\T} - \text{D}g^{\star}(\0\ |\ \rb)^{\T} \right\| \cdot \left\| \nabla f (\p) \right\| \nonumber\\
+ \left\| \text{D}g^{\star}(\0\ |\ \rb)\matG_{\text{D}g^{\star}(\0\ |\ \rb)}^{-1} \text{D}g^{\star}(\0\ |\ \rb)^{\T} - \text{D}{\varphi}(\0)\matG_{\text{D}{\varphi}(\0)}^{-1}\text{D}{\varphi}(\0)^{\T} \right\| \left\| \nabla f (\p) \right\| \nonumber\\
+ \left\| \text{D}{\varphi}(\0)\matG_{\text{D}{\varphi}(\0)}^{-1}\text{D}{\varphi}(\0)^{\T} - \text{D}\widetilde{\varphi}(\0)\matG_{\text{D}\widetilde{\varphi}(\0)}^{-1}\text{D}\widetilde{\varphi}(\0)^{\T} \right\| \left\| \nabla f (\p) \right\| + \left\| \nabla f(\rb)-\nabla f(\p) \right\| \nonumber\\
\leq \left( 2 c_{\mathcal{M}} \left\| \nabla f(\p) \right\| + c_{\widetilde{\mathcal{M}}} \left\| \nabla f(\p) \right\| + c_{\mmls} \right)\sqrt{D}h^{m}.\label{eq:approx_Rim_grad4_2}
\end{eqnarray}
Finally, using the bound for the first term and the second term in Eq. \eqref{eq:approx_Rim_grad4_1}, yields Eq. \eqref{eq:approx_Rim_grad4}.
\end{proof}

Next, we define an approximate-retraction via MMLS projection itself, thus requires $O(\tmmls)$ operations, performed
on $\rb+\xi$ where $\rb\in {\widetilde{\mathcal{M}}}$ and $\xi\in\widetilde{T}_{\rb}\widetilde{\mathcal{M}}$ is the step we take on the approximate-tangent space, constrained such that $\rb+\xi\in U_{\mathrm{unique}}$ to ensure MMLS projection
is defined.
Explicitly,
\begin{equation}
\widetilde{R}_{\rb}(\xi)\coloneqq {\mathcal P}_{m}^{h}(\rb+\xi)\quad (\widetilde{R}_{\rb}:\widetilde{T}_{\rb}\widetilde{\mathcal{M}}\to\widetilde{{\mathcal{M}}}).\label{eq:approximate-retraction_V1}
\end{equation}

Finally, we define an approximate-vector
transport using the orthogonal projection on $\widetilde{T}_{\rb}\widetilde{\mathcal{M}}$ (see \cite[Section 8.1.3]{AMS09}),
once we define an approximate-retraction $\widetilde{R}_{\rb}(\cdot):\widetilde{T}_{\rb}\widetilde{\mathcal{M}}\to\widetilde{{\mathcal{M}}}$, via
the following formula 
\[
\forall\eta,\xi\in\widetilde{T}_{\rb}\widetilde{\mathcal{M}},\ \widetilde{\tau}_{\eta}\xi\coloneqq\widetilde{\Pi}_{\widetilde{R}_{\rb}(\eta)}(\xi)\in\widetilde{T}_{\widetilde{R}_{\rb}(\eta)}{\mathcal{M}}.
\]
The computational cost of approximate-vector
transport is $(\tmmls + Dd^{2})$, since it consists of computing an approximate-retraction and applying an orthogonal projection on a vector.

Recall that for $\p\in U_{\mathrm{unique}}$, we have that ${\mathcal P}_{m}^{h}(\p)\in\widetilde{{\mathcal{M}}}$
and MMLS projection is smooth \cite[Theorem 4.21]{sober2020manifold}.
Thus, if we limit the step-size on $\widetilde{T}_{\rb}\widetilde{\mathcal{M}}$ to ensure
that $\p = \rb + \xi\in U_{\mathrm{unique}}$ ($U_{\mathrm{unique}}$ typically depends on the reach of the manifold $\mathcal{M}$, thus limiting the step-size $\xi$), then MMLS projection of $\rb+\xi$
on ${\cal \widetilde{{\mathcal{M}}}}$ is well defined. We assume that for all $\rb\in\widetilde{\mathcal{M}}$ it is possible to move along every $\xi\in \widetilde{T}_{\rb}\widetilde{\mathcal{M}}$ of limited size:
\begin{assumption}[Approximate-retraction domain]
\label{assu:ret_well_defined}For each $\rb\in\widetilde{\mathcal{M}}$, there exists a ball of radius $Q_{\rb}>0$ around $\rb$ such that for all $\xi\in \widetilde{T}_{\rb}\widetilde{\mathcal{M}}$, $\|\xi\|\leq Q_{\rb}$ we have that $\rb+\xi\in U_{\mathrm{unique}}$. Moreover, it is assumed that $Q\coloneqq\inf_{r\in\widetilde{\mathcal{M}}} Q_{\rb} > 0$. In other words, we assume that the approximate-retraction is defined (at least) in a compact subset of the approximate-tangent bundle 
\begin{equation}\label{eq:ret_def_in_tan_bun}
K \coloneqq \left\{(\rb, \xi)\in \widetilde{T}\widetilde{\mathcal{M}} \ : \  \|\xi\|\leq Q\right\} \subset \widetilde{T}\widetilde{\mathcal{M}}.
\end{equation}
\end{assumption}

Next, we want to ensure that the approximate-retraction we define satisfies (approximately) similar conditions to the two conditions in Definition \ref{def:retraction} modified for $\widetilde{T}_{\rb}\widetilde{\mathcal{M}}$
and ${\cal \widetilde{{\mathcal{M}}}}$. Explicitly, we want $\widetilde{R}_{\rb}(\0_{\rb}) = \rb$ and
$\text{D}\widetilde{R}_{\rb}(\0_{\rb})\approx \text{Id}_{\widetilde{T}_{\rb}\widetilde{\mathcal{M}}}$
where $\0_{\rb}\in\widetilde{T}_{\rb}\widetilde{\mathcal{M}}$ is the zero vector in $\widetilde{T}_{\rb}\widetilde{\mathcal{M}}$
and $\text{Id}_{\widetilde{T}_{\rb}\widetilde{\mathcal{M}}}$ is the identity mapping on
$\widetilde{T}_{\rb}\widetilde{\mathcal{M}}$. For $\text{D}\widetilde{R}_{\rb}(\0_{\rb})\approx \text{Id}_{\widetilde{T}_{\rb}\widetilde{\mathcal{M}}}$, recall that it can equivalently be shown that for all $\xi\in \widetilde{T}_{\rb}\widetilde{\mathcal{M}}$ we have $\frac{d}{dt}\widetilde{R}_{\rb}(t\xi)|_{t=0}\approx \xi$. These conditions are particularly important for the analysis of the global convergence of the Riemannian optimization methods (see for example \cite{boumal2019global}). We begin by showing that if we had access to the exact tangent spaces of $\widetilde{\mathcal{M}}$, then MMLS projection is indeed a retraction on $\widetilde{\mathcal{M}}$.

\begin{lemma}[A retraction on $\widetilde{\mathcal{M}}$]\label{lem:retraction_on_exact_tangent_property}
Let $\rb\in \widetilde{\mathcal{M}}$ and $\u\in T_{\rb}\widetilde{\mathcal{M}}$ such that $\rb + \u \in U_{\mathrm{unique}}$. Then the mapping
\begin{equation}
R_{\rb}^{\widetilde{\mathcal{M}}}(\xi)\coloneqq {\mathcal P}_{m}^{h}(\rb+\u)\quad (R_{\rb}^{\widetilde{\mathcal{M}}}:{T}_{\rb}\widetilde{{\mathcal{M}}}\to\widetilde{{\mathcal{M}}}),\label{eq:approximate-retraction_onMtil}
\end{equation}
is a retraction on $\widetilde{\mathcal{M}}$.
\end{lemma}

\begin{proof}
In order to prove it, we use Theorem 15 and Definition 14 from \cite{absil2012projection} (see Appendix \ref{subsec:background_proj-like}). In other words, we show that MMLS procedure applied on some $\u\in T_{\rb}\widetilde{\mathcal{M}}$ such that $\rb + \u \in U_{\mathrm{unique}}$ defines a retraction. 

To that end, first define the mapping $A$ from the tangent bundle of $\widetilde{\mathcal{M}}$ into $\text{Gr}(D-d,D)$ in the following way 
$$A: (\rb, \u) \mapsto H^{\perp}(\rb + \u).$$
This mapping $A$ is a retractor \cite[Definition 14]{absil2012projection}. Indeed, this mapping is smooth as a composition of the sum function and the first step of MMLS which provides a linear space $H(\cdot)$ which varies smoothly when $\theta_{1}(\cdot), \theta_{2}(\cdot)\in C^{\infty}$ \cite[Theorem 4.12]{sober2020manifold} (and thus its orthogonal compliment $H^{\perp} (\cdot)$ varies smoothly as well). In addition, the domain of $A$ contains a neighborhood of the zero section of $T\widetilde{\mathcal{M}}$, since that applying MMLS in $U_{\mathrm{unique}}$ result is also in $U_{\mathrm{unique}}$ \cite[Corollary 4.18]{sober2020manifold}. Finally, $A(\rb, 0_{\rb}) = H^{\perp}(\rb)$ has a trivial intersection with $T_{\rb}\widetilde{\mathcal{M}}$ since that for $\rb\in \widetilde{\mathcal{M}}$ MMLS projection (which returns a point on $H^{\perp}(\rb)$ that is the closest to $\q(\rb)+H(\rb)$ on $\widetilde{\mathcal{M}}$) returns a unique point on $\widetilde{\mathcal{M}}$ which is $\rb$ \cite[Lemma 4.7]{sober2020manifold}. 

Thus, from \cite[Theorem 15]{absil2012projection} $\widetilde{R}_{\rb}^{\widetilde{\mathcal{M}}}(\xi)$ is indeed a retraction on $\widetilde{\mathcal{M}}$, since that MMLS projection of $\rb+\u$ provides the closest points on $\widetilde{\mathcal{M}}\cap (r + \u + A(\rb, \u))$ to $\rb + \u$.
\end{proof}

Now, using the previous lemma we show the condition which the approximate-retraction satisfies.

\begin{lemma}[Approximate-retraction properties]\label{lem:retraction_2_property}
The approximate-retraction defined in Eq. \eqref{eq:approximate-retraction_V1} satisfies the following properties for $(\rb, \xi)\in K$, assuming $h$ is small enough:
\begin{enumerate}
    \item $\widetilde{R}_{\rb}(\0_{\rb}) = \rb$.
    \item We have
\begin{equation}\label{eq:retraction_2_property}
    \text{D}\widetilde{R}_{\rb}(\0_{\rb})[\xi] = \frac{d}{dt}\widetilde{R}_{\rb}(t\xi) |_{t=0} = \xi + \v'_{\rb, \xi}(0)\in T_{\rb}{\widetilde{\mathcal{M}}},
\end{equation}
where $\v'_{\rb,\xi}(0)$ is the derivative at $0$ of $\v_{\rb, \xi}(t):\R \to H^{\perp}(\rb+t\xi)$, which is a smooth vector-function such that $t\xi + \v_{\rb, \xi}(t)\in T_{\rb}{\widetilde{\mathcal{M}}}$. In other words, the approximate-retraction satisfies the second condition of Definition \ref{def:retraction}, with some correction $\v'_{\rb,\xi}(0)$ which depends on the proximity between $H(\rb)$ and $T_{\rb}{\widetilde{\mathcal{M}}}$, and the derivative of $H^{\perp}(\rb+t\xi)$ with respect to $t$.
\end{enumerate}

\end{lemma}
\begin{proof}
To show $\widetilde{R}_{\rb}(\0_{\rb}) = \rb$, recall that \cite[Lemma 4.7]{sober2020manifold}
ensures that given $\rb\in U_{\mathrm{unique}}$ such that ${\mathcal P}_{m}^{h}(\rb)\in\widetilde{{\mathcal{M}}}$,
we have that $\widetilde{R}_{\rb}(\0_{\rb})={\mathcal P}_{m}^{h}({\mathcal P}_{m}^{h}(\rb))={\mathcal P}_{m}^{h}(\rb)$. Moreover, for $h$ small enough and $\rb\in\widetilde{{\mathcal{M}}}$,
we have that ${\mathcal P}_{m}^{h}(\rb) = \rb$ (see \cite[Corollary 4.18]{sober2020manifold}), making $\widetilde{R}_{\rb}(\0_{\rb}) = \rb$. Thus, if $(\rb, \xi)\in K$, we ensure $\widetilde{R}_{\rb}(\0_{\rb}) = \rb$.

To show Eq. \eqref{eq:retraction_2_property}, we use Lemma \ref{lem:retraction_on_exact_tangent_property}. Since Eq. \eqref{eq:approximate-retraction_onMtil} is indeed a retraction, then for all $\u\in T_{\rb}\widetilde{\mathcal{M}}$
\begin{equation}\label{eq:ret_proof1}
\text{D} R_{\rb}^{\widetilde{\mathcal{M}}}(0_{\rb})[\u] = \text{D} {\mathcal P}_{m}^{h}(\rb) [\u] = \u.
\end{equation}

Next, using the uniqueness property on MMLS procedure \cite[Lemma 4.7]{sober2020manifold}, given $\rb+t\xi\in U_{\mathrm{unique}}$ where $\xi\in{\widetilde{T}_{\rb}\widetilde{\mathcal{M}}}$ and $t\in \R$, for all $\p\in U_{\mathrm{unique}}$ such that $\p$ is close enough to $\q(\rb+t\xi)$, and $\p - \q(\rb+t\xi) \in H^{\perp}(\rb+t\xi)$, we have $\q(\rb+t\xi)=\q(\p)$ and $H^{\perp}(\rb+t\xi) = H^{\perp}(\p)$. In particular, using the smoothness of $H^{\perp}(\cdot)$, for $t$ small enough there exists a unique $\p$ that is close enough to $\q(\rb+t\xi)$, i.e., $\p-\q(\rb+t\xi) \in H^{\perp}(\rb+t\xi)$, and $\p=\rb+\u$ where $\u\in T_{\rb}\widetilde{\mathcal{M}}$. Thus, in the domain of uniqueness of $\p$ we can define a smooth vector-function $\v_{\rb, \xi} (t):\R \to H^{\perp}(\rb+t\xi)$ such that $\u(t) \coloneqq t\xi + \v_{\rb, \xi} (t) \in T_{\rb}\widetilde{\mathcal{M}}$, $\v'_{\rb, \xi}(0) = 0$ and 
\begin{equation}\label{eq:ret_proof2}
     {\mathcal P}_{m}^{h}(\rb+t\xi) = {\mathcal P}_{m}^{h}(\rb+t\xi + \v_{\rb, \xi}(t)) = {\mathcal P}_{m}^{h}(\rb+\u(t)).
\end{equation}

Finally, using Eq. \eqref{eq:ret_proof1} and Eq. \eqref{eq:ret_proof2} we conclude the proof:
\begin{eqnarray*}
    \frac{d}{dt}\widetilde{R}_{\rb}(t\xi) |_{t=0} &=& \frac{d}{dt}{\mathcal P}_{m}^{h}(\rb+t\xi)|_{t=0} = \frac{d}{dt}{\mathcal P}_{m}^{h}(\rb+t\xi + \v_{\rb, \xi}(t)) |_{t=0}=\\
    &=& \frac{d}{dt}{\mathcal P}_{m}^{h} (r+\u(t)) |_{t=0} = \text{D} {\mathcal P}_{m}^{h}(\rb) \left[\u'(0) \right] = \xi + \v'_{\rb,\xi}(0),
\end{eqnarray*}
where the last equality is true since $\xi + \v'_{\rb,\xi}(0)\in T_{\rb}\widetilde{\mathcal{M}}$, which in itself arise from the definition of $\u(t)$ $$\u(t):\R \to T_{\rb}\widetilde{\mathcal{M}},$$ making $\u'(t) \in T_{\rb}\widetilde{\mathcal{M}}$.

\end{proof}

\begin{rem}[Bound on $\|\v'_{\rb,\xi}(0) \|$]\label{rem:v_bound}
Note that in Lemma \ref{lem:retraction_2_property}, the function $\v_{\rb, \xi}(t)$ is smooth in the compact set $K$, thus there exist some constant $L_{\v}>0$ such that $\|\v'_{\rb,\xi}(0) \| \leq L_{\v}$. Moreover, at the limit case $h\to 0$, the aforementioned bound goes to $0$, i.e., $L_{\v} \to 0$ (as MMLS converges). 
\end{rem}

Both Assumption \ref{assu:Lipschitz-gradient} and Lemma \ref{lem:retraction_2_property} allow us to conclude that the \emph{pullback} function $ f \circ \widetilde{R}: \widetilde{T}_{\rb}\widetilde{\mathcal{M}}\to \R$ satisfies a Lipschitz-type
gradient property in a similar manner to \cite[Lemma 4]{boumal2019global}. We show it in the following lemma (the proof is in Appendix \ref{subsec:proof_lipschitz}).

\begin{lemma}[Lipschitz-type gradient for pullbacks]\label{lem:pullback_Lip} Let $f:\R^{D} \to \R$ satisfy Assumption \ref{assu:Lipschitz-gradient}, and let Assumption \ref{assu:ret_well_defined} hold. Then, for all $\rb\in\widetilde{\mathcal{M}}$ and for all $\xi\in \widetilde{T}_{\rb}\widetilde{\mathcal{M}}$ such that $\|\xi\|\leq Q$, the pullback function $f \circ \widetilde{R}$ satisfies a Lipschitz-type gradient property with some $\widetilde{L}>0$ independent of $\rb$ and $\xi$:
\begin{equation*}
    |f(\widetilde{R}_{\rb}(\xi))-[f(\rb)+\text{D}f(\rb)[\text{D}\widetilde{R}_{\rb}(0_{\rb})[\xi]]|\leq\frac{\widetilde{L}}{2}\|\xi\|^{2},
\end{equation*}
where $$\text{D}f(\rb)[\text{D}\widetilde{R}_{\rb}(0_{\rb})[\xi]]=\dotprod{\gradMtil{f(\rb)}}{\xi+\v'_{\rb,\xi}(0)},$$ and $$\gradMtil{f(\rb)}=\Pi_{T_{\rb}{\widetilde{{\mathcal{M}}}}}(\nabla f(\rb)).$$
\end{lemma}

The proposed geometrical components in this section and their computational costs are summarized in
Table \ref{tab:riemannian-approx}. With these components
it is possible to adapt first-order Riemannian algorithms, e.g., Riemannian
gradient method and Riemannian CG, to our setting based on Riemannian
optimization \cite{AMS09,boumal2022intromanifolds}. An example MMLS-RO
gradient descent algorithm is described in Algorithm \ref{alg:MMLS-for-Riemannian}
(based on \cite[Algorithm 4.1]{boumal2022intromanifolds}) where the
step-size can be chosen in any standard way, i.e., fixed, optimal,
backtracking (e.g., Algorithm \ref{alg:MMLS-RO_backtrack}) constrained to satisfy that each iteration belongs to
$U_{\mathrm{unique}}$. Its global convergence (with a fixed step-size, and backtracking Armijo line-search) is analyzed in Section \ref{sec:Convergence-Analysis-of}. Also an example MMLS-RO CG algorithm is described in Algorithm \ref{alg:MMLS-RO_CG} (based on \cite[Algorithm 13]{AMS09}). Note that in all the proposed algorithms, we require $Q$ to be given. However, in most cases $Q$ is unknown in advance. We explain how to approximately have a step-size smaller than $Q$ in Subsection \ref{subsec:implement_details}. The effectiveness of all the presented algorithms is demonstrated empirically in Section \ref{sec:Numerical-experiments}.  

\begin{table}
\caption{\label{tab:riemannian-approx}Riemannian components for MMLS-RO}

\centering{}{\scriptsize{}}%
\begin{tabular}{|>{\centering}p{8.5cm}|>{\centering}p{4.3 cm}|>{\centering}p{4 cm}|}
\hline 
\textbf{\scriptsize{}Riemannian approximate components} & \textbf{\scriptsize{}Explicit formulas} & \textbf{\scriptsize{}Cost}\tabularnewline
\hline 
\hline
{\small{}MMLS projection of $\rb\in R^{D}$ on $\widetilde{\mathcal{M}}$ \cite{sober2020manifold}} & {\small{}${\mathcal P}_{m}^{h}(\rb)$} & {\small{}$\tmmls \coloneqq O(Dd^{m} + d^{3m})$}\tabularnewline
\hline
{\small{}Approximate-tangent space $\widetilde{T}_{\rb}\widetilde{\mathcal{M}}$ } & {\small{}$\Range{\text{D}g^{\star}(\0\ |\ \rb)}$} & {\small{}$O(\tmmls + Dd)$}\tabularnewline
\hline 
{\small{}Approximate-tangent bundle $\widetilde{T}\widetilde{\mathcal{M}}$ } & {\scriptsize{}$\left\{(\rb, \xi)\ :\ \rb\in \widetilde{\mathcal{M}}\ \cap \ \xi \in \widetilde{T}_{\rb}\widetilde{\mathcal{M}}\right\}$} & {-}\tabularnewline
\hline 
{\small{}Approximate-retraction of $\widetilde{R}_{\rb}(\xi)$} & {\small{}${\mathcal P}_{m}^{h}(\rb+\xi)$} & {\small{}$O(\tmmls)$}\tabularnewline
\hline 
{\small{}Orthogonal projection of $\widetilde{\Pi}_{\rb}(\xi)$} & {\small{}$\text{D}g^{\star}(\0\ |\ \rb)\text{D}g^{\star}(\0\ |\ \rb)^{\pinv}\xi$} & {\small{}$O(\tmmls + Dd^{2})$}\tabularnewline
\hline 
{\small{}Approximate-Riemannian gradient of a given $\nabla f(\rb)$} & {\small{}$\widetilde{\Pi}_{\rb}(\nabla f(\rb))$} & {\small{}$O(\tmmls + Dd^{2} + T_{\nabla  f})$}\tabularnewline
\hline 
{\small{}Approximate-Riemannian gradient when approximating $\nabla \widehat{f}$} & {\scriptsize{}$\text{D}g^{\star}(\0\ |\ \rb)\matG_{\text{D}g^{\star}(\0\ |\ \rb)}^{-1}\nabla p_{\rb}^{f}(\0)$} & {\small{}$O(\tmmls + Dd^{2} + d)$}\tabularnewline
\hline
{\small{}Approximate-vector transport, $\widetilde{\tau}_{\eta}\xi$, of
$\xi\in\widetilde{T}_{\rb}\widetilde{\mathcal{M}}$ to $\widetilde{T}_{\widetilde{R}_{\rb}(\eta)}{\mathcal{M}}$} & {\scriptsize{}$\widetilde{\tau}_{\eta}\xi\coloneqq\widetilde{\Pi}_{\widetilde{R}_{\rb}(\eta)}(\xi)$} & {\small{}$O(\tmmls + Dd^{2})$}\tabularnewline
\hline 
\end{tabular}{\scriptsize\par}
\end{table}

\begin{algorithm}[tb]
\caption{\label{alg:MMLS-for-Riemannian}MMLS-RO gradient descent algorithm (based on \cite[Algorithm 4.1]{boumal2022intromanifolds})}

\begin{algorithmic}[1]

\STATE\textbf{ Input: }$f$ a gradient Lipschitz function defined on $\R^{D}$ fully known or only given by samples at the points $S=\{\rb_{i}\}_{i=1}^{I}\subset{\mathcal{M}}$, where $S$ is a quasi-uniform sample
set. $Q$ from Assumption \ref{assu:ret_well_defined}. A tolerance $\epsilon>0$. 

\STATE \textbf{Choose an initial point:} $\p_{0}$ (a point
from the given point cloud, or a point close to it).

\STATE \textbf{Use} MMLS to form $(\q(\p_{0}),H(\p_{0}))$ and $\x_{0}={\mathcal P}_{m}^{h}(\p_{0})\in\widetilde{{\mathcal{M}}}$.

\STATE \textbf{Init} $i\leftarrow0$.

\STATE \textbf{While} $\left \|{\gradtil{f(\x_{i})}}\right \|>\epsilon$:

\STATE $\ $ \textbf{Compute} $\text{D}g^{\star}(\0\ |\ \x_{i})$ and
$\text{D}g^{\star}(\0\ |\ \x_{i})^{\pinv}$.

\STATE $\;$ \textbf{Take a search direction} $\xi_{\x_{i}}=-{\gradtil{f(\x_{i})}}$ on $\widetilde{T}_{\x_{i}}{\mathcal{M}}$, and a step-size $\alpha_{i}>0$ (fixed or via backtracking, e.g., Algorithm \ref{alg:MMLS-RO_backtrack}) such that $\left \|\alpha_{i}\xi_{\x_{i}}\right \|\leq Q$.

\STATE $\;$ \textbf{Set} $\x_{i+1}=\widetilde{R}_{\x_{i}}(\alpha_{i}\xi_{\x_{i}})$.

\STATE $\;$ $i\leftarrow i+1$.

\STATE \textbf{End while }

\STATE \textbf{Return $\x_{i}$}

\end{algorithmic}
\end{algorithm}

\begin{algorithm}[tb]
\caption{\label{alg:MMLS-RO_backtrack}Backtracking Armijo line-search (based on \cite[Algorithm 4.2]{boumal2022intromanifolds})}

\begin{algorithmic}[1]

\STATE\textbf{Input: }$\x_{i}\in \widetilde{\mathcal{M}}$, $\bar{\alpha}_{i} = \min \left\{\bar{\alpha}\ ,\ Q/\left\| \gradMtil{f(\x_{i})} \right\| \right\} >0$, $\gamma\in (0,1)$, $\delta\in (0,1)$. 

\STATE \textbf{Init:} $\alpha\leftarrow \bar{\alpha}_{i}$.

\STATE \textbf{While} $f(\x_{i})-f(\widetilde{R}_{\x_{i}}(-\alpha\gradtil{f(\x_{i})}) < \delta \alpha \left \|{\gradtil{f(\x_{i})}}\right \|^{2}$:

\STATE $\;$ \textbf{do} $\alpha\leftarrow \gamma t$.

\STATE \textbf{End while }

\STATE \textbf{Return} $\alpha$.

\end{algorithmic}
\end{algorithm}

\begin{algorithm}[tb]
\caption{\label{alg:MMLS-RO_CG}MMLS-RO CG algorithm (based on \cite[Algorithm 13]{AMS09})}

\begin{algorithmic}[1]

\STATE\textbf{Input: }$f$ a gradient Lipschitz function defined on $\R^{D}$ fully known or only given by samples at the points $S=\{\rb_{i}\}_{i=1}^{I}\subset{\mathcal{M}}$, where $S$ is a quasi-uniform sample
set. $Q$ from Assumption \ref{assu:ret_well_defined}. A tolerance $\epsilon>0$. 

\STATE \textbf{Choose an initial point:} $\p_{0}$ (a point
from the given point cloud, or a point close to it).

\STATE \textbf{Use} MMLS to form $(\q(\p_{0}),H(\p_{0}))$ and $\x_{0}={\mathcal P}_{m}^{h}(\p_{0})\in\widetilde{{\mathcal{M}}}$.

\STATE \textbf{Set} $\xi_{\x_{0}} = - {\gradtil{f(\x_{0})}}$.

\STATE \textbf{Init} $i\leftarrow0$.

\STATE \textbf{While} $\left \|{\gradtil{f(\x_{i})}}\right \|>\epsilon$:

\STATE $\;$ \textbf{Compute a step-size} $\alpha_{i}>0$ using a line-search backtracking procedure (e.g., Algorithm \ref{alg:MMLS-RO_backtrack}) such that $\left \|\alpha_{i}\xi_{\x_{i}}\right \|\leq Q$.

\STATE $\;$ \textbf{Set} $\x_{i+1}=\widetilde{R}_{\x_{i}}(\tau_{i}\xi_{\x_{i}})$.

\STATE $\;$ \textbf{Compute} $\beta_{i+1}$ via e.g., \cite[Eq. (8.28) or (8.29)]{AMS09}.

\STATE $\;$ \textbf{Set} $\xi_{\x_{i+1}} = - {\gradtil{f(\x_{i+1})}} + \beta_{i+1}\widetilde{\tau}_{\alpha_{i} \xi_{\x_{i}}}\xi_{\x_{i}}$.

\STATE $\;$ $i\leftarrow i+1$.

\STATE \textbf{End while }

\STATE \textbf{Return $\x_{i}$}

\end{algorithmic}
\end{algorithm}

\section{\label{sec:Convergence-Analysis-of}Convergence Analysis of MMLS-RO Gradient Algorithm}

In this section we analyze the global convergence of our
propose MMLS-RO gradient algorithm (Algorithm \ref{alg:MMLS-for-Riemannian}) with a fixed step-size and with backtracking (Algorithm \ref{alg:MMLS-RO_backtrack}) in a similar manner
to the analysis in \cite{boumal2019global}. First, we perform the analysis for a generic manifold learning method which provides us similar tools as MMLS (see Subsection \ref{subsec:Convergence-Analysis-of1}), and then we conclude the results for our proposed method (see Subsection \ref{subsec:Convergence-Analysis-of2}).

\subsection{\label{subsec:Convergence-Analysis-of1}Convergence Analysis for a Generic Manifold Learning Method}
In this subsection, we analyze the global convergence of a method which approximates the solution of 
$$\min _{\x\in{\mathcal{M}}} f(\x), $$
via some generic method which approximates $\mathcal{M}$, a $d$-dimensional smooth and closed manifold (possibly also approximating the Euclidean gradient of $ f$, and $f$ itself if required), using a quasi-uniform sample set with a fill distance $h$ (Definition \ref{def:Quasi-uniform-sample-set}). We begin by listing our method's assumptions regarding ${\mathcal M}$ (similar to the tools MMLS provides).
\begin{assumption}[Manifold learning method properties]\label{assu:list_of_req}
Assume that the method we use for approximating $\mathcal{M}$ provides:
\begin{enumerate}
    \item An approximating manifold $\widehat{\mathcal{M}}$ of $\mathcal{M}$, which is also $d$-dimensional smooth and closed.
    \item An approximation of the tangent spaces at $\rb\in \widehat{\mathcal{M}}$, denoted by $\widehat{T}_{\rb}\widehat{\mathcal{M}}$, that are $O(\sqrt{D}h^{m})$ approximations of the tangent spaces in the Euclidean norm, in the sense that the orthogonal projection of some $\u\in \R^{D}$ on $\widehat{T}_{\rb}\widehat{\mathcal{M}}$, $\Pi_{\widehat{T}_{\rb}\widehat{\mathcal{M}}}(\u)$, and on $T_{\rb}\widehat{\mathcal{M}}$, $\Pi_{T_{\rb}\widehat{\mathcal{M}}}(\u)$, satisfy
    \begin{equation*}
        \left\| \Pi_{\widehat{T}_{\rb}\widehat{\mathcal{M}}}(\u) - \Pi_{T_{\rb}\widehat{\mathcal{M}}}(\u) \right\| \leq c \left\| \u \right\| \sqrt{D} h^{m},
    \end{equation*}
    for some constant $c>0$ independent of $\rb$ and of $\u$.
    \item An approximation of the tangent bundle $\widehat{T}\widehat{\mathcal{M}}$.
    \item An approximation of the retraction map, $\widehat{R}_{(\cdot)}(\cdot):\widehat{T}\widehat{\mathcal{M}}\to \widehat{\mathcal{M}}$, defined over the set
    $$\widehat{K} \coloneqq \left\{(\rb, \xi)\in \widehat{T}\widehat{\mathcal{M}}\ :\ \|\xi \| \leq Q  \right\}, $$
    for some $Q>0$. It also satisfies $\widehat{R}_{\rb}(0_{\rb})=\rb$ and $\text{D}\widehat{R}_{\rb}(0_{\rb})[\xi]= \xi + \v_{\xi} \in T_{\rb}\widehat{\mathcal{M}}$, with $\|\v_{\xi}\|\leq L_{\v}$ such that $L_{\v}\to 0$ when $h\to 0$. Thus, recall that the Riemannian gradient on $\widehat{\mathcal{M}}$ (denoted by $\gradMhat{f(\cdot)}$) is the orthogonal projection on $T_{\rb}\widehat{\mathcal{M}}$, then 
    $$\text{D}f(\rb)[\text{D}\widehat{R}_{\rb}(0_{\rb})[\xi]] =  \dotprod{\gradMhat{f(\rb)}}{\xi+\v_{\xi}}.$$
    \item \label{item:list_of_req_grad}An approximation of the Riemannian gradient on $\widehat{\mathcal{M}}$, i.e., $\gradhat{f(\rb)}$, which satisfies 
    $$\|{\gradhat{f(\rb)}}-{\gradMhat{f(\rb)}}\| \leq (c_{0} + c_{1} \left\| \nabla  f (\rb) \right\|)\sqrt{D} h^{m},$$
    for some constants $c_{0}, c_{1} >0$ independent of $\rb$ and of $\nabla  f$, where $\gradMhat$ denotes the Riemannian gradient on $\widehat{\mathcal{M}}$.
\end{enumerate} 
\end{assumption}

We also state some general assumptions which are standard in analyzing global convergence of gradient methods (e.g., \cite{boumal2019global}). 

\begin{assumption}[Lower bound on $f$]
\label{assu:lower_bound} There exists a \textbf{lower bound} $f^{\star}>-\infty$ for $f$ on ${\mathcal{M}}\cup \widehat{\mathcal{M}}$,
i.e., $f(\x)\geq f^{\star}$ for all $\x\in{\mathcal{M}}\cup \widehat{\mathcal{M}}$.
\end{assumption}

\begin{assumption}[Restricted Lipschitz-type
gradient for pullbacks]
\label{assu:Lipschitz-type-gradient} There exists $\widehat{L}\geq0$ such that, for all
$\x_{k}\in\widehat{\mathcal{M}}$ among $\x_{0},\x_{1},...\in\widehat{\mathcal{M}}$ generated by a specified algorithm,
the compositions $f\circ\widehat{R}_{\x_{i}}$
satisfies that for all $(\x_{i},\xi)\in \widehat{K}$
\begin{equation*}
    |f(\widehat{R}_{\x_{i}}(\xi))-[f(\x_{i}) + \text{D}f(\x_{i})[\text{D}\widehat{R}_{\x_{i}}(\0_{\x_{i}})[\xi]]]|\leq\frac{\widehat{L}}{2}\|\xi\|^{2}.
\end{equation*}
\end{assumption}

Using the assumptions above (\ref{assu:list_of_req}, \ref{assu:lower_bound}, and \ref{assu:Lipschitz-type-gradient}), we can state the following theorems regarding a Riemannian gradient type-algorithm (e.g., Algorithm \ref{alg:MMLS-for-Riemannian}) with a fixed-step and with backtracking (i.e., Algorithm \ref{alg:MMLS-RO_backtrack}). Theorem \ref{thm:main_fixed_step} and Corollary \ref{cor:main_fixed_step} are for the fixed-step case (similar to \cite[theorems 3,5]{boumal2019global}), and Theorem \ref{thm:main_backtrack_step} and Corollary \ref{cor:main_backtrack_step} are for the backtracking procedure from Algorithm \ref{alg:MMLS-RO_backtrack} (similar to \cite[Theorems 3,7,8]{boumal2019global}).

\begin{thm}[Fixed-step gradient-descent decrease]\label{thm:main_fixed_step}
Under Assumptions \ref{assu:list_of_req} and \ref{assu:Lipschitz-type-gradient}, provided all the iterations are performed on $\widehat{\mathcal{M}}$, $h$ is small enough such that
\begin{equation}\label{eq:grad_cond0}
    8 L_{\v} \leq Q,
\end{equation}
and
\begin{equation}\label{eq:grad_cond}
    \max \left\{16\widehat{L}L_{\v}\ ,\  2 (c_{0} + c_{1} \left\| \nabla  f (\x_{i}) \right\|)\sqrt{D} h^{m} \right\} \leq \left\|\gradhat{f(\x_{i})}\right\|,
\end{equation}
holds for all the iterations, a Riemannian gradient algorithm, i.e., $\x_{i+1}\coloneqq \widehat{R}_{\x_{i}}(\alpha_{i}\xi_{\x_{i}})$, with the following search direction
\begin{equation}\label{eq:thm_step}
    \xi_{\x_{i}} \coloneqq - \gradhat{f(\x_{i})},
\end{equation}
and
\begin{equation}\label{eq:thm_step_size}
    \alpha_{i} \coloneqq \min \left\{\frac{Q}{\left\|\gradhat{f(\x_{i})} \right\|}\ ,\  \frac{1}{\widehat{L}} \beta_{i} \right\} > 0,
\end{equation}
where
\begin{equation}\label{eq:main_fixed_step_beta}
    \beta_{i} \coloneqq  \left(1 - \frac{(c_{0} + c_{1} \left\| \nabla  f (\x_{i}) \right\|)\sqrt{D} h^{m}}{\left\|\gradhat{f(\x_{i})} \right\|} \right) > 0.
\end{equation}
achieves the following decrease between two iterations
\begin{equation}\label{eq:ROgrad_dec}
    f(\x_{i})-f(\x_{i+1}) \geq \frac{1}{16}\min \left\{ \frac{\left\|\gradhat{f(\x_{i})} \right\|}{2\widehat{L}}, Q \right\} \left\|\gradhat{f(\x_{i})} \right\|.
\end{equation}
\end{thm}

\begin{proof}
From Assumption \ref{assu:Lipschitz-type-gradient} we have that for all $(\x_{i},\xi)\in \widehat{K}$ (where $\widehat{K}$ is from Assumption \ref{assu:list_of_req}),
\begin{eqnarray*}
     f(\widehat{R}_{\x_{i}}(\xi)) &\leq& f(\x_{i}) + \text{D}f(\x_{i})[\text{D}\widehat{R}_{\x_{i}}(\0_{x_{i}})[\xi]] + \frac{\widehat{L}}{2}\|\xi\|^{2}\\
     &=& f(\x_{i}) + \dotprod{\gradMhat{f(\x_{i})}}{\xi+\v_{\xi}} + \frac{\widehat{L}}{2}\|\xi\|^{2}\\
     &=& f(\x_{i}) + \dotprod{\gradMhat{f(\x_{i})} - \gradhat{f(\x_{i})}}{\xi} + \dotprod{\gradhat{f(\x_{i})}}{\xi}+\\
     &+& \dotprod{\gradMhat{f(\x_{i})} - \gradhat{f(\x_{i})}}{\v_{\xi}} + \dotprod{\gradhat{f(\x_{i})}}{\v_{\xi}} + \frac{\widehat{L}}{2}\|\xi\|^{2}.
\end{eqnarray*}
In particular, take $\xi = \alpha_{i}\xi_{\x_{i}}$, so that $f(\widehat{R}_{\x_{i}}(\xi))=f(\x_{i+1})$, and reorder the above inequality,
\begin{eqnarray*}
     f(\x_{i}) - f(\x_{i+1}) &\geq& - \dotprod{\gradMhat{f(\x_{i})} - \gradhat{f(\x_{i})}}{\alpha_{i}\xi_{\x_{i}}} - \dotprod{\gradhat{f(\x_{i})}}{\alpha_{i}\xi_{\x_{i}}}\\
     &-& \dotprod{\gradMhat{f(\x_{i})} - \gradhat{f(\x_{i})}}{\v_{\alpha_{i}\xi_{\x_{i}}}} - \dotprod{\gradhat{f(\x_{i})}}{\v_{\alpha_{i}\xi_{\x_{i}}}} - \frac{\widehat{L}}{2}\|\alpha_{i}\xi_{\x_{i}}\|^{2}.
\end{eqnarray*}

Substitute $\xi_{\x_{i}} = - \gradhat{f(\x_{i})}$, and use the Cauchy-Schwartz inequality together with the bound on the approximate Riemannian gradient (Assumption \ref{assu:list_of_req})
\begin{eqnarray*}
     f(\x_{i}) - f(\x_{i+1}) &\geq& \alpha_{i}\left\|\gradhat{f(\x_{i})} \right\|^{2} + \alpha_{i} \dotprod{\gradMhat{f(\x_{i})} - \gradhat{f(\x_{i})}}{\gradhat{f(\x_{i})}} - \frac{\widehat{L}\alpha_{i}^{2}}{2}\left\|\gradhat{f(\x_{i})} \right\|^{2}\\
     &-& \dotprod{\gradMhat{f(\x_{i})} - \gradhat{f(\x_{i})}}{\v_{\alpha_{i}\xi_{\x_{i}}}} - \dotprod{\gradhat{f(\x_{i})}}{\v_{\alpha_{i}\xi_{\x_{i}}}} \\
     &\geq& \alpha_{i}\left\|\gradhat{f(\x_{i})} \right\|^{2}\left(1 - \frac{(c_{0} + c_{1} \left\| \nabla  f (\x_{i}) \right\|)\sqrt{D} h^{m}}{\left\|\gradhat{f(\x_{i})} \right\|}- \frac{\widehat{L}\alpha_{i}}{2}\right)\\ &-& L_{\v} \left((c_{0} + c_{1} \left\| \nabla  f (\x_{i}) \right\|)\sqrt{D} h^{m} + \left\|\gradhat{f(\x_{i})} \right\| \right).
\end{eqnarray*}
The right-hand side of the above inequality is quadratic in $\alpha_{i}$, thus it is positive between its roots (if there exist two real roots). In particular, the maximal value is achieved for
\begin{equation*}
    \alpha_{i}^{\star} \coloneqq \frac{1}{\widehat{L}} \left(1 - \frac{(c_{0} + c_{1} \left\| \nabla  f (\x_{i}) \right\|)\sqrt{D} h^{m}}{\left\|\gradhat{f(\x_{i})} \right\|} \right),
\end{equation*}
which is positive if $\left\|\gradhat{f(\x_{i})} \right\| > (c_{0} + c_{1} \left\| \nabla  f (\x_{i}) \right\|)\sqrt{D} h^{m}$. Thus, we restrict the step $\alpha_{i}$ according to Eq. \eqref{eq:thm_step_size}, and using the conditions we impose in Eq. \eqref{eq:grad_cond0} and Eq. \eqref{eq:grad_cond} find a (positive) lower bound on $f(\x_{i}) - f(\x_{i+1})$.

With the step $\alpha_{i}$ according to Eq. \eqref{eq:thm_step_size}, we have
\begin{eqnarray*}
     f(\x_{i}) - f(\x_{i+1}) &\geq& \alpha_{i}\left\|\gradhat{f(\x_{i})} \right\|^{2}\left(\beta_{i} -  \frac{\widehat{L}\alpha_{i}}{2}\right) - L_{\v} \left((c_{0} + c_{1} \left\| \nabla  f (\x_{i}) \right\|)\sqrt{D} h^{m} + \left\|\gradhat{f(\x_{i})} \right\| \right)\\
     &=& \min \left\{Q\ ,\ \frac{\left\|\gradhat{f(\x_{i})} \right\| }{\widehat{L}}\beta_{i} \right\} \left (\beta_{i} - \frac{\widehat{L}}{2} \min \left\{\frac{Q}{\left\|\gradhat{f(\x_{i})} \right\|}\ ,\  \frac{1}{\widehat{L}} \beta_{i} \right\}\right ) \left\|\gradhat{f(\x_{i})} \right\| -\\
     &-& L_{\v} \left(\beta_{i} + \frac{2 (c_{0} + c_{1} \left\| \nabla  f (\x_{i}) \right\|)\sqrt{D} h^{m}}{\left\|\gradhat{f(\x_{i})} \right\|} \right) \left\|\gradhat{f(\x_{i})} \right\|\\
     &\geq& \left (\min \left\{Q\ ,\ \frac{\left\|\gradhat{f(\x_{i})} \right\| }{\widehat{L}}\beta_{i} \right\} \frac{1}{2} \beta_{i} - L_{\v} \left(\beta_{i} + \frac{2(c_{0} + c_{1} \left\| \nabla  f (\x_{i}) \right\|)\sqrt{D} h^{m}}{\left\|\gradhat{f(\x_{i})} \right\|} \right)\right )\left\|\gradhat{f(\x_{i})} \right\|.
\end{eqnarray*}
Constraint $\left\|\gradhat{f(\x_{i})} \right\| \geq 2(c_{0} + c_{1} \left\| \nabla  f (\x_{i}) \right\|)\sqrt{D} h^{m}$, leads to $\beta_{i} \geq 0.5$ or equivalently $-1 \geq - 2 \beta_{i}$, and we can conclude 
\begin{equation}\label{eq:thm_step_size_proof}
     f(\x_{i}) - f(\x_{i+1}) \geq \left (\min \left\{\frac{Q}{2}\ ,\ \frac{\left\|\gradhat{f(\x_{i})} \right\| }{4\widehat{L}} \right\} - 3 L_{\v} \right ) \beta_{i} \left\|\gradhat{f(\x_{i})} \right\|.
\end{equation}

Finally, adding the constraints $\left\|\gradhat{f(\x_{i})} \right\| \geq 16 \widehat{L}L_{\v}$ and $8 L_{\v} \leq Q$ (for $h$ small enough) lead to
\begin{eqnarray*}
     f(\x_{i}) - f(\x_{i+1}) &\geq& \left (\min \left\{4 L_{\v}\ ,\ 4 L_{\v} \right\} - 3 L_{\v} \right ) \beta_{i} \left\|\gradhat{f(\x_{i})} \right\| > 0.
\end{eqnarray*}
Thus, with the constraints $\left\|\gradhat{f(\x_{i})} \right\| \geq \max \left\{16\widehat{L}L_{\v}\ ,\  2 (c_{0} + c_{1} \left\| \nabla  f (\x_{i}) \right\|)\sqrt{D} h^{m}\right\}$ and $8 L_{\v} \leq Q$, i.e, 
$$L_{\v} \leq \frac{1}{2} \min \left\{\frac{Q}{4}\ ,\ \frac{\left\|\gradhat{f(\x_{i})} \right\| }{8\widehat{L}} \right\},$$
we can rewrite Eq. \eqref{eq:thm_step_size_proof} in the following way,
\begin{eqnarray*}
     f(\x_{i}) - f(\x_{i+1}) &\geq& \left (\min \left\{\frac{Q}{2}\ ,\ \frac{\left\|\gradhat{f(\x_{i})} \right\| }{4\widehat{L}} \right\} - 3 L_{\v} \right ) \beta_{i} \left\|\gradhat{f(\x_{i})} \right\| \\
     &\geq& \left (\min \left\{\frac{Q}{4}\ ,\ \frac{\left\|\gradhat{f(\x_{i})} \right\| }{8\widehat{L}} \right\} - \frac{3}{2} L_{\v} \right )  \left\|\gradhat{f(\x_{i})} \right\| \\
      &\geq& \left (1 - \frac{3}{4} \right ) \min \left\{\frac{Q}{4}\ ,\ \frac{\left\|\gradhat{f(\x_{i})} \right\| }{8\widehat{L}} \right\}  \left\|\gradhat{f(\x_{i})} \right\| \\
     &\geq& \frac{1}{16}\min \left\{ \frac{\left\|\gradhat{f(\x_{i})} \right\|}{2\widehat{L}}, Q \right\} \left\|\gradhat{f(\x_{i})} \right\|.
\end{eqnarray*}
\end{proof}

\begin{corollary}[Fixed-step gradient-descent iteration bound]\label{cor:main_fixed_step}
Under Assumptions \ref{assu:list_of_req}, \ref{assu:lower_bound} and \ref{assu:Lipschitz-type-gradient}, provided all the iterations are performed on $\widehat{\mathcal{M}}$, a Riemannian gradient algorithm, i.e., $\x_{i+1}\coloneqq \widehat{R}_{\x_{i}}(\alpha_{i}\xi_{\x_{i}})$, which is defined via Eqs. \eqref{eq:thm_step} and \eqref{eq:thm_step_size}, where Eq. \eqref{eq:grad_cond0} holds for $h$ small enough and assuming Eq. \eqref{eq:grad_cond} holds for all the iterations, then the algorithm returns a point $\x\in\widehat{\mathcal{M}}$ satisfying $f(\x)\leq f(\x_{0})$ and 
\begin{equation}\label{eq:grad_thm_bound}
   \left \|\gradhat{f(\x)} \right \| \leq \max \left\{16\widehat{L}L_{\v}\ ,\  2 (c_{0} + c_{1} \left\| \nabla  f (\x) \right\|)\sqrt{D} h^{m} \right\} + \varepsilon \coloneqq \varepsilon_{1} (\widehat{\mathcal{M}}).
\end{equation}
for any $\varepsilon>0$, provided we perform enough iterations. Moreover, if $\varepsilon_{1} (\widehat{\mathcal{M}}) > 2 Q \widehat{L} $, then the bound in Eq. \eqref{eq:grad_thm_bound} is achieved in at most
\begin{equation}\label{eq:it_num1}
    \left \lceil \frac{16(f(\x_0) - f^{\star})}{Q} \cdot \frac{1}{\varepsilon_{1} (\widehat{\mathcal{M}})} \right \rceil
\end{equation}
iterations. If $\varepsilon_{1} (\widehat{\mathcal{M}}) \leq 2 Q \widehat{L}$, then the bound in Eq. \eqref{eq:grad_thm_bound} is achieved in at most
\begin{equation}\label{eq:it_num2}
    \left \lceil 32(f(\x_0) - f^{\star})\widehat{L} \cdot \frac{1}{\varepsilon_{1} (\widehat{\mathcal{M}})^{2}} \right \rceil
\end{equation}
iterations.
Each iteration requires one cost and approximate-Riemannian
gradient evaluation, and one approximate-retraction computation.
\end{corollary}

\begin{proof}
Using Assumptions \ref{assu:list_of_req}, \ref{assu:lower_bound} and \ref{assu:Lipschitz-type-gradient}, iterations of the form $\x_{i+1}\coloneqq \widehat{R}_{\x_{i}}(\alpha_{i}\xi_{\x_{i}})$ with Eq. \eqref{eq:thm_step} and Eq. \eqref{eq:thm_step_size}, and also assuming that Eq. \eqref{eq:grad_cond0} holds for $h$ small enough and Eq. \eqref{eq:grad_cond} holds, then according to Theorem \ref{thm:main_fixed_step} Eq. \eqref{eq:ROgrad_dec} holds, i.e., 
\begin{equation*}
    f(\x_{i})-f(\x_{i+1}) \geq \frac{1}{16}\min \left\{ \frac{\left\|\gradhat{f(\x_{i})} \right\|}{2\widehat{L}}, Q \right\} \left\|\gradhat{f(\x_{i})} \right\|.
\end{equation*}
Thus, at the stopping point of the algorithm ,$\x$, we have $f(\x)\leq f(\x_{0})$.

Suppose that the algorithms did not stop after $K-1$ iterations, i.e., $\left \|\gradhat{f(\x_{i})} \right \| > \varepsilon_{1} (\widehat{\mathcal{M}})$ (thus, Eq. \eqref{eq:grad_cond0} and Eq. \eqref{eq:grad_cond} hold) for all $i=0,...,K-1$. Thus, using Assumption \ref{assu:lower_bound}, Eq. \eqref{eq:ROgrad_dec}, and a telescopic sum argument, we have
\begin{eqnarray*}
     f(\x_{0}) - f^{\star} \geq f(\x_{0}) - f(\x_{K}) \geq \sum _{i=0} ^{K-1} f(\x_{i}) - f(\x_{i+1}) &\geq& \sum _{i=0} ^{K-1} \frac{1}{16}\min \left\{ \frac{\left\|\gradhat{f(\x_{i})} \right\|}{2\widehat{L}}, Q \right\} \left\|\gradhat{f(\x_{i})} \right\|\\
     &>& \frac{K}{16}\min \left\{ \frac{\varepsilon_{1} (\widehat{\mathcal{M}})}{2\widehat{L}}, Q \right\} \varepsilon_{1} (\widehat{\mathcal{M}}).
\end{eqnarray*}
Thus,
\begin{equation}\label{eq:iter_num_proof}
    K < \frac{16(f(\x_{0}) - f^{\star})}{\min \left\{ \frac{\varepsilon_{1} (\widehat{\mathcal{M}})}{2\widehat{L}}, Q \right\} \varepsilon_{1} (\widehat{\mathcal{M}})},
\end{equation}
and the algorithm stops after 
\begin{equation*}
    K \geq \frac{16(f(\x_{0}) - f^{\star})}{\min \left\{ \frac{\varepsilon_{1} (\widehat{\mathcal{M}})}{2\widehat{L}}, Q \right\} \varepsilon_{1} (\widehat{\mathcal{M}})}.
\end{equation*}
But, then we reach a contradiction $f(\x_{0}) - f^{\star} > f(\x_{0}) - f^{\star}$. Thus, the algorithm must stop after $K$ iterations which satisfy Eq. \eqref{eq:iter_num_proof}.
\end{proof}

\begin{thm}[Backtracking gradient-descent decrease]\label{thm:main_backtrack_step}
Under Assumptions \ref{assu:list_of_req} and \ref{assu:Lipschitz-type-gradient}, given $\x\in {\widehat{\mathcal{M}}}$, provided $h$ is small enough such that 
\begin{equation}\label{eq:grad_backtrack_cond0}
    6 L_{\v} \leq Q,
\end{equation}
and
\begin{equation}\label{eq:grad_backtrack_cond}
    \max \left\{\frac{6L_{\v}}{\bar{\alpha}}\ ,\  2 (c_{0} + c_{1} \left\| \nabla  f (\x) \right\|)\sqrt{D} h^{m} \right\} \leq \left\|\gradhat{f(\x)}\right\|,
\end{equation}
the backtracking procedure from Algorithm \ref{alg:MMLS-RO_backtrack}, with an initial step-size
\begin{equation}\label{eq:thm_step_size_backtrack}
    \widehat{\alpha} \coloneqq \min \left\{ \bar{\alpha}\ ,\ \frac{Q}{\left \| \gradhat{f(\x)} \right \|}\right\}>0,
\end{equation}
and parameters $\gamma\in (0,1)$, $\delta\in (0, 0.5 \cdot \beta)$, and some $\bar{\alpha} > 0$, outputs a step-size $\alpha$ such that
\begin{equation}\label{eq:ROgrad_backtrack_dec}
    f(\x)-f(\widehat{R}_{\x}(-\alpha \gradhat{f(\x)})) \geq \delta \min \left\{ \bar{\alpha}\ ,\ \frac{Q}{\left \| \gradhat{f(\x)} \right \|}\ ,\ \frac{2\gamma \left(\frac{\beta}{2} - \delta \right)}{\widehat{L}} \right\} \left\|\gradhat{f(\x)} \right\|^2,
\end{equation}
where 
$$\beta \coloneqq  \left(1 - \frac{(c_{0} + c_{1} \left\| \nabla  f (\x) \right\|)\sqrt{D} h^{m}}{\left\|\gradhat{f(\x)} \right\|} \right) > 0,$$
after computing at most
\begin{equation*}
    \max \left\{1\ ,\ 2 + \log_{\gamma^{-1}}{\left(\frac{\widehat{\alpha} \widehat{L}}{2\gamma\left(\frac{\beta}{2} - \delta \right)} \right)} \right\}
\end{equation*}
retractions and cost function evaluations (assuming $f(\x)$ and $\gradhat{f(\x)}$ were
already computed).
\end{thm}

\begin{proof}
From Assumption \ref{assu:Lipschitz-type-gradient}, denoting for simplicity $\v_{-\alpha \gradhat{f(\x)})} \coloneqq \v$, we have that for all $\left(\x, -\alpha \gradhat{f(\x)})\right)\in \widehat{K}$ (Assumption \ref{assu:list_of_req}), when the initial $\alpha$ from Algorithm \ref{alg:MMLS-RO_backtrack} satisfies Eq. \eqref{eq:thm_step_size_backtrack}:
\begin{eqnarray*}
     f(\widehat{R}_{\x}(-\alpha \gradhat{f(\x)})) &\leq& f(\x) + \dotprod{\gradMhat{f(\x)}}{-\alpha \gradhat{f(\x)}+\v} + \frac{\widehat{L}\alpha^{2}}{2}\left\|\gradhat{f(\x)}\right\|^{2}\\
     &=& f(\x) - \alpha \dotprod{\gradMhat{f(\x)} - \gradhat{f(\x)}}{\gradhat{f(\x)}} - \alpha \left\|\gradhat{f(\x)}\right\|^{2}+\\
     &+& \dotprod{\gradMhat{f(\x)} - \gradhat{f(\x)}}{\v} + \dotprod{\gradhat{f(\x)}}{\v} + \frac{\widehat{L}\alpha^{2}}{2}\left\|\gradhat{f(\x)}\right\|^{2}.
\end{eqnarray*}
Using the Cauchy-Schwartz inequality, the bound on the approximate Riemannian gradient, the definition of $\beta$, and the assumption that $\left\|\gradhat{f(\x)} \right\| \geq 2 (c_{0} + c_{1} \left\| \nabla  f (\x) \right\|)\sqrt{D} h^{m}$ which leads to $\beta \geq 0.5$, we have
\begin{eqnarray*}
     f(\x) - f(\widehat{R}_{\x}(-\alpha \gradhat{f(\x)})) &\geq& \alpha \dotprod{\gradMhat{f(\x)} - \gradhat{f(\x)}}{\gradhat{f(\x)}} + \alpha \left\|\gradhat{f(\x)}\right\|^{2}\\
     &-& \dotprod{\gradMhat{f(\x)} - \gradhat{f(\x)}}{\v} - \dotprod{\gradhat{f(\x)}}{\v} - \frac{\widehat{L}\alpha^{2}}{2}\left\|\gradhat{f(\x)}\right\|^{2} \\
     &\geq& \alpha\left\|\gradhat{f(\x)} \right\|^{2}\left(1 - \frac{(c_{0} + c_{1} \left\| \nabla  f (\x) \right\|)\sqrt{D} h^{m}}{\left\|\gradhat{f(\x)} \right\|} -  \frac{\widehat{L}\alpha}{2}\right) \\
     &-& L_{\v} \left((c_{0} + c_{1} \left\| \nabla  f (\x) \right\|)\sqrt{D} h^{m} + \left\|\gradhat{f(\x)} \right\| \right)\\
     &=& \alpha\left\|\gradhat{f(\x)} \right\|^{2}\left(\beta -  \frac{\widehat{L}\alpha}{2}\right) - L_{\v} \left(\beta + \frac{2 (c_{0} + c_{1} \left\| \nabla  f (\x) \right\|)\sqrt{D} h^{m}}{\left\|\gradhat{f(\x)} \right\|} \right)\left\|\gradhat{f(\x)} \right\|\\
     &\geq& \alpha\left\|\gradhat{f(\x)} \right\|^{2}\left(\beta -  \frac{\widehat{L}\alpha}{2}\right) - 3 \beta L_{\v} \left\|\gradhat{f(\x)} \right\|.
\end{eqnarray*}

Assuming in addition that $\left\|\gradhat{f(\x)} \right\| \geq \frac{6L_{\v}}{\bar{\alpha}}$ and $6 L_{\v} \leq Q$ (given $h$ is small enough), ensures that the right-hand side of the above inequality is larger than 
\begin{equation}\label{eq:thm_bound}
    \frac{\alpha}{2}\left\|\gradhat{f(\x)} \right\|^{2}\left(\beta -  \widehat{L}\alpha\right).
\end{equation}
Indeed, we have that
\begin{equation*}
     \alpha\left\|\gradhat{f(\x)} \right\|^{2}\left(\beta -  \frac{\widehat{L}\alpha}{2}\right) - 3 \beta L_{\v} \left\|\gradhat{f(\x)} \right\| \geq \frac{\alpha}{2}\left\|\gradhat{f(\x)} \right\|^{2}\left(\beta -  \widehat{L}\alpha\right),
\end{equation*}
leads to 
\begin{equation*}
     \left\|\gradhat{f(\x)} \right\| \geq \frac{6 L_{\v}}{\alpha} \geq 6 L_{\v} \max \left\{\frac{1}{\bar{\alpha}}\ ,\  \frac{\left\|\gradhat{f(\x)} \right\|}{Q}\right\}.
\end{equation*}

Thus, we have from Eq. \eqref{eq:thm_bound} that
\begin{equation}\label{eq:thm_bound2}
     f(\x) - f(\widehat{R}_{\x}(-\alpha \gradhat{f(\x)})) \geq \frac{\alpha}{2}\left\|\gradhat{f(\x)} \right\|^{2}\left(\beta -  \widehat{L}\alpha\right).
\end{equation}
On the other hand, if the algorithm does not terminate for a certain $\alpha$, then
\begin{equation}\label{eq:thm_bound3}
     f(\x) - f(\widehat{R}_{\x}(-\alpha \gradhat{f(\x)})) < \delta \alpha  \left\|\gradhat{f(\x)} \right\|^{2}.
\end{equation}
Taking both Eq. \eqref{eq:thm_bound2} and Eq. \eqref{eq:thm_bound3} to hold simultaneously, then
\begin{equation}\label{eq:thm_bound4}
    \alpha \geq \frac{2\left( \frac{\beta}{2} - \delta \right)}{\widehat{L}} > 0,
\end{equation}
where we require $\delta < 0.5 \beta$.

Finally, for $\alpha$ which are smaller than the bound in Eq. \eqref{eq:thm_bound4}, Algorithm \ref{alg:MMLS-RO_backtrack} stops. It happens either if the initial $\alpha$, i.e., $\widehat{\alpha}$ is smaller than the right-hand side of Eq. \eqref{eq:thm_bound4}, either after a reduction of $\alpha$ by a factor $\beta$. Thus, Algorithm \ref{alg:MMLS-RO_backtrack} returns $\alpha$ which satisfies
\begin{equation*}
    \alpha \geq \min \left\{ \widehat{\alpha}\ ,\ \frac{2\gamma \left(\frac{\beta}{2} - \delta \right)}{\widehat{L}} \right\} = \min \left\{ \bar{\alpha}\ ,\ \frac{Q}{\left \| \gradhat{f(\x)} \right \|}\ ,\ \frac{2\gamma \left(\frac{\beta}{2} - \delta \right)}{\widehat{L}} \right\}.
\end{equation*}
Accordingly, the output of Algorithm \ref{alg:MMLS-RO_backtrack} is $\alpha = \widehat{\alpha}\gamma^{k-1}$, where $k$ is the
number of retractions and cost function evaluations in Algorithm \ref{alg:MMLS-RO_backtrack}. Therefore,
\begin{equation*}
    k = 1 + \log_{\gamma^{-1}}{\left(\frac{\widehat{\alpha}}{\alpha} \right)} \leq 1 + \max \left\{0\ ,\  \log_{\gamma^{-1}}{\left(\frac{\widehat{\alpha} \widehat{L}}{2\gamma\left(\frac{\beta}{2} - \delta \right)} \right)}\right\}.
\end{equation*}
\end{proof}

\begin{corollary}[Backtracking gradient-descent iteration bound]\label{cor:main_backtrack_step} Under Assumptions \ref{assu:list_of_req}, \ref{assu:lower_bound}, and \ref{assu:Lipschitz-type-gradient}, provided all the iterations are performed on $\widehat{\mathcal{M}}$, a Riemannian gradient algorithm, i.e., $\x_{i+1}\coloneqq \widehat{R}_{\x_{i}}(\alpha_{i}\xi_{\x_{i}})$, with the backtracking procedure from Algorithm \ref{alg:MMLS-RO_backtrack} to determine $\alpha_{i}$, with an initial step-size
\begin{equation}\label{eq:cor_step_size_backtrack}
    \widehat{\alpha}_{i} \coloneqq \min \left\{ \bar{\alpha}\ ,\ \frac{Q}{\left \| \gradhat{f(\x_{i})} \right \|}\right\}>0,
\end{equation}
and parameters $\gamma\in (0,1)$, $\delta\in (0, 0.25)$, and some $\bar{\alpha} > 0$, Eq. \eqref{eq:grad_backtrack_cond0} holds for $h$ small enough and assuming Eq. \eqref{eq:grad_backtrack_cond} holds for any $\x = \x_{i}$, then the algorithm returns a point $\x\in\widehat{\mathcal{M}}$ satisfying $f(\x)\leq f(\x_{0})$ and 
\begin{equation}\label{eq:grad_cor_bound}
   \left \|\gradhat{f(\x)} \right \| \leq \max \left\{\frac{6L_{\v}}{\bar{\alpha}}\ ,\  2 (c_{0} + c_{1} \left\| \nabla  f (\x) \right\|)\sqrt{D} h^{m} \right\} + \varepsilon \coloneqq \varepsilon_{2} (\widehat{\mathcal{M}}).
\end{equation}
for any $\varepsilon>0$, provided enough iterations are performed. Moreover, if 
$$\varepsilon_{2} (\widehat{\mathcal{M}}) > \frac{Q}{\left ( \min \left\{ \bar{\alpha}\ ,\ \frac{2\gamma \left(0.25 - \delta \right)}{\widehat{L}} \right\} \right)} , $$ 
then the bound in Eq. \eqref{eq:grad_cor_bound} is achieved in at most
\begin{equation}\label{eq:it_num3}
    \left \lceil \frac{(f(\x_0) - f^{\star})}{Q \delta } \cdot \frac{1}{\varepsilon_{2} (\widehat{\mathcal{M}})} \right \rceil
\end{equation}
iterations. If 
$$\varepsilon_{2} (\widehat{\mathcal{M}}) \leq \frac{Q}{\left ( \min \left\{ \bar{\alpha}\ ,\ \frac{2\gamma \left(0.25 - \delta \right)}{\widehat{L}} \right\} \right)} , $$
then the bound in Eq. \eqref{eq:grad_cor_bound} is achieved in at most
\begin{equation}\label{eq:it_num4}
    \left \lceil \frac{f(\x_0) - f^{\star}}{\delta \min \left\{ \bar{\alpha}\ ,\ \frac{2\gamma \left(0.25 - \delta \right)}{\widehat{L}} \right\}} \cdot \frac{1}{\varepsilon_{2} (\widehat{\mathcal{M}})^{2}} \right \rceil
\end{equation}
iterations.
Each iteration, $i$, requires at most 
\begin{equation*}
    \max \left\{1\ ,\ 2 + \log_{\gamma^{-1}}{\left(\frac{\widehat{\alpha}_{i} \widehat{L}}{2\gamma\left(\frac{\beta_{i}}{2} - \delta \right)} \right)} \right\}
\end{equation*}
retractions and cost function evaluations (assuming $f(\x_{0})$ and $\gradhat{f(\x_{0})}$ were already computed), where $\beta_{i}$ is defined in Eq. \eqref{eq:main_fixed_step_beta}.
\end{corollary}

\begin{proof}
The proof is similar to the proof of Corollary \ref{cor:main_fixed_step}. Using assumptions \ref{assu:list_of_req}, \ref{assu:lower_bound} and \ref{assu:Lipschitz-type-gradient}, iterations of the form $\x_{i+1}\coloneqq \widehat{R}_{\x_{i}}(\alpha_{i}\xi_{\x_{i}})$ with the backtracking procedure from Algorithm \ref{alg:MMLS-RO_backtrack} to determine $\alpha_{i}$, with an initial step-size defined in Eq. \eqref{eq:cor_step_size_backtrack}, and parameters $\gamma\in (0,1)$ and $\delta\in (0, 0.25)$, Eq. \eqref{eq:grad_backtrack_cond0} holds for $h$ small enough and Eq. \eqref{eq:grad_backtrack_cond} holds for any $\x = \x_{i}$, then according to Theorem \ref{thm:main_backtrack_step} Eq. \eqref{eq:ROgrad_backtrack_dec} holds for any iteration $i$, i.e., 
\begin{equation}\label{eq:ROgrad_dec_backtrack}
    f(\x_{i})-f(\x_{i+1}) \geq \delta \min \left\{ \bar{\alpha}\ ,\ \frac{Q}{\left \| \gradhat{f(\x_{i})} \right \|}\ ,\ \frac{2\gamma \left(\frac{\beta_{i}}{2} - \delta \right)}{\widehat{L}} \right\} \left\|\gradhat{f(\x_{i})} \right\|^2,
\end{equation}
after computing at most
\begin{equation*}
    \max \left\{1\ ,\ 2 + \log_{\gamma^{-1}}{\left(\frac{\widehat{\alpha}_{i} \widehat{L}}{2\gamma\left(\frac{\beta_{i}}{2} - \delta \right)} \right)} \right\}
\end{equation*}
retractions and cost function evaluations. Thus, at the stopping point of the algorithm , $\x$, we have $f(\x) \leq f(x_{0})$.

Suppose that the algorithms did not stop after $K-1$ iterations, i.e., $\left \|\gradhat{f(\x_{i})} \right \| > \varepsilon_{2} (\widehat{\mathcal{M}})$ (thus, Eq. \eqref{eq:grad_backtrack_cond0} and Eq. \eqref{eq:grad_backtrack_cond} hold) for all $i=0,...,K-1$. Thus, using Assumption \ref{assu:lower_bound}, Eq. \eqref{eq:ROgrad_dec_backtrack}, Eq. \eqref{eq:grad_backtrack_cond} for $\x = \x_{i}$ which implies $\beta_{i} > 0.5$ for any $i$, and a telescopic sum argument, we have
\begin{eqnarray*}
     f(\x_{0}) - f^{\star} \geq f(\x_{0}) - f(\x_{K}) \geq \sum _{i=0} ^{K-1} f(\x_{i}) - f(\x_{i+1}) \geq \\ \geq \sum _{i=0} ^{K-1} \delta \min \left\{ \bar{\alpha}\ ,\ \frac{Q}{\left \| \gradhat{f(\x_{i})} \right \|}\ ,\ \frac{2\gamma \left(\frac{\beta_{i}}{2} - \delta \right)}{\widehat{L}} \right\} \left\|\gradhat{f(\x_{i})} \right\|^2\\
     \geq \sum _{i=0} ^{K-1} \delta \min \left\{ \bar{\alpha}\left\|\gradhat{f(\x_{i})} \right\|\ ,\ Q\ ,\ \frac{2\gamma \left(0.25 - \delta \right)}{\widehat{L}}\left\|\gradhat{f(\x_{i})} \right\| \right\} \left\|\gradhat{f(\x_{i})} \right\|\\
     > K \delta \min \left\{ \bar{\alpha}\varepsilon_{2} (\widehat{\mathcal{M}})\ ,\ Q\ ,\ \frac{2\gamma \left(0.25 - \delta \right)}{\widehat{L}}\varepsilon_{2} (\widehat{\mathcal{M}}) \right\} \varepsilon_{2} (\widehat{\mathcal{M}}).
\end{eqnarray*}
Thus,
\begin{equation}\label{eq:iter_num_proof_cor}
    K < \frac{f(\x_{0}) - f^{\star}}{\delta \min \left\{ \bar{\alpha}\varepsilon_{2} (\widehat{\mathcal{M}})\ ,\ Q\ ,\ \frac{2\gamma \left(0.25 - \delta \right)}{\widehat{L}}\varepsilon_{2} (\widehat{\mathcal{M}}) \right\} \varepsilon_{2} (\widehat{\mathcal{M}})},
\end{equation}
and the algorithm stops after 
\begin{equation*}
    K \geq \frac{f(\x_{0}) - f^{\star}}{\delta \min \left\{ \bar{\alpha}\varepsilon_{2} (\widehat{\mathcal{M}})\ ,\ Q\ ,\ \frac{2\gamma \left(0.25 - \delta \right)}{\widehat{L}}\varepsilon_{2} (\widehat{\mathcal{M}}) \right\} \varepsilon_{2} (\widehat{\mathcal{M}})}.
\end{equation*}
But, then we reach a contradiction $f(\x_{0}) - f^{\star} > f(\x_{0}) - f^{\star}$. Thus, the algorithm must stop after $K$ iterations which satisfy Eq. \eqref{eq:iter_num_proof_cor}.
\end{proof}

\begin{rem}[Bounds on $\left \|\gradhat{f(\x)} \right \|$]
Corollary \ref{cor:main_fixed_step} and Corollary \ref{cor:main_backtrack_step} provide bounds for $\left \|\gradhat{f(\x)} \right \|$ at the end of the corresponding optimization algorithm in Eq. \eqref{eq:grad_thm_bound} and Eq. \eqref{eq:grad_cor_bound}. Note that both these bounds are as "good" as our manifold learning method. Explicitly, as $h \to 0$ also $L_{\v} \to 0$, thus $\varepsilon_{1} (\widehat{\mathcal{M}}) \to \varepsilon$ and $\varepsilon_{2} (\widehat{\mathcal{M}}) \to \varepsilon$, and correspondingly the bounds on $\left \|\gradhat{f(\x)} \right \|$ tighten.  
\end{rem}

\subsection{\label{subsec:Convergence-Analysis-of2}Consequences for MMLS-RO Gradient Algorithm}
In this subsection, we plug-in the results of the analysis of the previous subsection into our proposed components, i.e., Table \ref{tab:riemannian-approx}, using Algorithm \ref{alg:MMLS-for-Riemannian} with a fixed step-size or backtracking from Algorithm \ref{alg:MMLS-RO_backtrack}. We formulate it separately for the case of a cost function $f$ such that $f$ and its Euclidean gradient are given (Theorem \ref{thm:conc_fixed_step_fknown} and Theorem \ref{thm:conc_backtrack_fknown}), and the case where the Euclidean gradient of $\widehat{f}$ and possibly $f$ itself are approximated (Theorem \ref{thm:conc_fixed_step_funknown} and Theorem \ref{thm:conc_backtrack_funknown}).

\begin{thm}[Fixed-step MMLS-RO gradient-descent given $f$ and $\nabla  f$]\label{thm:conc_fixed_step_fknown}
Under Assumptions \ref{assu:MMLS_Samples}, \ref{assu:Lipschitz-gradient}, \ref{assu:ret_well_defined}, and \ref{assu:lower_bound}, given that $f$ and its Euclidean gradient are known, provided all the iterations are performed on points on $\widetilde{{\mathcal{M}}}$ where $\widetilde{{\mathcal{M}}}$ is a manifold, $h$ is small enough such that
\begin{equation}\label{eq:grad_cond_conc0}
    8 L_{\v} \leq Q,
\end{equation}
and
\begin{equation}\label{eq:grad_cond_conc}
    \max \left\{16\widetilde{L}L_{\v}\ ,\  2 c_{\widetilde{\mathcal{M}}} \left\| \nabla  f (\x_{i}) \right\|\sqrt{D} h^{m} \right\} \leq \left\|\gradtil{f(\x_{i})}\right\|,
\end{equation}
holds for all the iterations, where $\widetilde{L}$ and $L_{\v}$ are the constants from Lemma \ref{lem:pullback_Lip} and Remark \ref{rem:v_bound} correspondingly, a Riemannian gradient algorithm, i.e., Algorithm \ref{alg:MMLS-for-Riemannian} with a fixed step-size of the form
\begin{equation}\label{eq:conc_step}
    \xi_{\x_{i}} \coloneqq - \gradtil{f(\x_{i})},
\end{equation}
and
\begin{equation}\label{eq:conc_step_size}
    \alpha_{i} \coloneqq \min \left\{\frac{Q}{\left\|\gradtil{f(\x_{i})} \right\|}\ ,\  \frac{1}{\widetilde{L}} \beta_{i} \right\} > 0,
\end{equation}
where
\begin{equation}\label{eq:main_conc_fixed_step_beta}
    \beta_{i} \coloneqq  \left(1 - \frac{c_{\widetilde{\mathcal{M}}} \left\| \nabla  f (\x_{i}) \right\|\sqrt{D} h^{m}}{\left\|\gradtil{f(\x_{i})} \right\|} \right) > 0,
\end{equation}
returns a point $\x\in\widetilde{\mathcal{M}}$ satisfying $f(\x)\leq f(\x_{0})$ and 
\begin{equation}\label{eq:grad_conc_bound}
   \left \|\gradtil{f(\x)} \right \| \leq \max \left\{16\widetilde{L}L_{\v}\ ,\  2 c_{\widetilde{\mathcal{M}}} \left\| \nabla  f (\x) \right\|\sqrt{D} h^{m} \right\} + \varepsilon \coloneqq \varepsilon_{3} (\widetilde{\mathcal{M}}),
\end{equation}
for any $\varepsilon>0$, provided enough iterations are performed.

In addition, the following bound holds for the exact Riemannian gradient at a point $\p\in\mathcal{M}$ such that ${\mathcal P}_{m}^{h}(\p) = \x$ is the returned point
\begin{equation}\label{eq:true_grad_bound_fixed_step}
    \left\| \gradM{f(\p)}  \right\| \leq \left(c_{\mathcal{M}} G + c_{\mmls}\right)\sqrt{D}h^{m} + \varepsilon_{3} (\widetilde{\mathcal{M}}),
\end{equation}
where $\|\nabla  f(\x)\|\leq G$ and
\begin{equation}\label{eq:dist_p_and_x}
    \left\| \p - \x \right\| \leq c_{\mmls}\sqrt{D}h^{m+1}
\end{equation}
holds. Moreover, when $h\to 0$ then the bound in Eq. \eqref{eq:true_grad_bound_fixed_step} goes to $\varepsilon$, and the left-hand sides of Eq. \eqref{eq:grad_cond_conc0} and Eq. \eqref{eq:grad_cond_conc} go to $0$.

Finally, if $\varepsilon_{3} (\widetilde{\mathcal{M}}) > 2 Q \widetilde{L} $, then the bounds in Eqs. \eqref{eq:grad_conc_bound} and  \eqref{eq:true_grad_bound_fixed_step} are achieved in at most
\begin{equation}\label{eq:it_num5}
    \left \lceil \frac{16(f(\x_0) - f^{\star})}{Q} \cdot \frac{1}{\varepsilon_{3} (\widetilde{\mathcal{M}})} \right \rceil
\end{equation}
iterations. If $\varepsilon_{3} (\widetilde{\mathcal{M}}) \leq 2 Q \widetilde{L}$, then the bounds in Eq. \eqref{eq:grad_conc_bound} and Eq. \eqref{eq:true_grad_bound_fixed_step} are achieved in at most
\begin{equation}\label{eq:it_num6}
    \left \lceil 32(f(\x_0) - f^{\star})\widetilde{L} \cdot \frac{1}{\varepsilon_{3}(\widetilde{\mathcal{M}})^{2}} \right \rceil
\end{equation}
iterations.
Each iteration requires one cost and approximate-Riemannian
gradient evaluation, and one approximate-retraction computation.
\end{thm}

\begin{proof}
Assumptions \ref{assu:MMLS_Samples}, \ref{assu:Lipschitz-gradient}, and \ref{assu:ret_well_defined}, together with Lemma \ref{lem:our_proj_prop} (Appendix \ref{subsec:background_MMLS}), Lemma \ref{lem:approx_Rim_grad}, Lemma \ref{lem:retraction_2_property}, and Remark \ref{rem:v_bound}, imply that the components in Table \ref{tab:riemannian-approx} (and MMLS) satisfy Assumption \ref{assu:list_of_req}. Lemma \ref{lem:pullback_Lip} implies Assumption \ref{assu:Lipschitz-type-gradient}. Together with Assumption \ref{assu:lower_bound}, provided $h$ is small enough such that Eq. \eqref{eq:grad_cond_conc0} holds and assuming Eq. \eqref{eq:grad_cond_conc} holds for all the iterations, Corollary \ref{cor:main_fixed_step} can be applied for Algorithm \ref{alg:MMLS-for-Riemannian} with a fixed step-size according to Eq. \eqref{eq:conc_step_size}, using the components from Table \ref{tab:riemannian-approx}. Thus, we can conclude that Eq. \eqref{eq:grad_conc_bound} and the bounds on the number of iterations, Eq. \eqref{eq:it_num5} and Eq. \eqref{eq:it_num6}, hold. 

Finally, to show that Eq. \eqref{eq:true_grad_bound_fixed_step} holds, we use Eq. \eqref{eq:approx_Rim_grad2} from Lemma \ref{lem:approx_Rim_grad} and
\begin{eqnarray*}
 \left\| \gradM{f(\p)}  \right\| \leq \left\| \gradM{f(\p)} - \gradtil{f(\x)} \right\| + \left \|\gradtil{f(\x)} \right \| \leq \\
 \leq \left(c_{\mathcal{M}} \|\nabla  f(\x)\| + c_{\mmls}\right)\sqrt{D}h^{m} + \varepsilon_{3} (\widetilde{\mathcal{M}}) \leq \left(c_{\mathcal{M}} G + c_{\mmls}\right)\sqrt{D}h^{m} + \varepsilon_{3} (\widetilde{\mathcal{M}}),
\end{eqnarray*}
where $\|\nabla  f(\x)\|\leq G$  is finite and exists since $\widetilde{\mathcal{M}}$ is a compact manifold and $\nabla  f(\x)$ is assumed to be Lipschitz continuous. Eq. \eqref{eq:dist_p_and_x} holds due to Eq. \eqref{eq:distbetween_p_and_r}. Using Remark \ref{rem:v_bound} we get that when $h\to 0$ then $L_{\v} \to 0$, and we can conclude that the bound in Eq. \eqref{eq:true_grad_bound_fixed_step} goes to $\varepsilon$, and that the left-hand sides of Eq. \eqref{eq:grad_cond_conc0} and Eq. \eqref{eq:grad_cond_conc} go to $0$. 
\end{proof}

\begin{thm}[Backtracking MMLS-RO gradient-descent given $f$ and $\nabla  f$]\label{thm:conc_backtrack_fknown}
Under Assumptions \ref{assu:MMLS_Samples}, \ref{assu:Lipschitz-gradient}, \ref{assu:ret_well_defined}, and \ref{assu:lower_bound}, given that $f$ and its Euclidean gradient are known, provided all the iterations are performed on points on $\widetilde{{\mathcal{M}}}$ where $\widetilde{{\mathcal{M}}}$ is a manifold, $h$ is small enough such that
\begin{equation}\label{eq:grad_cond_conc_backtrack0}
    6 L_{\v} \leq Q,
\end{equation}
and
\begin{equation}\label{eq:grad_cond_conc_backtrack}
    \max \left\{\frac{6L_{\v}}{\bar{\alpha}}\ ,\  2 c_{\widetilde{\mathcal{M}}} \left\| \nabla  f (\x_{i}) \right\|\sqrt{D} h^{m} \right\} \leq \left\|\gradtil{f(\x_{i})}\right\|, 
\end{equation}
holds for all the iterations, where $L_{\v}$ is the constant from  Remark \ref{rem:v_bound}, a Riemannian gradient algorithm, i.e., Algorithm \ref{alg:MMLS-for-Riemannian} with the backtracking procedure from Algorithm \ref{alg:MMLS-RO_backtrack} to determine $\alpha_{i}$, with an initial step-size
\begin{equation}\label{eq:conc_step_size_backtrack}
    \widehat{\alpha}_{i} \coloneqq \min \left\{ \bar{\alpha}\ ,\ \frac{Q}{\left \| \gradtil{f(\x_{i})} \right \|}\right\}>0,
\end{equation}
and parameters $\gamma\in (0,1)$, $\delta\in (0, 0.25)$, and some $\bar{\alpha} > 0$, returns a point $\x\in\widetilde{\mathcal{M}}$ satisfying $f(\x)\leq f(\x_{0})$ and 
\begin{equation}\label{eq:grad_conc_bound_backtrack}
   \left \|\gradtil{f(\x)} \right \| \leq \max \left\{\frac{6L_{\v}}{\bar{\alpha}}\ ,\  2 c_{\widetilde{\mathcal{M}}} \left\| \nabla  f (\x) \right\|\sqrt{D} h^{m} \right\} + \varepsilon \coloneqq \varepsilon_{4} (\widetilde{\mathcal{M}}),
\end{equation}
for any $\varepsilon>0$, provided enough iterations are performed.

In addition, the following bound holds for the exact Riemannian gradient at a point $\p\in\mathcal{M}$ such that ${\mathcal P}_{m}^{h}(\p) = \x$ is the returned point
\begin{equation}\label{eq:true_grad_bound_backtrack}
    \left\| \gradM{f(\p)}  \right\| \leq \left(c_{\mathcal{M}} G + c_{\mmls}\right)\sqrt{D}h^{m} + \varepsilon_{4} (\widetilde{\mathcal{M}}),
\end{equation}
where $\|\nabla  f(\x)\|\leq G$ and Eq. \eqref{eq:dist_p_and_x} holds as well. Moreover, when $h\to 0$ then the bound in Eq. \eqref{eq:true_grad_bound_backtrack} goes to $\varepsilon$, and the left-hand sides  of Eq. \eqref{eq:grad_cond_conc_backtrack0} and Eq. \eqref{eq:grad_cond_conc_backtrack} go to $0$.

Finally, if 
$$\varepsilon_{4} (\widetilde{\mathcal{M}}) > \frac{Q}{\left ( \min \left\{ \bar{\alpha}\ ,\ \frac{2\gamma \left(0.25 - \delta \right)}{\widetilde{L}} \right\} \right)} , $$ 
where $\widetilde{L}$ is the constant from Lemma \ref{lem:pullback_Lip}, then the bounds in Eq. \eqref{eq:grad_conc_bound_backtrack} and Eq. \eqref{eq:true_grad_bound_backtrack} are achieved in at most
\begin{equation}\label{eq:it_num7}
    \left \lceil \frac{(f(\x_0) - f^{\star})}{Q \delta } \cdot \frac{1}{\varepsilon_{4} (\widetilde{\mathcal{M}})} \right \rceil
\end{equation}
iterations. If 
$$\varepsilon_{4} (\widetilde{\mathcal{M}}) \leq \frac{Q}{\left ( \min \left\{ \bar{\alpha}\ ,\ \frac{2\gamma \left(0.25 - \delta \right)}{\widetilde{L}} \right\} \right)} , $$
then the bounds in Eq. \eqref{eq:grad_conc_bound_backtrack} and Eq. \eqref{eq:true_grad_bound_backtrack} are achieved in at most
\begin{equation}\label{eq:it_num8}
    \left \lceil \frac{f(\x_0) - f^{\star}}{\delta \min \left\{ \bar{\alpha}\ ,\ \frac{2\gamma \left(0.25 - \delta \right)}{\widetilde{L}} \right\}} \cdot \frac{1}{\varepsilon_{4}(\widetilde{\mathcal{M}})^{2}} \right \rceil
\end{equation}
iterations.
Each iteration, $i$, requires at most 
\begin{equation}\label{eq:calc_num_conc}
    \max \left\{1\ ,\ 2 + \log_{\gamma^{-1}}{\left(\frac{\widehat{\alpha}_{i} \widetilde{L}}{2\gamma\left(\frac{\beta_{i}}{2} - \delta \right)} \right)} \right\}
\end{equation}
retractions and cost function evaluations (assuming $f(\x_{0})$ and $\gradtil{f(\x_{0})}$ were already computed), where $\beta_{i}$ is defined in Eq. \eqref{eq:main_conc_fixed_step_beta}.
\end{thm}

\begin{proof}
Assumptions \ref{assu:MMLS_Samples}, \ref{assu:Lipschitz-gradient}, and \ref{assu:ret_well_defined}, together with Lemma \ref{lem:our_proj_prop} (Appendix \ref{subsec:background_MMLS}), Lemma \ref{lem:approx_Rim_grad}, Lemma \ref{lem:retraction_2_property}, and Remark \ref{rem:v_bound}, imply that the components in Table \ref{tab:riemannian-approx} (and MMLS) satisfy Assumption \ref{assu:list_of_req}. Lemma \ref{lem:pullback_Lip} implies Assumption \ref{assu:Lipschitz-type-gradient}. Together with Assumption \ref{assu:lower_bound}, provided $h$ is small enough such that Eq. \eqref{eq:grad_cond_conc_backtrack0} holds and assuming Eq. \eqref{eq:grad_cond_conc_backtrack} holds for all the iterations, Corollary \ref{cor:main_backtrack_step} can be applied for Algorithm \ref{alg:MMLS-for-Riemannian} with the backtracking procedure from Algorithm \ref{alg:MMLS-RO_backtrack} to determine $\alpha_{i}$, with an initial step-size according to Eq. \eqref{eq:conc_step_size_backtrack} and parameters $\gamma\in (0,1)$, $\delta\in (0, 0.25)$, and some $\bar{\alpha} > 0$, using the components from Table \ref{tab:riemannian-approx}. Thus, we can conclude that Eq. \eqref{eq:grad_conc_bound_backtrack} and the bounds on the number of iterations, Eq. \eqref{eq:it_num7} and Eq. \eqref{eq:it_num8}, hold. 

Finally, to show that Eq. \eqref{eq:true_grad_bound_backtrack} holds, we use Eq. \eqref{eq:approx_Rim_grad2} from Lemma \ref{lem:approx_Rim_grad} and
\begin{eqnarray*}
 \left\| \gradM{f(\p)}  \right\| \leq \left\| \gradM{f(\p)} - \gradtil{f(\x)} \right\| + \left \|\gradtil{f(\x)} \right \| \leq \\
 \leq \left(c_{\mathcal{M}} \|\nabla  f(\x)\| + c_{\mmls}\right)\sqrt{D}h^{m} + \varepsilon_{4} (\widetilde{\mathcal{M}}) \leq \left(c_{\mathcal{M}} G + c_{\mmls}\right)\sqrt{D}h^{m} + \varepsilon_{4} (\widetilde{\mathcal{M}}),
\end{eqnarray*}
where $\|\nabla  f(\x)\|\leq G$  is finite and exists since $\widetilde{\mathcal{M}}$ is a compact manifold and $\nabla  f(\x)$ is assumed to be Lipschitz continuous. Eq. \eqref{eq:dist_p_and_x} holds due to Eq. \eqref{eq:distbetween_p_and_r}. Using Remark \ref{rem:v_bound} we get that when $h\to 0$ then $L_{\v} \to 0$, and we can conclude that the bound in Eq. \eqref{eq:true_grad_bound_backtrack} goes to $\varepsilon$, and the left-hand sides of Eq. \eqref{eq:grad_cond_conc_backtrack0} and Eq. \eqref{eq:grad_cond_conc_backtrack} go to $0$.
\end{proof}

\begin{thm}[Fixed-step MMLS-RO gradient-descent approximating $\nabla \widehat{f}$]\label{thm:conc_fixed_step_funknown}
Under Assumptions \ref{assu:MMLS_func_Samples}, \ref{assu:Lipschitz-gradient}, \ref{assu:ret_well_defined}, and \ref{assu:lower_bound}, $\nabla \widehat{f}$ is approximated for the Riemannian gradient approximation (Eq. \eqref{eq:approximate_f_Rgrad}), provided all the iterations are performed on points on $\widetilde{{\mathcal{M}}}$ where $\widetilde{{\mathcal{M}}}$ is a manifold, $h$ is small enough such that
\begin{equation}\label{eq:grad_cond_conc_funknown0}
    8 L_{\v} \leq Q,
\end{equation}
and
\begin{equation}\label{eq:grad_cond_conc_funknown}
    \max \left\{16\widetilde{L}L_{\v}\ ,\  2 \left( 2 c_{\mathcal{M}} \left\| \nabla  f(\p_{i}) \right\| + c_{\widetilde{\mathcal{M}}} \left\| \nabla  f(\p_{i}) \right\| + c_{f} + c_{\mmls} \right)\sqrt{D}h^{m} \right\} \leq \left\|\gradtil{f(\x_{i})}\right\|, 
\end{equation}
holds for all the iterations, where ${\mathcal P}_{m}^{h}(\p_{i}) = \x_{i}$, $\p_{i}\in\mathcal{M}$, $\widetilde{L}$ and $L_{\v}$ are the constants from Lemma \ref{lem:pullback_Lip} and Remark \ref{rem:v_bound} correspondingly, a Riemannian gradient algorithm, i.e., Algorithm \ref{alg:MMLS-for-Riemannian} with a fixed step-size of the form
\begin{equation}\label{eq:conc_step_funknown}
    \xi_{\x_{i}} \coloneqq - \gradtil{f(\x_{i})},
\end{equation}
and
\begin{equation}\label{eq:conc_step_size_funknown}
    \alpha_{i} \coloneqq \min \left\{\frac{Q}{\left\|\gradtil{f(\x_{i})} \right\|}\ ,\  \frac{1}{\widetilde{L}} \beta_{i} \right\} > 0,
\end{equation}
where
\begin{equation}\label{eq:main_conc_fixed_step_beta_funknown}
    \beta_{i} \coloneqq  \left(1 - \frac{\left( 2 c_{\mathcal{M}} \left\| \nabla  f(\p_{i}) \right\| + c_{\widetilde{\mathcal{M}}} \left\| \nabla  f(\p_{i}) \right\| + c_{f} + c_{\mmls} \right)\sqrt{D}h^{m}}{\left\|\gradtil{f(\x_{i})} \right\|} \right) > 0, \ {\mathcal P}_{m}^{h}(\p_{i}) = \x_{i},\ \p_{i}\in\mathcal{M},
\end{equation}
returns a point $\x\in\widetilde{\mathcal{M}}$ satisfying $f(\x)\leq f(\x_{0})$ and 
\begin{equation}\label{eq:grad_conc_bound_funknown}
   \left \|\gradtil{f(\x)} \right \| \leq \max \left\{16\widetilde{L}L_{\v}\ ,\  2 \left(2 c_{\mathcal{M}} \left\| \nabla  f(\p) \right\| + c_{\widetilde{\mathcal{M}}} \left\| \nabla  f(\p) \right\| + c_{f} + c_{\mmls} \right)\sqrt{D}h^{m} \right\} + \varepsilon \coloneqq \varepsilon_{5} (\widetilde{\mathcal{M}}),
\end{equation}
where $\p\in\mathcal{M}$ satisfies ${\mathcal P}_{m}^{h}(\p) = \x$, for any $\varepsilon>0$, provided enough iterations are performed.

In addition, the following bound holds for the exact Riemannian gradient at a point $\p\in\mathcal{M}$ such that ${\mathcal P}_{m}^{h}(\p) = \x$ is the returned point
\begin{equation}\label{eq:true_grad_bound_fixed_step_funknown}
    \left\| \gradM{f(\p)}  \right\| \leq \left( c_{f} + 2 c_{\mathcal{M}}  G \right)\sqrt{D}h^{m} + \varepsilon_{5} (\widetilde{\mathcal{M}}),
\end{equation}
where $\|\nabla  f(\p)\|\leq G$ and Eq. \eqref{eq:dist_p_and_x} holds as well. Moreover, when $h\to 0$ then the bound in Eq. \eqref{eq:true_grad_bound_fixed_step_funknown} goes to $\varepsilon$, and the left-hand sides of Eq. \eqref{eq:grad_cond_conc_funknown0} and Eq. \eqref{eq:grad_cond_conc_funknown} go to $0$.

Finally, if $\varepsilon_{5} (\widetilde{\mathcal{M}}) > 2 Q \widetilde{L} $, then the bounds in Eq. \eqref{eq:grad_conc_bound_funknown} and Eq. \eqref{eq:true_grad_bound_fixed_step_funknown} are achieved in at most
\begin{equation}\label{eq:it_num9}
    \left \lceil \frac{16(f(\x_0) - f^{\star})}{Q} \cdot \frac{1}{\varepsilon_{5} (\widetilde{\mathcal{M}})} \right \rceil
\end{equation}
iterations. If $\varepsilon_{5} (\widetilde{\mathcal{M}}) \leq 2 Q \widetilde{L}$, then the bounds in Eq. \eqref{eq:grad_conc_bound_funknown} and Eq. \eqref{eq:true_grad_bound_fixed_step_funknown} are achieved in at most
\begin{equation}\label{eq:it_num10}
    \left \lceil 32(f(\x_0) - f^{\star})\widetilde{L} \cdot \frac{1}{\varepsilon_{5}(\widetilde{\mathcal{M}})^{2}} \right \rceil
\end{equation}
iterations.
Each iteration requires one cost and approximate-Riemannian
gradient evaluation, and one approximate-retraction computation.
\end{thm}

\begin{proof}
Assumptions \ref{assu:MMLS_func_Samples}, \ref{assu:Lipschitz-gradient}, and \ref{assu:ret_well_defined}, together with Lemma \ref{lem:our_proj_prop} (Appendix \ref{subsec:background_MMLS}), Lemma \ref{lem:approx_Rim_grad}, Lemma \ref{lem:retraction_2_property}, and Remark \ref{rem:v_bound}, imply that the components in Table \ref{tab:riemannian-approx} (and MMLS) satisfy Assumption \ref{assu:list_of_req}, where Item \ref{item:list_of_req_grad} is replaced with Eq. \eqref{eq:approx_Rim_grad4} and $\rb$ is replaced with $\p\in \mathcal{M}$ such that ${\mathcal P}_{m}^{h}(\p)=\rb$. Lemma \ref{lem:pullback_Lip} implies Assumption \ref{assu:Lipschitz-type-gradient}. Together with Assumption \ref{assu:lower_bound}, provided $h$ is small enough such that Eq. \eqref{eq:grad_cond_conc_funknown} holds and assuming Eq. \eqref{eq:grad_cond_conc_funknown} holds for all the iterations, Corollary \ref{cor:main_fixed_step} can be applied for Algorithm \ref{alg:MMLS-for-Riemannian} with a fixed step-size according to Eq. \eqref{eq:conc_step_size_funknown}, using the components from Table \ref{tab:riemannian-approx}. Thus, we can conclude that Eq. \eqref{eq:grad_conc_bound_funknown} and the bounds on the number of iterations, Eq. \eqref{eq:it_num9} and Eq. \eqref{eq:it_num10}, hold. 

Finally, to show that Eq. \eqref{eq:true_grad_bound_fixed_step_funknown} holds, we use Eq. \eqref{eq:approx_Rim_grad3} from Lemma \ref{lem:approx_Rim_grad} and
\begin{eqnarray*}
 \left\| \gradM{f(\p)}  \right\| \leq \left\| \gradM{f(\p)} - \gradtil{f(\x)} \right\| + \left \|\gradtil{f(\x)} \right \| \leq \\
 \leq \left( c_{f} + 2 c_{\mathcal{M}}  \left\| \nabla  f(\p) \right\| \right)\sqrt{D}h^{m} + \varepsilon_{5} (\widetilde{\mathcal{M}}) \leq \left( c_{f} + 2 c_{\mathcal{M}}  G \right)\sqrt{D}h^{m} + \varepsilon_{5} (\widetilde{\mathcal{M}}),
\end{eqnarray*}
where $\|\nabla  f(\p)\|\leq G$  is finite and exists since $\mathcal{M}$ is a compact manifold and $\nabla  f(\p)$ is assumed to be Lipschitz
continuous. Eq. \eqref{eq:dist_p_and_x} holds due to Eq. \eqref{eq:distbetween_p_and_r}. Using Remark \ref{rem:v_bound} we get that when $h\to 0$ then $L_{\v} \to 0$, and we can conclude that the bound in Eq. \eqref{eq:true_grad_bound_fixed_step_funknown} goes to $\varepsilon$, and that the left-hand sides of Eq. \eqref{eq:grad_cond_conc_funknown0} and Eq. \eqref{eq:grad_cond_conc_funknown} go to $0$. 
\end{proof}

\begin{thm}[Backtracking MMLS-RO gradient-descent approximating $\nabla \widehat{f}$]\label{thm:conc_backtrack_funknown}
Under Assumptions \ref{assu:MMLS_func_Samples}, \ref{assu:Lipschitz-gradient}, \ref{assu:ret_well_defined}, and \ref{assu:lower_bound}, $\nabla \widehat{f}$ is approximated for the Riemannian gradient approximation (Eq. \eqref{eq:approximate_f_Rgrad}), provided all the iterations are performed on points on $\widetilde{{\mathcal{M}}}$ where $\widetilde{{\mathcal{M}}}$ is a manifold, $h$ is small enough such that
\begin{equation}\label{eq:grad_cond_conc_backtrack_funknown0}
    6 L_{\v} \leq Q,
\end{equation}
and
\begin{equation}\label{eq:grad_cond_conc_backtrack_funknown}
    \max \left\{\frac{6L_{\v}}{\bar{\alpha}}\ ,\  2 \left(2 c_{\mathcal{M}} \left\| \nabla  f(\p_{i}) \right\| + c_{\widetilde{\mathcal{M}}} \left\| \nabla  f(\p_{i}) \right\| + c_{f} + c_{\mmls} \right)\sqrt{D}h^{m} \right\} \leq \left\|\gradtil{f(\x_{i})}\right\|
\end{equation}
holds for all the iterations, where ${\mathcal P}_{m}^{h}(\p_{i}) = \x_{i}$, $\p_{i}\in\mathcal{M}$, $L_{\v}$ is the constant from  Remark \ref{rem:v_bound}, a Riemannian gradient algorithm, i.e., Algorithm \ref{alg:MMLS-for-Riemannian} with the backtracking procedure from Algorithm \ref{alg:MMLS-RO_backtrack} to determine $\alpha_{i}$, with an initial step-size
\begin{equation}\label{eq:conc_step_size_backtrack_funknown}
    \widehat{\alpha}_{i} \coloneqq \min \left\{ \bar{\alpha}\ ,\ \frac{Q}{\left \| \gradtil{f(\x_{i})} \right \|}\right\}>0,
\end{equation}
and parameters $\gamma\in (0,1)$, $\delta\in (0, 0.25)$, and some $\bar{\alpha} > 0$, returns a point $\x\in\widetilde{\mathcal{M}}$ satisfying $f(\x)\leq f(\x_{0})$ and 
\begin{equation}\label{eq:grad_conc_bound_backtrack_funknown}
   \left \|\gradtil{f(\x)} \right \| \leq \max \left\{\frac{6L_{\v}}{\bar{\alpha}}\ ,\  2 \left( 2 c_{\mathcal{M}} \left\| \nabla  f(\p) \right\| + c_{\widetilde{\mathcal{M}}} \left\| \nabla  f(\p) \right\| + c_{f} + c_{\mmls} \right)\sqrt{D}h^{m} \right\} + \varepsilon \coloneqq \varepsilon_{6} (\widetilde{\mathcal{M}}),
\end{equation}
where $\p\in\mathcal{M}$ satisfies ${\mathcal P}_{m}^{h}(\p) = \x$, for any $\varepsilon>0$, provided enough iterations are performed.

In addition, the following bound holds for the exact Riemannian gradient at a point $\p\in\mathcal{M}$ such that ${\mathcal P}_{m}^{h}(\p) = \x$ is the returned point
\begin{equation}\label{eq:true_grad_bound_backtrack_funknown}
    \left\| \gradM{f(\p)}  \right\| \leq \left( c_{f} + 2 c_{\mathcal{M}}  G \right)\sqrt{D}h^{m} + \varepsilon_{6} (\widetilde{\mathcal{M}}),
\end{equation}
where $\|\nabla  f(\p)\|\leq G$ and Eq. \eqref{eq:dist_p_and_x} holds as well. Moreover, when $h\to 0$ then the bound in Eq. \eqref{eq:true_grad_bound_backtrack_funknown} goes to $\varepsilon$, and the left-hand sides of Eq. \eqref{eq:grad_cond_conc_backtrack_funknown0} and Eq. \eqref{eq:grad_cond_conc_backtrack_funknown} go to $0$.

Finally, if 
$$\varepsilon_{6} (\widetilde{\mathcal{M}}) > \frac{Q}{\left ( \min \left\{ \bar{\alpha}\ ,\ \frac{2\gamma \left(0.25 - \delta \right)}{\widetilde{L}} \right\} \right)} , $$ 
where $\widetilde{L}$ is the constant from Lemma \ref{lem:pullback_Lip}, then the bounds in Eq. \eqref{eq:grad_conc_bound_backtrack_funknown} and Eq. \eqref{eq:true_grad_bound_backtrack_funknown} are achieved in at most
\begin{equation}\label{eq:it_num11}
    \left \lceil \frac{(f(\x_0) - f^{\star})}{Q \delta } \cdot \frac{1}{\varepsilon_{6} (\widetilde{\mathcal{M}})} \right \rceil
\end{equation}
iterations. If 
$$\varepsilon_{6} (\widetilde{\mathcal{M}}) \leq \frac{Q}{\left ( \min \left\{ \bar{\alpha}\ ,\ \frac{2\gamma \left(0.25 - \delta \right)}{\widetilde{L}} \right\} \right)} , $$
then the bounds in Eq. \eqref{eq:grad_conc_bound_backtrack_funknown} and Eq. \eqref{eq:true_grad_bound_backtrack_funknown} are achieved in at most
\begin{equation}\label{eq:it_num12}
    \left \lceil \frac{f(\x_0) - f^{\star}}{\delta \min \left\{ \bar{\alpha}\ ,\ \frac{2\gamma \left(0.25 - \delta \right)}{\widetilde{L}} \right\}} \cdot \frac{1}{\varepsilon_{6}(\widetilde{\mathcal{M}})^{2}} \right \rceil
\end{equation}
iterations.
Each iteration, $i$, requires at most 
\begin{equation}\label{eq:calc_num_conc_funknown}
    \max \left\{1\ ,\ 2 + \log_{\gamma^{-1}}{\left(\frac{\widehat{\alpha}_{i} \widetilde{L}}{2\gamma\left(\frac{\beta_{i}}{2} - \delta \right)} \right)} \right\}
\end{equation}
retractions and cost function evaluations (assuming $f(\x_{0})$ and $\gradtil{f(\x_{0})}$ were already computed), where $\beta_{i}$ is defined in Eq. \eqref{eq:main_conc_fixed_step_beta_funknown}.
\end{thm}

\begin{proof}
Assumptions \ref{assu:MMLS_func_Samples}, \ref{assu:Lipschitz-gradient}, and \ref{assu:ret_well_defined}, together with Lemma \ref{lem:our_proj_prop} (Appendix \ref{subsec:background_MMLS}), Lemma \ref{lem:approx_Rim_grad}, Lemma \ref{lem:retraction_2_property}, and Remark \ref{rem:v_bound}, imply that the components in Table \ref{tab:riemannian-approx} (and MMLS) satisfy Assumption \ref{assu:list_of_req}, where Item \ref{item:list_of_req_grad} is replaced with Eq. \eqref{eq:approx_Rim_grad4} and $\rb$ is replaced with $\p\in \mathcal{M}$ such that ${\mathcal P}_{m}^{h}(\p)=\rb$. Lemma \ref{lem:pullback_Lip} implies Assumption \ref{assu:Lipschitz-type-gradient}. Together with Assumption \ref{assu:lower_bound}, provided $h$ is small enough such that Eq. \eqref{eq:grad_cond_conc_backtrack_funknown0} holds and assuming Eq. \eqref{eq:grad_cond_conc_backtrack_funknown} holds for all the iterations, Corollary \ref{cor:main_backtrack_step} can be applied for Algorithm \ref{alg:MMLS-for-Riemannian} with the backtracking procedure from Algorithm \ref{alg:MMLS-RO_backtrack} to determine $\alpha_{i}$, with an initial step-size according to Eq. \eqref{eq:conc_step_size_backtrack_funknown} and parameters $\gamma\in (0,1)$, $\delta\in (0, 0.25)$, and some $\bar{\alpha} > 0$, using the components from Table \ref{tab:riemannian-approx}. Thus, we can conclude that Eq. \eqref{eq:grad_conc_bound_backtrack_funknown} and the bounds on the
number of iterations, Eq. \eqref{eq:it_num11} and Eq. \eqref{eq:it_num12}, hold. 

Finally, to show that Eq. \eqref{eq:true_grad_bound_backtrack_funknown} holds, we use Eq. \eqref{eq:approx_Rim_grad3} from Lemma \ref{lem:approx_Rim_grad} and
\begin{eqnarray*}
 \left\| \gradM{f(\p)}  \right\| \leq \left\| \gradM{f(\p)} - \gradtil{f(\x)} \right\| + \left \|\gradtil{f(\x)} \right \| \leq \\
 \leq \left( c_{f} + 2 c_{\mathcal{M}}  \left\| \nabla  f(\p) \right\| \right)\sqrt{D}h^{m} + \varepsilon_{6} (\widetilde{\mathcal{M}}) \leq \left( c_{f} + 2 c_{\mathcal{M}}  G \right)\sqrt{D}h^{m} + \varepsilon_{6} (\widetilde{\mathcal{M}}),
\end{eqnarray*}
where $\|\nabla  f(\p)\|\leq G$  is finite and exists since $\mathcal{M}$ is a compact manifold and $\nabla  f(\p)$ is assumed to be Lipschitz
continuous. Eq. \eqref{eq:dist_p_and_x} holds due to Eq. \eqref{eq:distbetween_p_and_r}. Using Remark \ref{rem:v_bound} we get that when $h\to 0$ then $L_{\v} \to 0$, and we can conclude that the bound in Eq. \eqref{eq:true_grad_bound_backtrack_funknown} goes to $\varepsilon$, and the left-hand sides of Eq. \eqref{eq:grad_cond_conc_backtrack_funknown0} Eq. \eqref{eq:grad_cond_conc_backtrack_funknown} go to $0$.
\end{proof}

\section{\label{sec:Numerical-experiments}Numerical Experiments}
In this section we discuss some implementation details and present our experimental results.  

\subsection{\label{subsec:implement_details}Practical Implementation Details}
In this subsection, we point out some practical implementation details for MMLS-RO, specifically for algorithms \ref{alg:MMLS-for-Riemannian}, \ref{alg:MMLS-RO_backtrack}, and \ref{alg:MMLS-RO_CG} (for implementation details of MMLS algorithm see \cite[Section 3.2]{sober2020manifold}). First, the main difficulty in applying the aforementioned algorithms is keeping the step-size such that it is smaller than $Q$ (Assumption \ref{assu:ret_well_defined}), ensuring that the approximate-retraction is applied on a point in $U_{\mathrm{unique}}$. Unfortunately, in practice the set $U_{\mathrm{unique}}$ is unknown in general. But, typically $U_{\mathrm{unique}}$ depends on the reach of the manifold $\mathcal{M}$ which can be estimated (e.g., \cite{aamari2019estimating} where upper bounds on the reach based on samples of $\mathcal{M}$ are proposed). Recall from Assumption \ref{assu:ret_well_defined}, that the goal of this step-size limitation is to ensure that the approximate-retraction is defined in a compact subset of the approximate-tangent bundle. Even though we cannot ensure this step-size limitation exactly, heuristically for $h$ small enough, taking $Q = O(h)$ and $\mu \approx \reach{\mathcal{M}}/2$ (where $\reach{\mathcal{M}}$ is estimated, and also allowing us to heuristically fulfill Constraint \ref{item:MMLS_step1_2} from Problem \eqref{eq:MMLS_step1}) would approximately satisfy the step-size limitation. To achieve that goal, we implemented in our experiments a line-search procedure until MMLS algorithm successfully returns an output. Explicitly, Algorithm \ref{alg:MMLS-RO_backtrack} is implemented with an initial step-size that respects the distance imposed by the weight functions, $\theta_{j}(\cdot)$, $j=1,2,3$ in problems \eqref{eq:MMLS_step1}, \eqref{eq:MMLS_step2}, and \eqref{eq:step2-func_approx}, giving us an estimation of $O(h)$ (see the implementation of the functions \emph{calculateSigma} and \emph{calculateSigmaFromPoint} in \href{https://github.com/aizeny/manapprox/blob/main/manapprox/ManApprox.py}{MMLS}), and the step size is reduced by a factor $\gamma\in (0,1)$ until MMLS is successful. 
Another practical issue is that in many problems, the intrinsic dimension of the constraining manifold $\mathcal{M}$ is unknown, and has to be estimated. One can use statistical methods for estimating the intrinsic dimension of a manifold $\mathcal{M}$  based on its samples, e.g., \cite{camastra2002estimating, kegl2002intrinsic, costa2004geodesic, levina2004maximum}. 

With regards to the components presented in Table \ref{tab:riemannian-approx}, we propose the following alternative: replacing the tangent space estimation, $\widetilde{T}_{\rb}\widetilde{M}$, by $H(\rb)$. On the one hand, this choice saves computational time in finding the orthogonal projection on the approximate-tangent space (and its related components, e.g., approximate-Riemannian gradient), since the first step of MMLS provides an orthogonal basis of $H(\rb)$. On the other hand, to the best of our knowledge, no theoretical analysis of the approximation order of the tangent spaces by $H(\cdot)$ is known. In our experiments, we found that using these alternative components still allowed our proposed algorithms to converge in most cases, though in general the results were inferior to the results we achieved using the components in Table \ref{tab:riemannian-approx}.

Another issue related to the numerical stability of computing the components in Table \ref{tab:riemannian-approx}, is inverting Gram matrices as required for example in Eq. \eqref{eq:f_Rgrad}. We performed the Cholesky decomposition on the Gram matrix (alternatively, performed a QR factorization of the pre-multiplied matrix) prior to inverting it, and then solved the corresponding linear equations. 

Finally, we note two relaxations for our assumptions which could widen the possible applications of our algorithms, First, even though we assume clean samples of the cost function and the constraining manifold, in practice oftentimes samples are noisy. As we previously mention in Remark \ref{rem:clean_samples}, MMLS algorithm works for noisy sample sets, but some of the theory is still incomplete. Thus, in some experiments we relax the clean samples requirement, and present experiments with noisy samples as well. Another assumption we can relax, is that the constraints in the optimization problem (Eq. \eqref{eq:general_problem}) define some manifold globally. Instead, it is enough that these constraints define a manifold locally, possibly a different manifold at each neighborhood, such that MMLS is still able to produce an approximation locally.    

\subsection{\label{subsec:experiments}Experiments}
In this subsection, we present our experiments, which demonstrate the effectiveness of our proposed  components in Table \ref{tab:riemannian-approx} and Algorithms \ref{alg:MMLS-for-Riemannian}, \ref{alg:MMLS-RO_backtrack}, and \ref{alg:MMLS-RO_CG}, for both the scenarios where $f$ and its Euclidean gradient are explicitly known (first-order optimization with respect to the cost, zeroth-order with respect to the constraint, labeled by \emph{MMLS\_RO} in the figures via a blue line), and where $f$ is only accessed via samples (Assumption \ref{assu:MMLS_func_Samples}, zeroth-order optimization, labeled by \emph{ZO\_MMLS\_RO} in the figures via a green line). MMLS algorithms implementations for approximating manifolds and approximating functions on manifolds are based on the implementations in \url{https://github.com/aizeny/manapprox/tree/main/manapprox}. In particular, we use the parametric family of weight functions $\theta(k\ ;\ \cdot)$ implemented there, also used in the experiments in \cite[Section 4]{sober2021approximation}, for problems \eqref{eq:MMLS_step1}, \eqref{eq:MMLS_step2}, and \eqref{eq:step2-func_approx}. Explicitly,
\begin{equation*}
    \theta(k\ ;\ t) \coloneqq
    \begin{cases}
    e^{\frac{-t^2}{(t-kh)^{2}}} \cdot \chi_{kh}, & t\neq kh,\\
    0, & t=kh,
    \end{cases}
\end{equation*}
where $k$ is some parameter with the default value $k = 1.5$, $h$ is the fill distance, and $\chi_{kh}$ is an indicator function on the interval $[-kh, kh]$. Note that $\theta(k\ ;\ \cdot)$ is compactly supported and $C^{\infty}$. As explained in \cite[Section 4]{sober2021approximation}, the support size is chosen such that the local least-squares
matrix would be invertible.
Our experiments are performed both via our own implementations of the components in Table \ref{tab:riemannian-approx} and Algorithms \ref{alg:MMLS-for-Riemannian} and \ref{alg:MMLS-RO_backtrack}, and uses the framework of \textsc{PYMANOPT} \cite{townsend2016pymanopt}. 

Specifically, we implemented manifold classes which produce the components in Table \ref{tab:riemannian-approx}, in addition to approximations of $f$ in the case of zeroth-order optimization. Algorithms \ref{alg:MMLS-for-Riemannian} and \ref{alg:MMLS-RO_CG} are based on their corresponding implementations of Riemannain gradient-descent and Riemannian CG in \textsc{PYMANOPT}, with the following exceptions: the initial point is first projected on $\widetilde{\mathcal{M}}$ via an MMLS projection; the line-search procedure of \textsc{PYMANOPT} is modified based on Algorithm \ref{alg:MMLS-RO_backtrack}, such that the initial step-size respects the distance imposed by the weight function and the constraint on it to be smaller than $Q$ which is satisfied by manually searching the point at which MMLS algorithm is successful as explained in Subsection \ref{subsec:implement_details}. 

In the graphs, Riemannian gradient-descent is labeled by \emph{GD}, and Riemannian CG is labeled by \emph{CG}. The iterations based on noisy samples of the manifold (\emph{MMLS\_RO}) and noisy samples of the cost function (\emph{ZO\_MMLS\_RO}) are labeled by \emph{NGD} and \emph{NCG} for Riemannian gradient-descent and Riemannian CG correspondingly. In our experiments, we modeled noise in the samples according to the additive model presented in \cite{sober2020manifold,sober2021approximation}. Thus, in our experiments noise was added to the manifold samples, and independently, noise was added to the cost function samples after the clean samples of the manifold were given as an input to it.  

\subsubsection{\label{subsubsec:prelim_experiments}Preliminary Experiment}
In the first experiment, we demonstrate Algorithm \ref{alg:MMLS-for-Riemannian} with backtracking (Algorithm \ref{alg:MMLS-RO_backtrack}) using our own implementation on a zeroth-order optimization for solving the following problem:
\begin{equation}\label{eq:exp_prelim}
    \min_{\x\in \mathcal{M}} \sin{(2\pi x_{1})} + 4x_{2}^{2} + x_{1},
\end{equation}
where 
\begin{equation*}
    \mathcal{M} \coloneqq \left\{\x \in \R^{100}\ |\ x_{1},x_{2}\in \R,\ x_{3} = \sin{(2\pi (x_{1}^{2} + x_{2}^{2}))},\ x_{i} = 1\ \forall 4\leq i \leq 100 \right\}.
\end{equation*}
The ambient dimension is $D=100$ and the intrinsic dimension is $d=2$. We use $n=50000$ samples of $\mathcal{M}$, and polynomial approximations of degree $m=1$. The results are presented in Fig. \ref{fig:pre_zeroorder}, both for the case of clean samples and for the case of noisy samples. The noise is an additive Gaussian noise $\mathcal{N}(0, 10^{-3})$, added to the coordinates $x_{1}, x_{2}, x_{3}$ of the manifold samples, and to the cost function samples as well. The iterations (red) are illustrated both in the ambient space, where the manifold is illustrated with respect to $x_{1}, x_{2}, x_{3}$, and in the parametric space, where cost function values are plotted as a function of $\x_{1}$ and $\x_{2}$. Darker colors represent lower cost function values. 

For the clean samples, the algorithm terminated after $89$ iterations, starting from an initial approximated cost value of $2.27$ and an approximate-Riemannian gradient norm of $0.795$, and finishing with an approximated cost value of $-1.25$ and an approximate-Riemannian gradient norm of $0.004$. For the noisy samples, the algorithm terminated after $21$ iterations, starting from an initial approximated cost value of $2.24$ and an approximate-Riemannian gradient norm of $0.718$, and finishing with an approximated cost value of $-1.22$ and an approximate-Riemannian gradient norm of $0.004$.

\begin{figure}[tb]
\begin{centering}
\begin{tabular}{ccc}
\centering{}\includegraphics[scale=0.45]{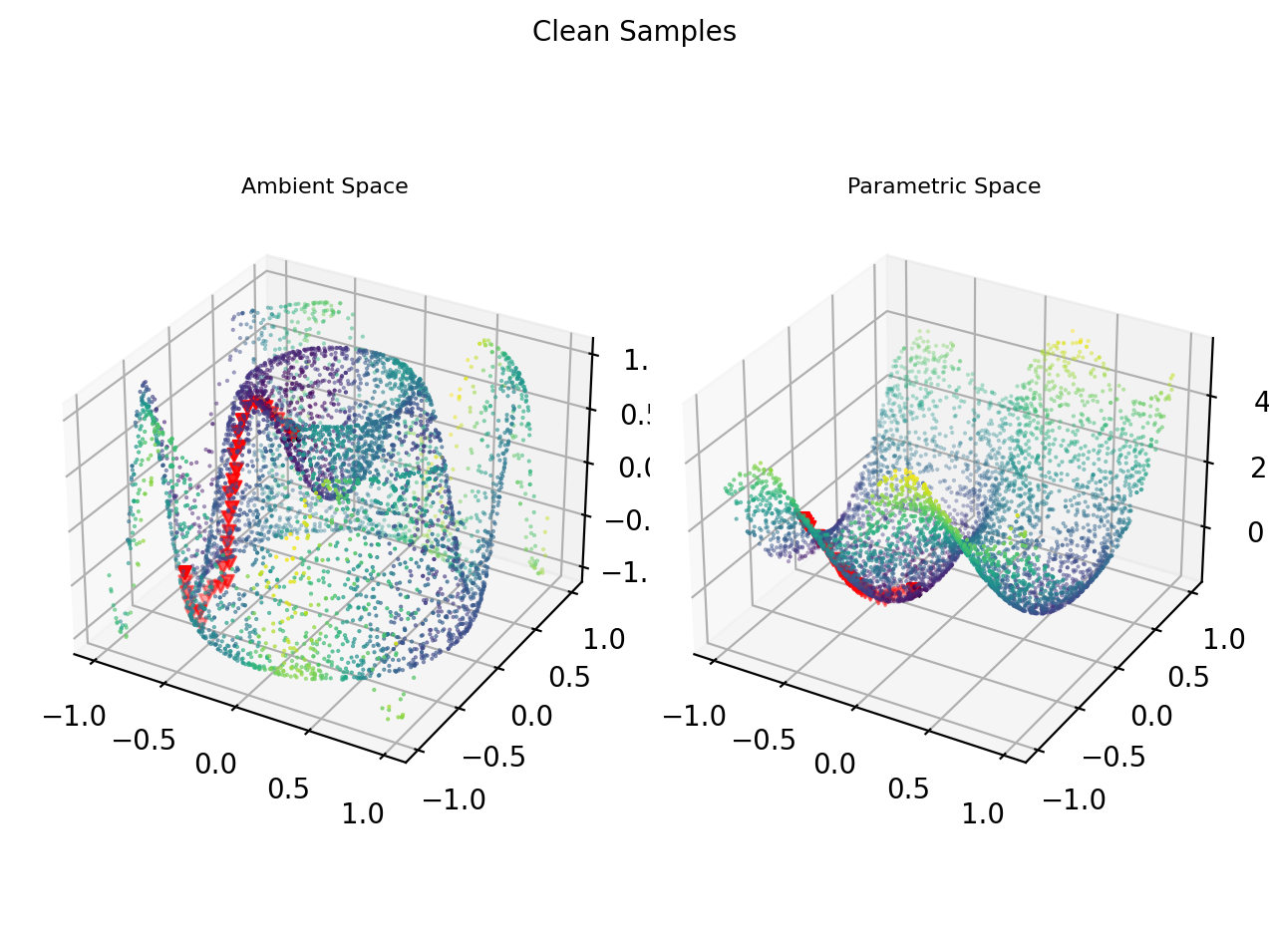} & ~ & \includegraphics[scale=0.45]{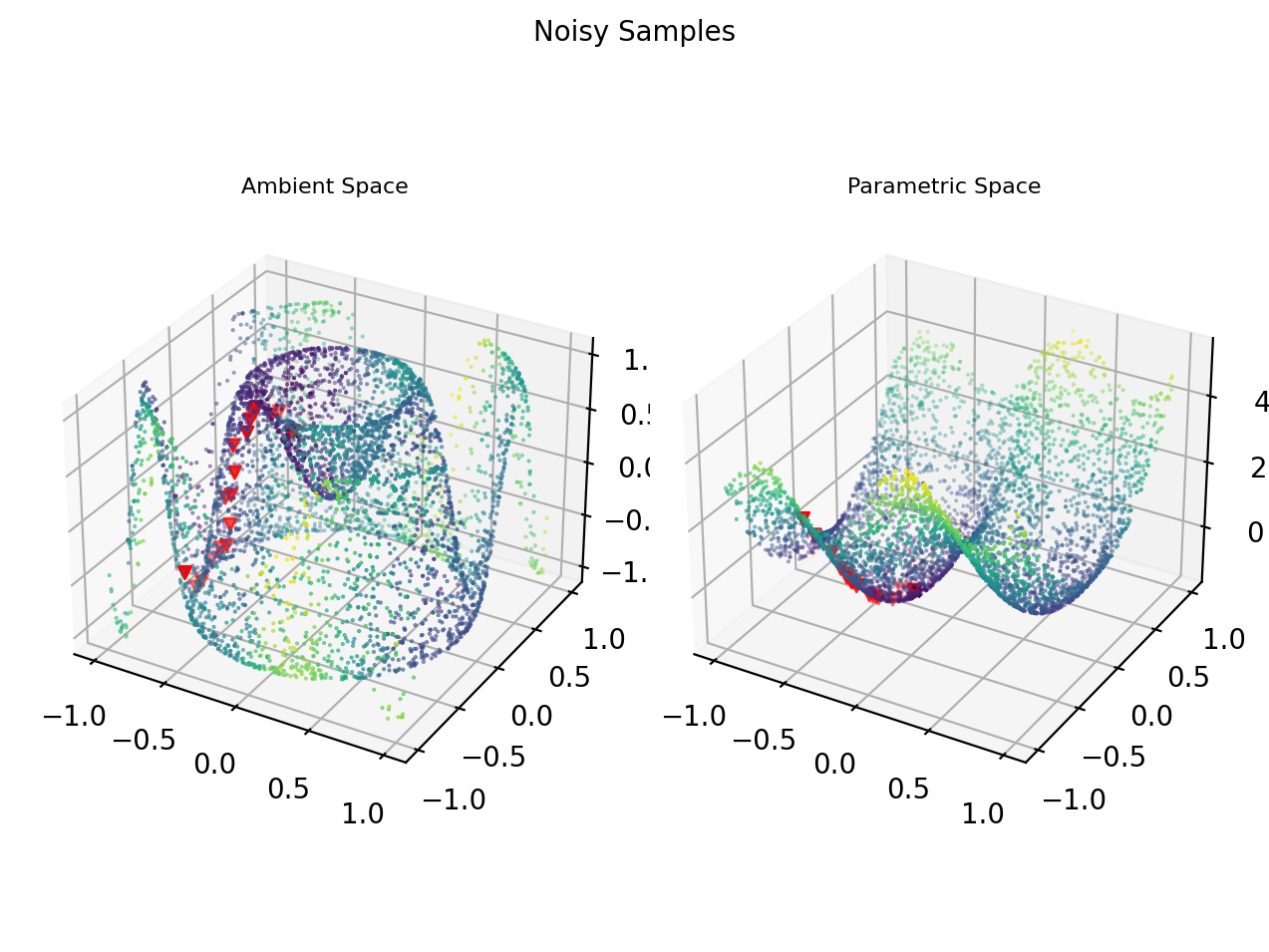}
\end{tabular}
\par\end{centering}
\caption{Iterations (red) of Algorithm \ref{alg:MMLS-for-Riemannian} with backtracking (Algorithm \ref{alg:MMLS-RO_backtrack}) for Problem \eqref{eq:exp_prelim} in the ambient space and the parametric space with noise (right) and without noise (left).\label{fig:pre_zeroorder}}
\end{figure}

\subsubsection{\label{subsubsec:manopt_experiments}Experiments with Matrix Manifolds}
In our main set of experiments we tested our algorithm on a few eigenvalue problems, a principal component analysis (PCA) problem, and a low-rank approximation of a matrix on the following manifolds: the sphere in $\R^{3}$, two Stiefel manifolds, and a fixed-rank manifold. We performed our experiments both with clean samples, and noisy samples of the constraining manifold and the cost function. We compered our results to the results obtained from \textsc{PYMANOPT} implementations of Riemannian gradient-descent and Riemannian CG for each of the problems and their corresponding manifolds, where all the information regarding the cost functions and the constraining manifolds is fully available to the solver. We label these results in the figures by \emph{Pymanopt} via a red line. To generate samples of each of the manifolds, we used the sampling method implemented for each of the manifolds in \textsc{PYMANOPT}. The initial point was chosen at random from the sampling set. We set the following stopping criteria (reaching one of them would stop the iterations) in addition to the default stopping in \textsc{PYMANOPT}:
\begin{itemize}
    \item A (approximate-)Riemannian gradient norm smaller than $0.005$.
    \item A step-size smaller than $10^{-10}$.
    \item Maximal number of $1000$ iterations.
\end{itemize}


The first problem we tackle is finding the top eigenvalue of a randomly generated SPD matrix, $\matA\in \R^{3\times 3}$, thus the constraining manifold is the sphere in $\R^{3}$. Explicitly,
\begin{equation}\label{eq:eig_sphere3D}
    \min_{\x\in \mathcal{S}^{3}} -\x^{\T} \matA \x,\ \matA \coloneqq \left(\begin{array}{ccc}
1.64 & 0.9 & 0.71\\
0.9 & 0.82 & 0.33\\
0.71 & 0.33 & 0.7
\end{array}\right),\ \mathcal{S}^{3} \coloneqq \{\x\in \R^{3}\ |\ \|\x\| = 1 \}.
\end{equation}
The ambient dimension is $D=3$ and the intrinsic dimension is $d=2$. We use $n=40000$ samples of $\mathcal{S}^{3}$, and polynomial approximations of degree $m=3$. The noise is an additive Gaussian noise $\mathcal{N}(0, 10^{-4})$, added to both the samples of $\mathcal{S}^{3}$, and the cost function samples in Problem \eqref{eq:eig_sphere3D}. The results are presented in Fig. \ref{fig:sphere3D}. The left plots, present suboplimality, i.e., $|\lambda_{1}(\matA) - \x^{\T} \matA \x| / \lambda_{1}(\matA)$ where $\lambda_{1}(\matA)$ is the largest eigenvalue of $\matA$, versus iteration count at the top plot, and versus time at the bottom plot. The right plots present the (approximate-)Riemannian gradient norms versus iteration count at the top plot, and versus time at the bottom plot.

\begin{figure}[tb]
\begin{centering}
\begin{tabular}{ccc}
\centering{}\includegraphics[scale=0.45]{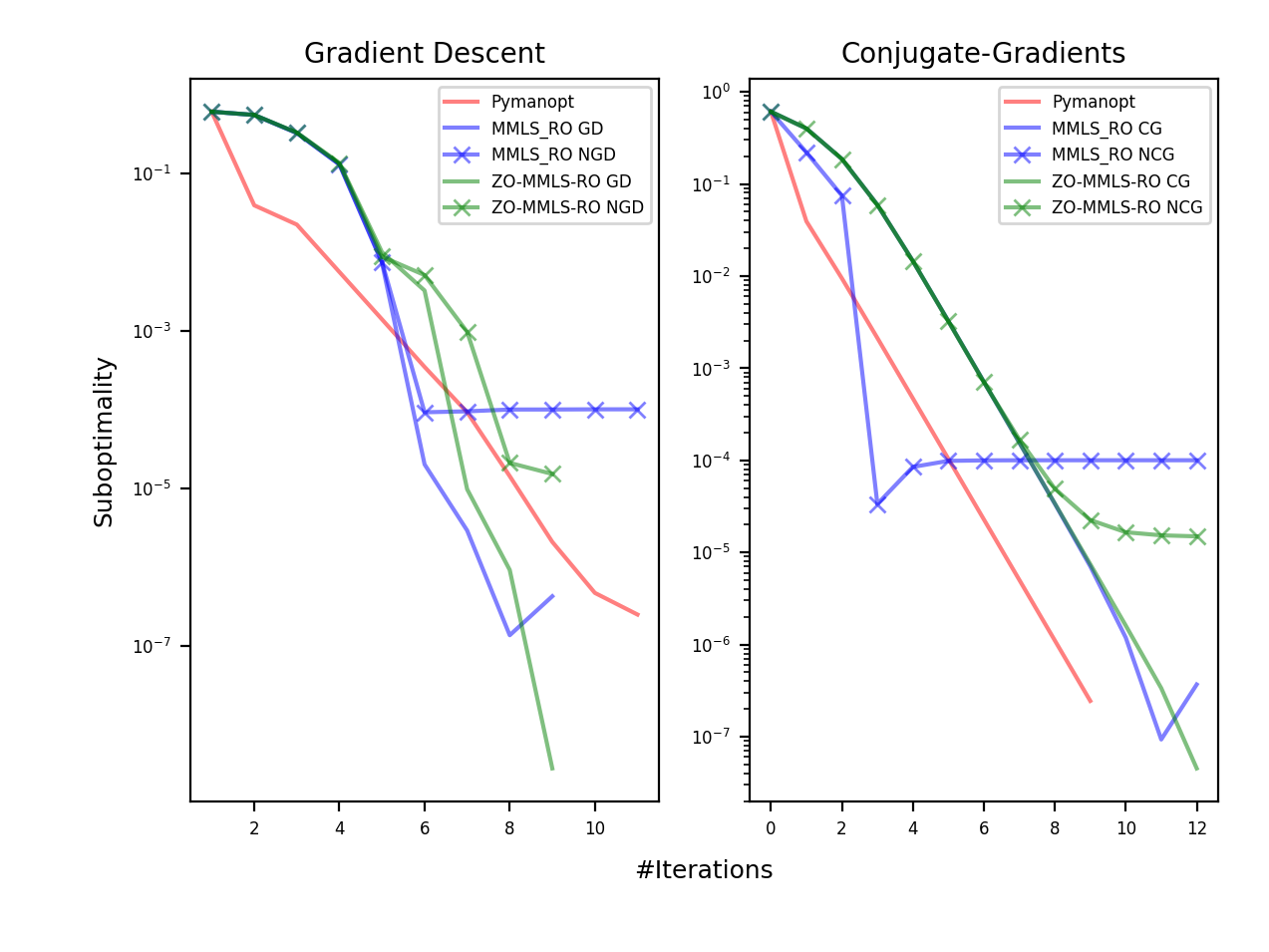} & ~ & \includegraphics[scale=0.45]{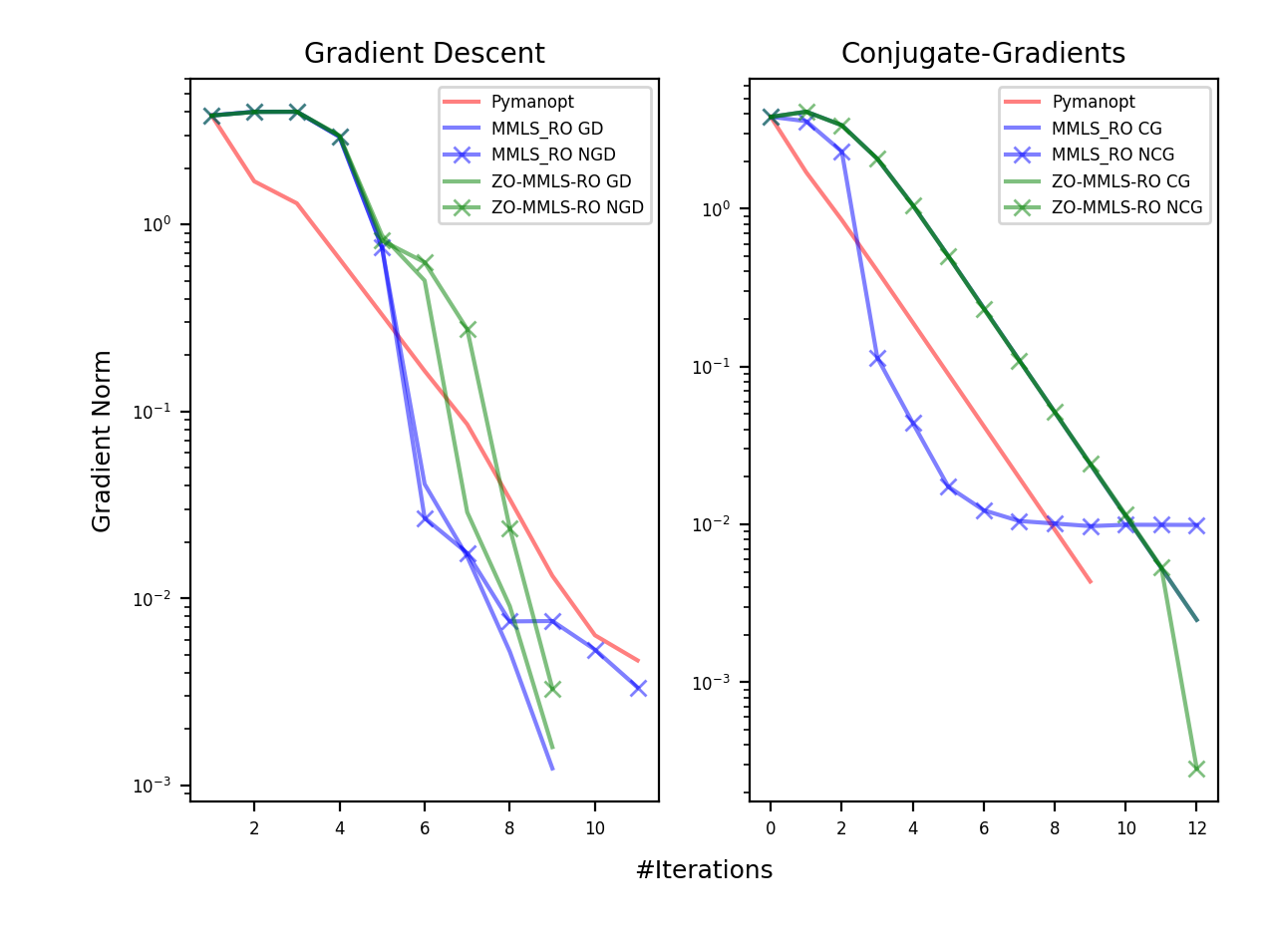} \\ \centering{}\includegraphics[scale=0.45]{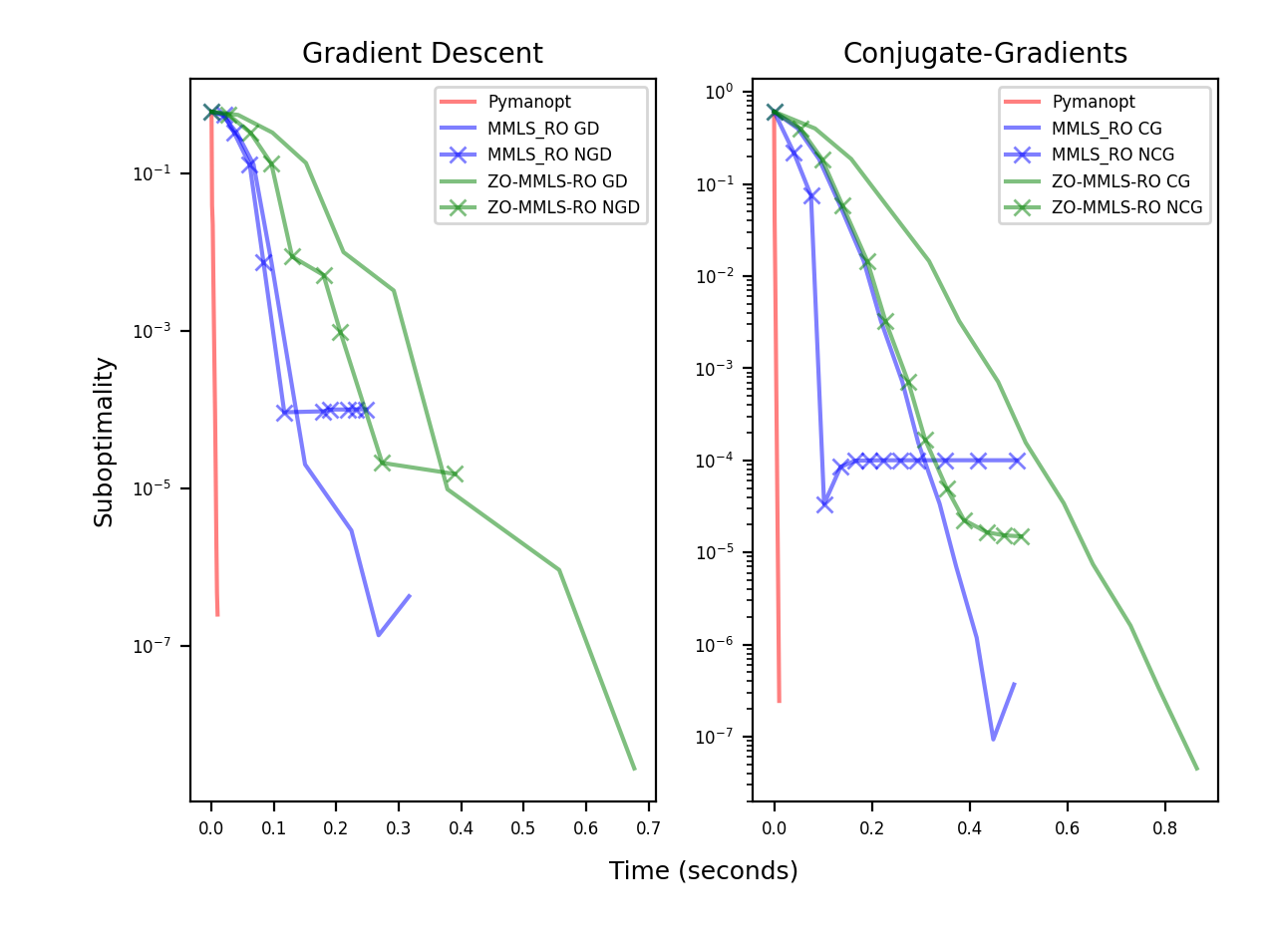} & ~ & \includegraphics[scale=0.45]{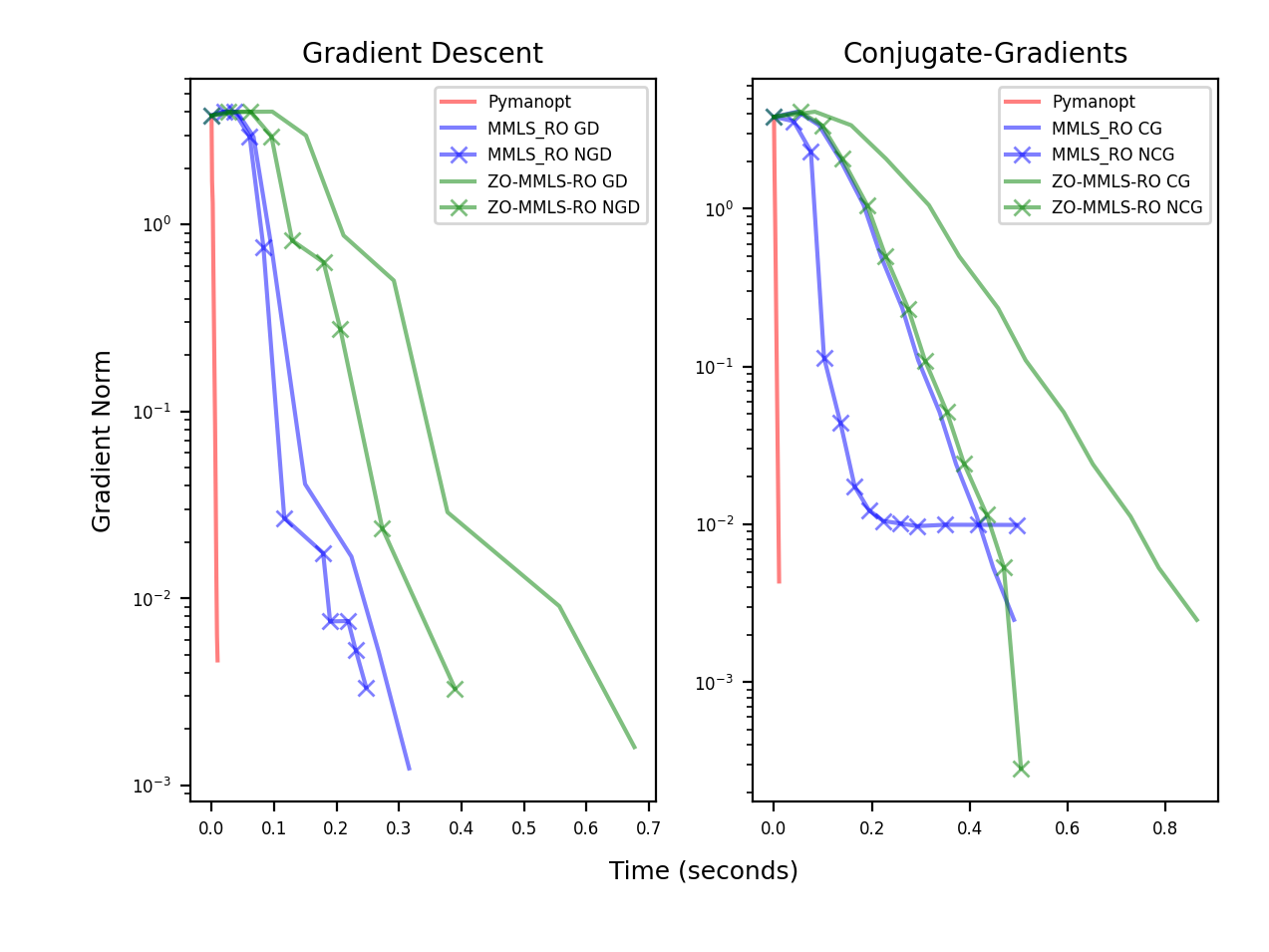}
\end{tabular}
\par\end{centering}
\caption{Suboptimality versus iteration count (top left) and versus time (bottom left), (Approximate-)Riemannian gradient norm versus iteration count (top right) and versus time (bottom right), for Problem \eqref{eq:eig_sphere3D}.\label{fig:sphere3D}}
\end{figure}

We solve two additional similar problems to Problem \eqref{eq:eig_sphere3D}, of finding the two largest eigenvalues of two randomly generated matrices, $\matA_{1}\in \R^{3\times 3}$ and $\matA_{2}\in \R^{4\times 4}$, thus the constraining manifolds are Stiefel manifolds. Explicitly,
\begin{equation}\label{eq:eig_Stiefel3D}
    \min_{\matX\in \stiefel(3,2)} -\Trace{\matX^{\T} \matA_{1} \matX},\ \matA_{1} \coloneqq
    \left(\begin{array}{ccc}
0.23 & 0.35 & 0.39\\
0.35 & 1.33 & 1.06\\
0.39 & 1.06 & 1.27
\end{array}\right),\ \stiefel(3,2) \coloneqq \{\matX\in \R^{3 \times 2}\ |\ \matX^{\T} \matX = \matI_{2} \},
\end{equation}
and 
\begin{equation}\label{eq:eig_Stiefel4D}
      \min_{\matX\in \stiefel(4,2)} -\Trace{\matX^{\T} \matA_{2} \matX},\ \matA_{2} \coloneqq
    \left(\begin{array}{cccc}
2.77 & 2.4 & 1.49 & 2.15\\
2.4 & 2.66 & 1.18 & 2.12\\
1.49 & 1.18 & 1.51 & 1.92\\
2.15 & 2.12 & 1.92 & 3.13
\end{array}\right),\ \stiefel(4,2) \coloneqq \{\matX\in \R^{4 \times 2}\ |\ \matX^{\T} \matX = \matI_{2} \}. 
\end{equation}
Note that since we are only interested in the two largest eigenvalues, we formulate Problem \eqref{eq:eig_Stiefel3D} and Problem \eqref{eq:eig_Stiefel4D} with cost functions designed to find a $2$-dimensional leading eigenspace. If we were also interested in the eigenvectors themselves, then the Brockett cost function \cite{brockett1991dynamical} should have replaced the current cost functions. 

In order to apply our algorithms, we flatten each of the matrices sampled from $\stiefel(3,2)$ and $\stiefel(4,2)$ to column-stack vectors in $\R^{6}$ and $\R^{8}$ correspondingly, while the inputs to the cost function were reshaped back into the corresponding matrix form. For Problem \eqref{eq:eig_Stiefel3D}, the ambient dimension is $D=6$ and the intrinsic dimension is $d=3$. We use $n=42875$ samples of $\stiefel(3,2)$, and polynomial approximations of degree $m=3$. The noise is an additive Gaussian noise $\mathcal{N}(0, 10^{-3})$, added to both the samples of $\stiefel(3,2)$, and the cost function samples in Problem \eqref{eq:eig_Stiefel3D}. For Problem \eqref{eq:eig_Stiefel4D}, the ambient dimension is $D=8$ and the intrinsic dimension is $d=5$. We use $n=100000$ samples of $\stiefel(4,2)$, and polynomial approximations of degree $m=4$. The noise is an additive Gaussian noise $\mathcal{N}(0, 10^{-4})$, added to both the samples of $\stiefel(4,2)$, and the cost function samples in Problem \eqref{eq:eig_Stiefel4D}. The results for Problem \eqref{eq:eig_Stiefel3D} are presented in Fig. \ref{fig:Stiefel3D}, and the results for Problem \eqref{eq:eig_Stiefel4D} are presented in Fig. \ref{fig:Stiefel4D}. In Fig. \ref{fig:Stiefel3D} and Fig. \ref{fig:Stiefel4D}, the left plots present suboplimality, i.e., $|(\lambda_{1}(\matA_{i}) + \lambda_{2}(\matA_{i})) - \Trace{\matX^{\T} \matA_{i} \matX}| / (\lambda_{1}(\matA_{i}) + \lambda_{2}(\matA_{i}))$ where $\lambda_{1}(\matA_{i})$ and $\lambda_{2}(\matA_{i})$ are the two largest eigenvalue of $\matA_{i}$, and $i=1,2$, versus iteration count (top) and time (bottom). The right plots in Fig. \ref{fig:Stiefel3D} and Fig. \ref{fig:Stiefel4D}, present (approximate-)Riemannian gradient norms versus iteration count (top) and versus time (bottom).

\begin{figure}[tb]
\begin{centering}
\begin{tabular}{ccc}
\centering{}\includegraphics[scale=0.45]{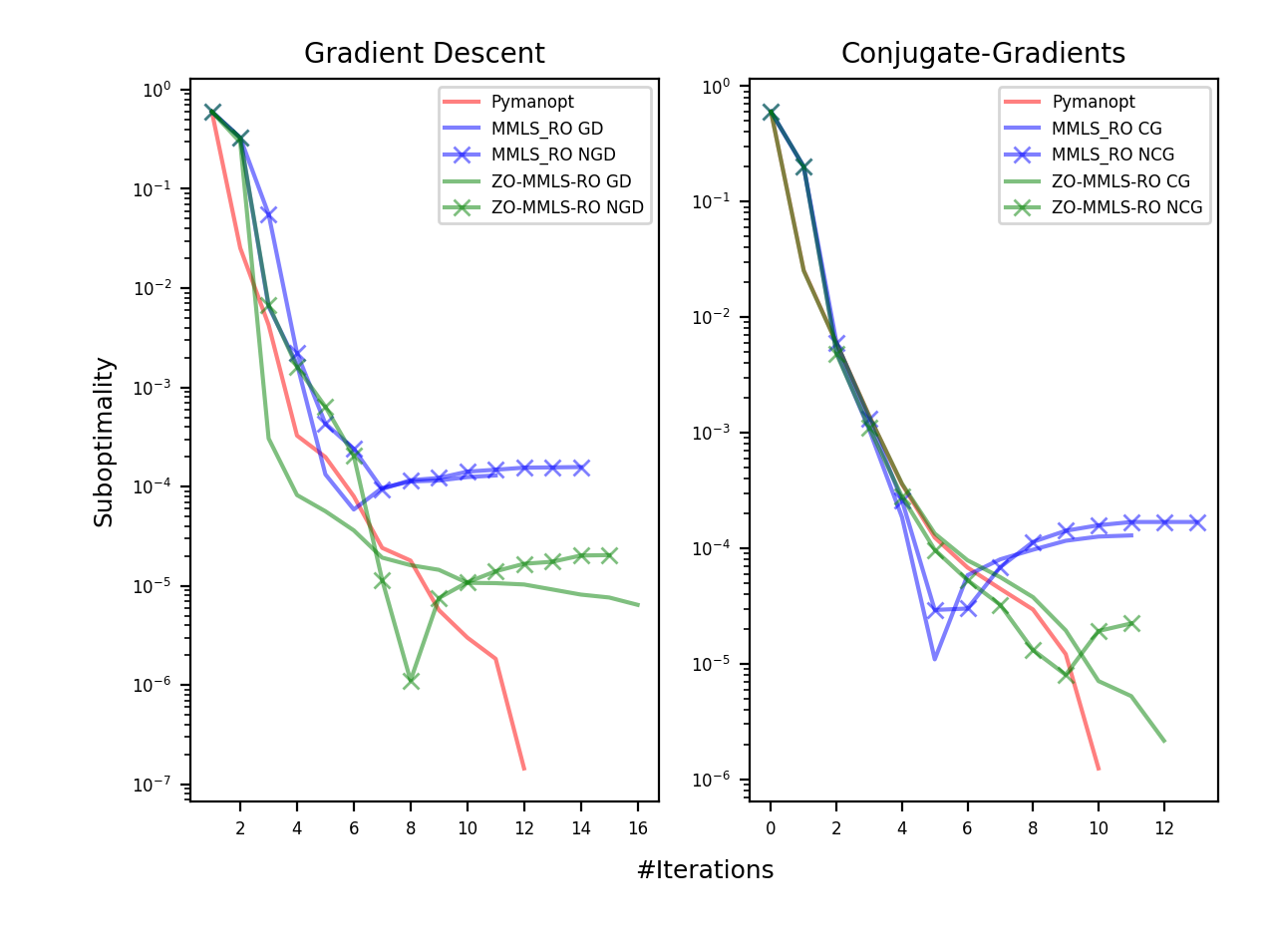} & ~ & \includegraphics[scale=0.45]{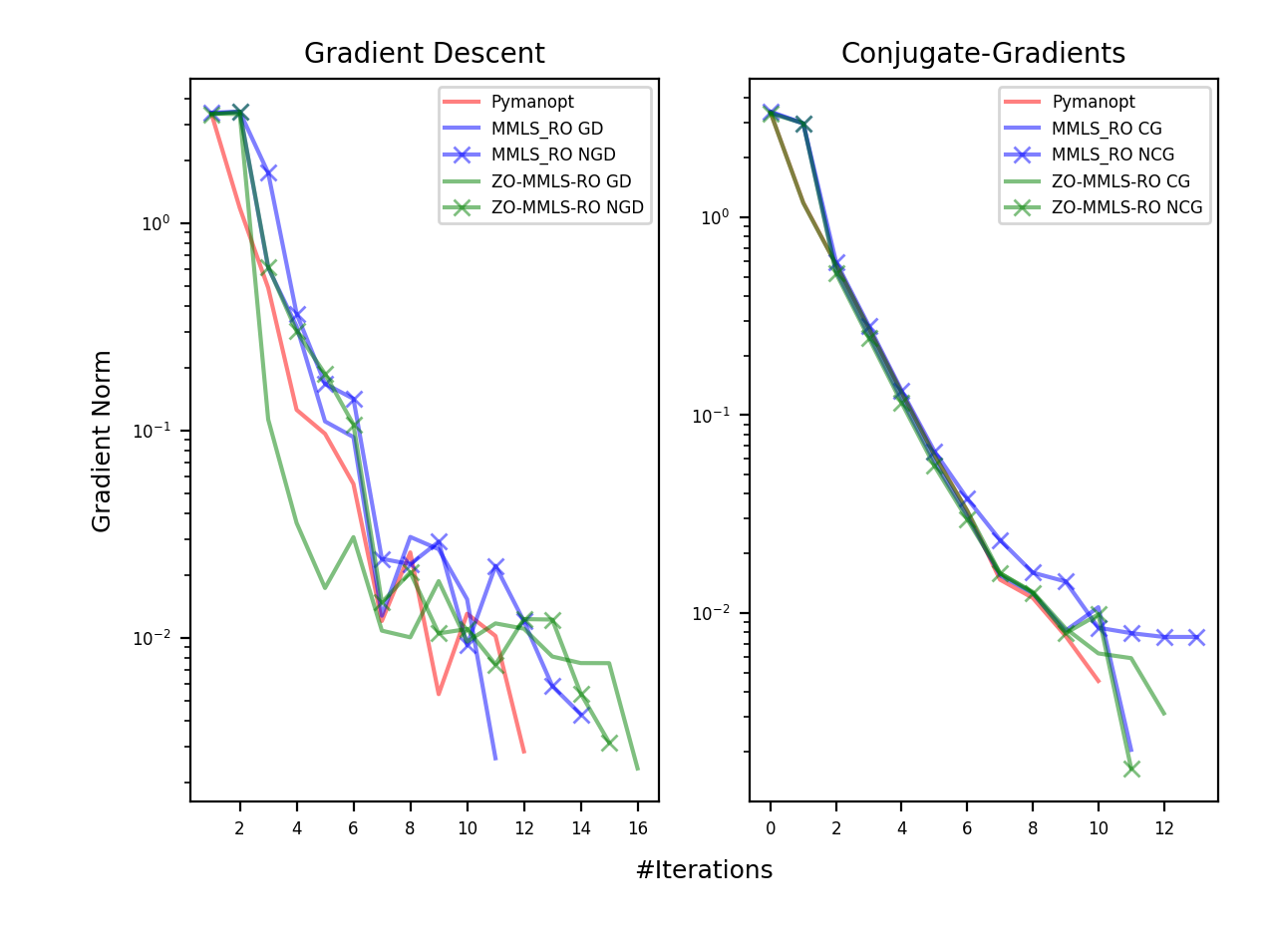} \\ \centering{}\includegraphics[scale=0.45]{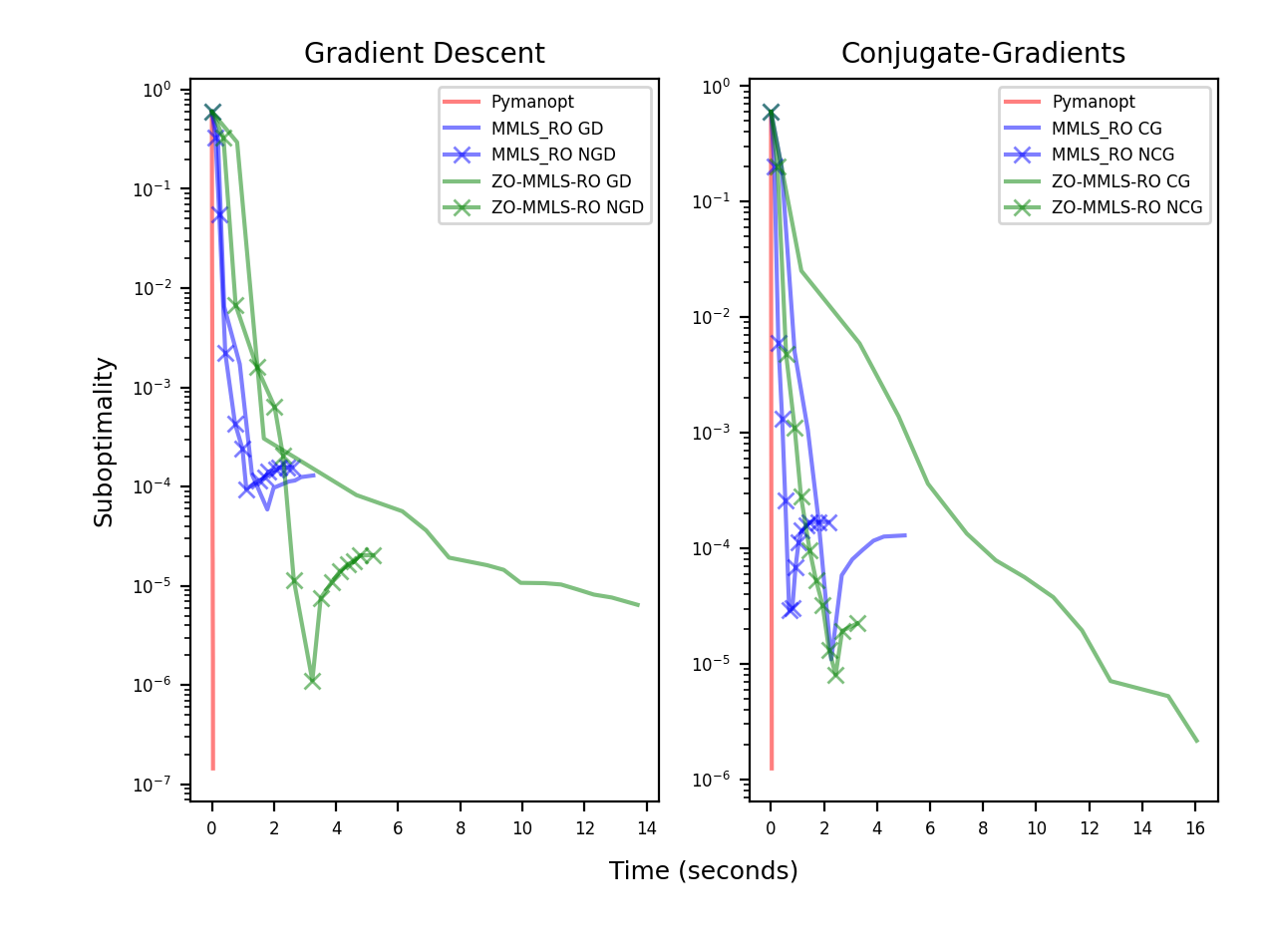} & ~ & \includegraphics[scale=0.45]{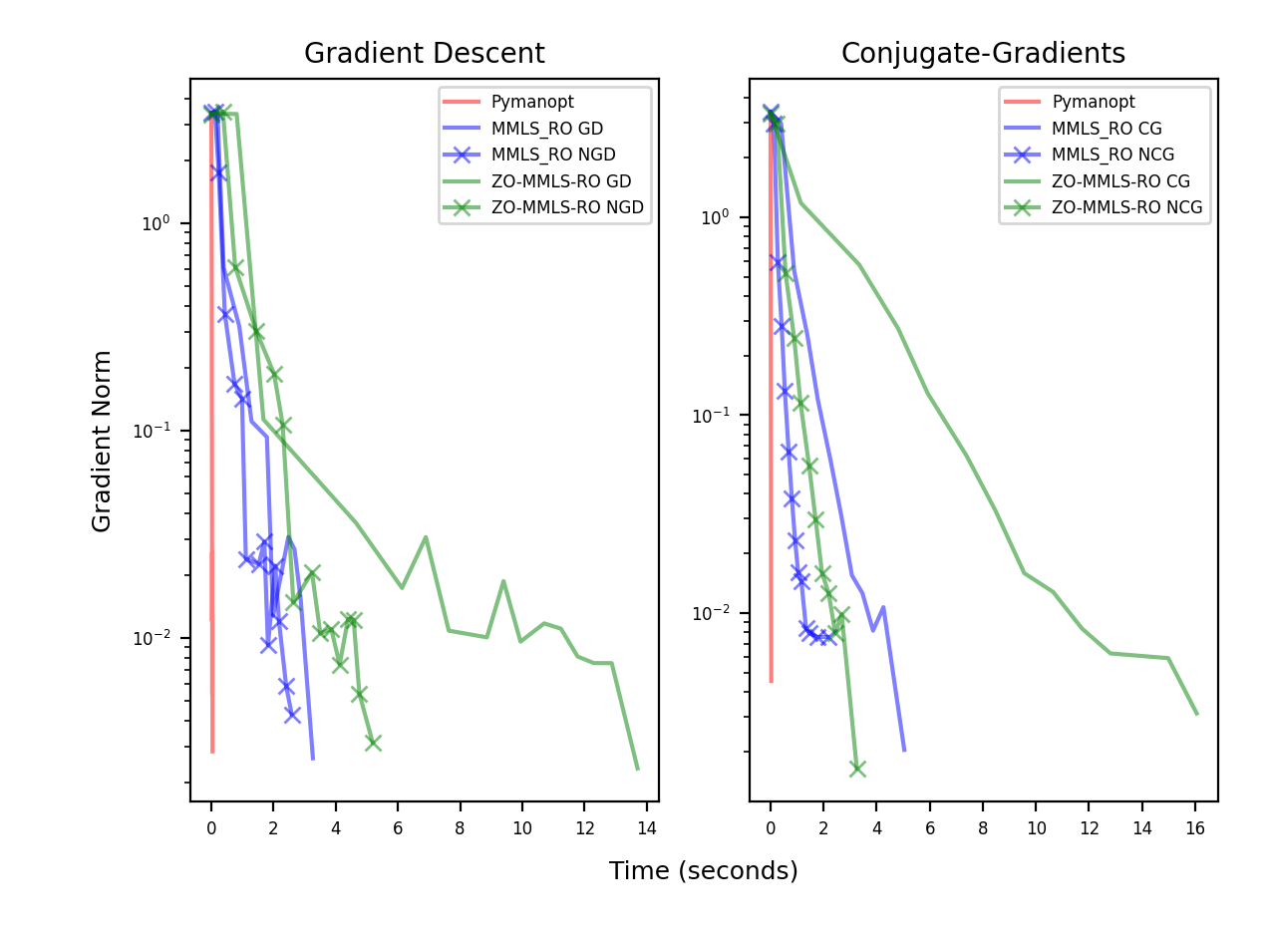}
\end{tabular}
\par\end{centering}
\caption{Suboptimality versus iteration count (top left) and versus time (bottom left), (Approximate-)Riemannian gradient norm versus iteration count (top right) and versus time (bottom right), for Problem \eqref{eq:eig_Stiefel3D}.\label{fig:Stiefel3D}}
\end{figure}

\begin{figure}[tb]
\begin{centering}
\begin{tabular}{ccc}
\centering{}\includegraphics[scale=0.45]{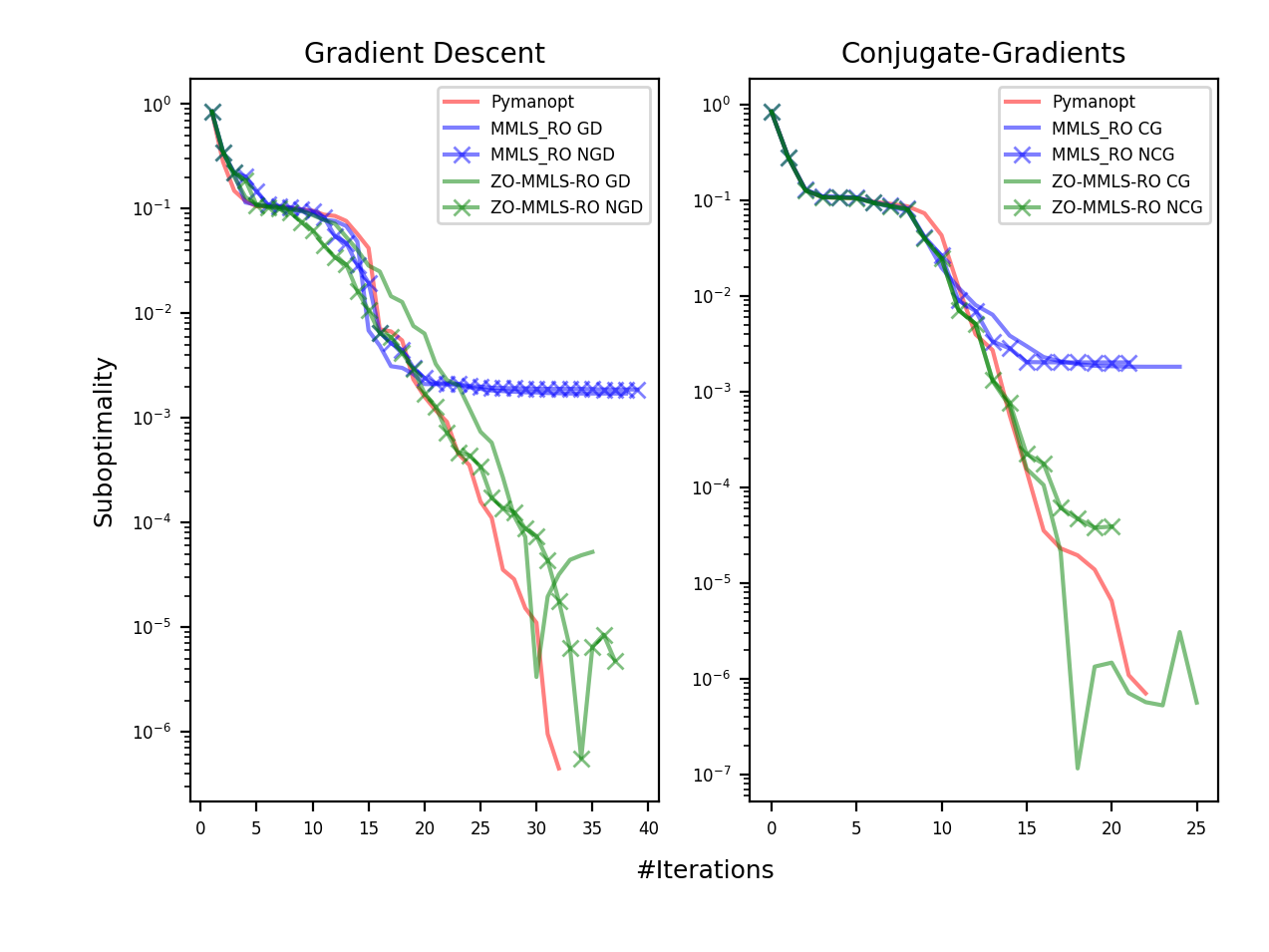} & ~ & \includegraphics[scale=0.45]{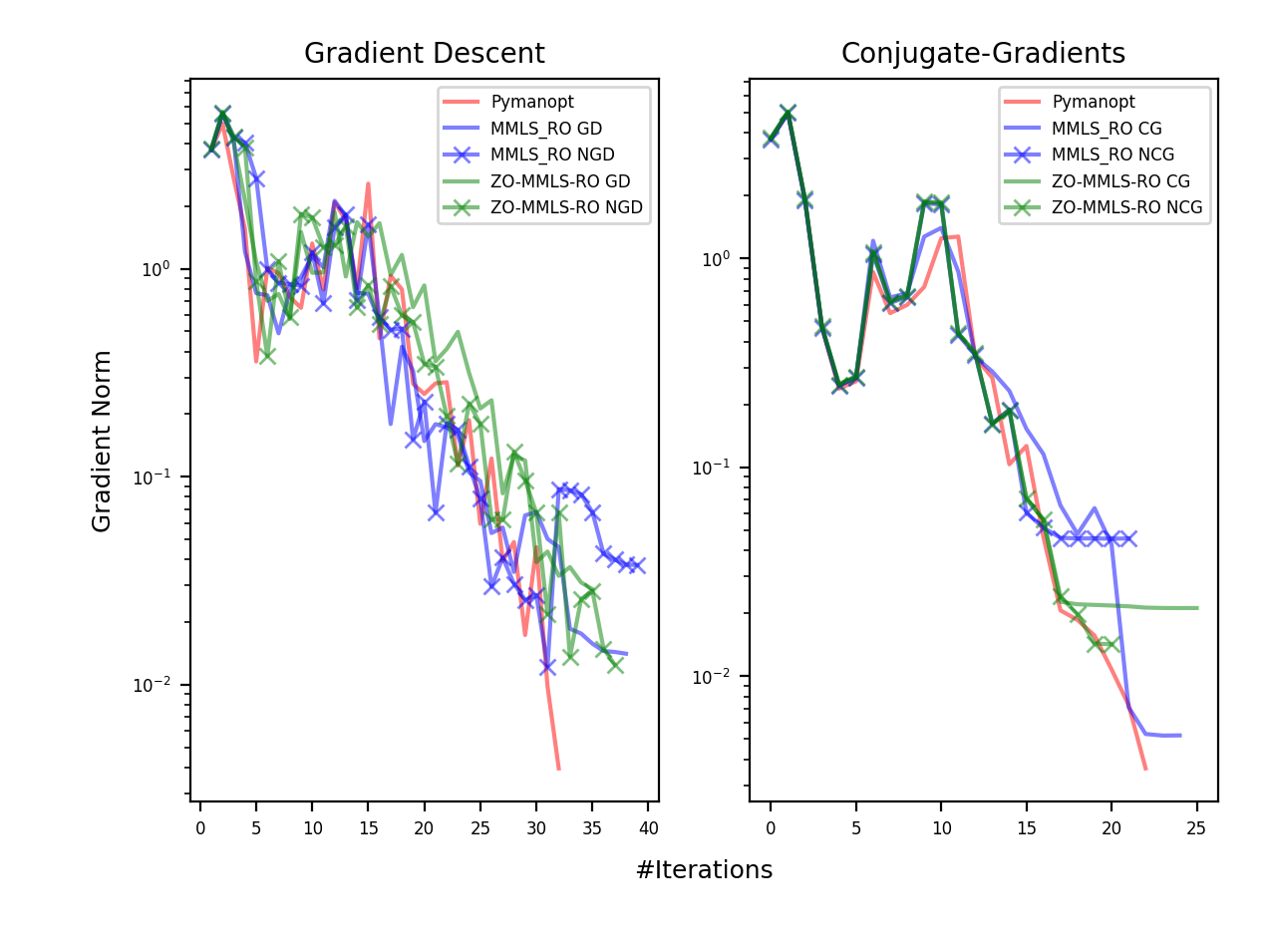} \\ \centering{}\includegraphics[scale=0.45]{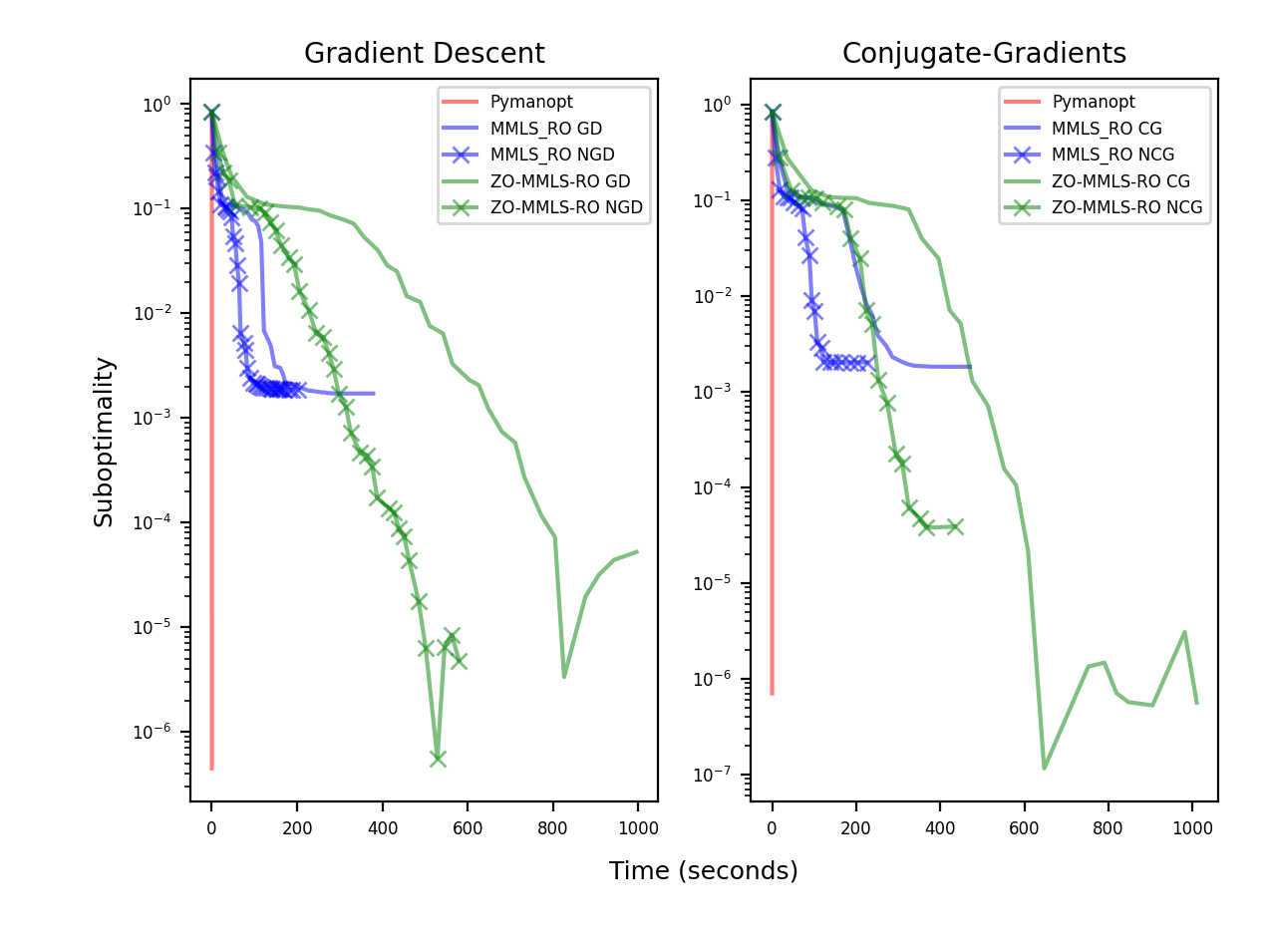} & ~ & \includegraphics[scale=0.45]{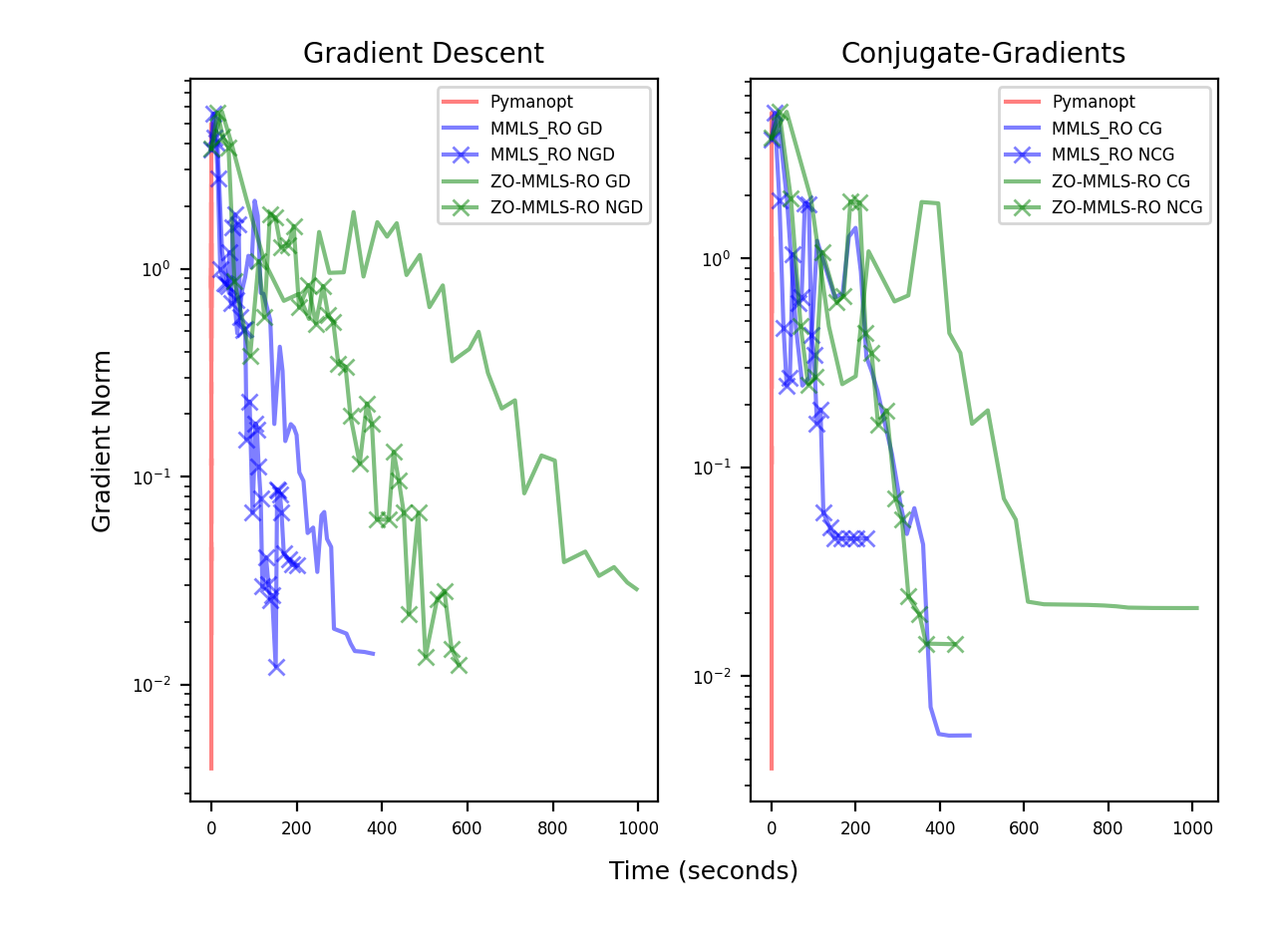}
\end{tabular}
\par\end{centering}
\caption{Suboptimality versus iteration count (top left) and versus time (bottom left), (Approximate-)Riemannian gradient norm versus iteration count (top right) and versus time (bottom right), for Problem \eqref{eq:eig_Stiefel4D}.\label{fig:Stiefel4D}}
\end{figure}

Next, we solve the following PCA problem of finding the two leading principal vectors of a randomly generated matrix ,$\matA_{2}\in \R^{200 \times 3}$, thus the constraining manifold is $\stiefel(3,2)$. Explicitly,
\begin{equation}\label{eq:PCA_Stiefel3D}
    \min_{\matX\in\stiefel(3,2)} \FNormS{\matA_{3} - \matA_{3} \matX \matX^{\T}},
\end{equation}
where $\FNorm{\cdot}$ denotes the Frobenius norm. As in the previous experiments on the Stiefel manifold, our algorithms received column stack vectors, while the inputs to the cost function were reshaped back into the corresponding matrix form. The ambient dimension in Problem \eqref{eq:PCA_Stiefel3D} is $D=6$ and the intrinsic dimension is $d=3$. We use $n=42875$ samples of $\stiefel(3,2)$, and polynomial approximations of degree $m=6$. The noise is an additive Gaussian noise $\mathcal{N}(0, 10^{-4})$, added to both the samples of $\stiefel(4,2)$, and the cost function samples in Problem \eqref{eq:PCA_Stiefel3D}. The results are presented in Fig. \ref{fig:Stiefel3D_PCA}. The left plots, present the relative error in Frobenius norm of the projection on the range spanned by the two leading principal vectors of $\matA_{3}$, i.e., $\FNorm{\matV \matV^{\T} - \matX \matX^{\T}} / \FNorm{\matV \matV^{\T}}$ where $\matV\in\R^{3 \times 2}$ is a matrix with the two leading principal vectors in its columns, versus iteration count (top) and versus time (bottom). The right plots present (approximate-)Riemannian gradient norms versus iteration count (top) and versus time (bottom).

\begin{figure}[tb]
\begin{centering}
\begin{tabular}{ccc}
\centering{}\includegraphics[scale=0.45]{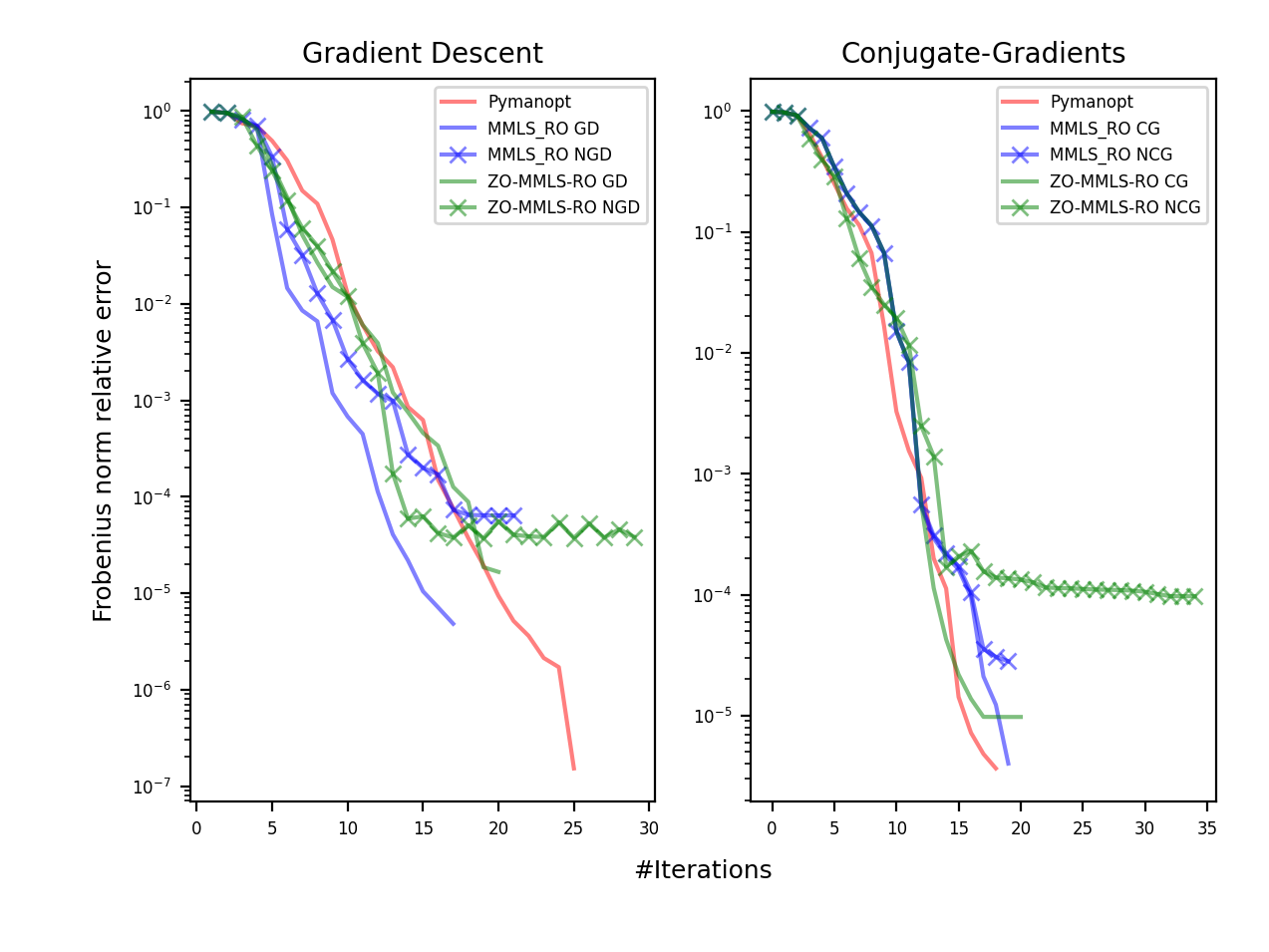} & ~ & \includegraphics[scale=0.45]{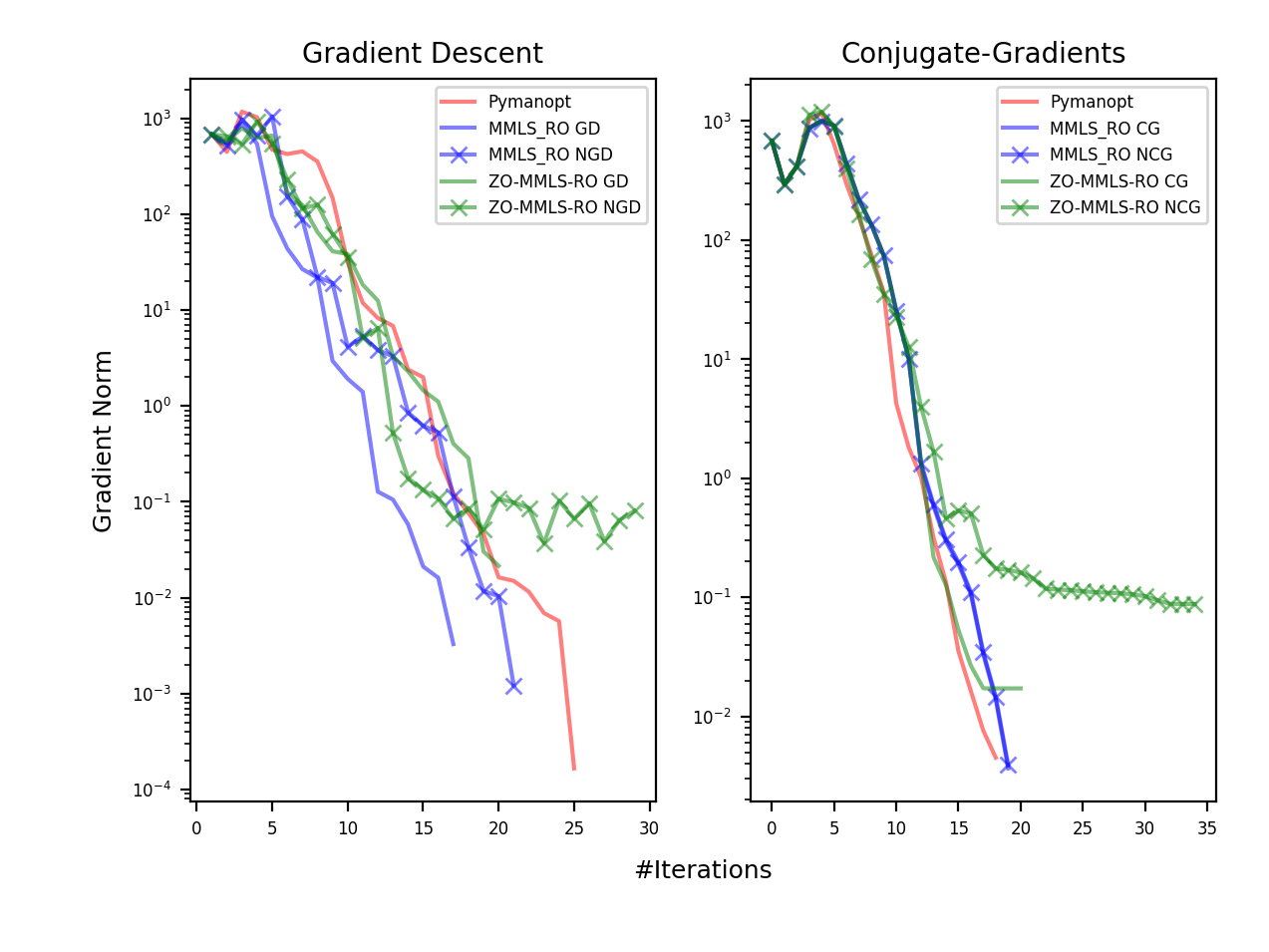} \\ \centering{}\includegraphics[scale=0.45]{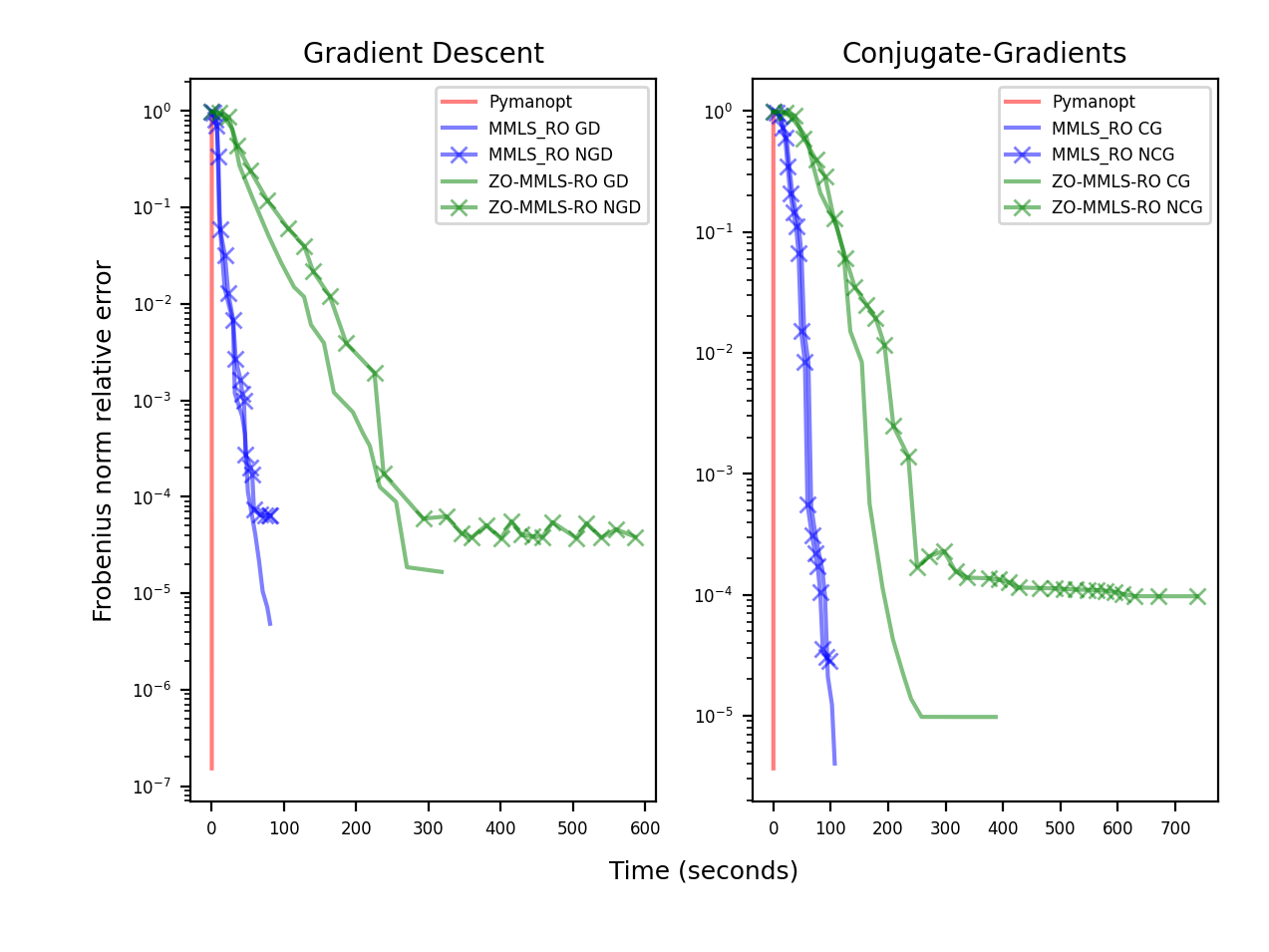} & ~ & \includegraphics[scale=0.45]{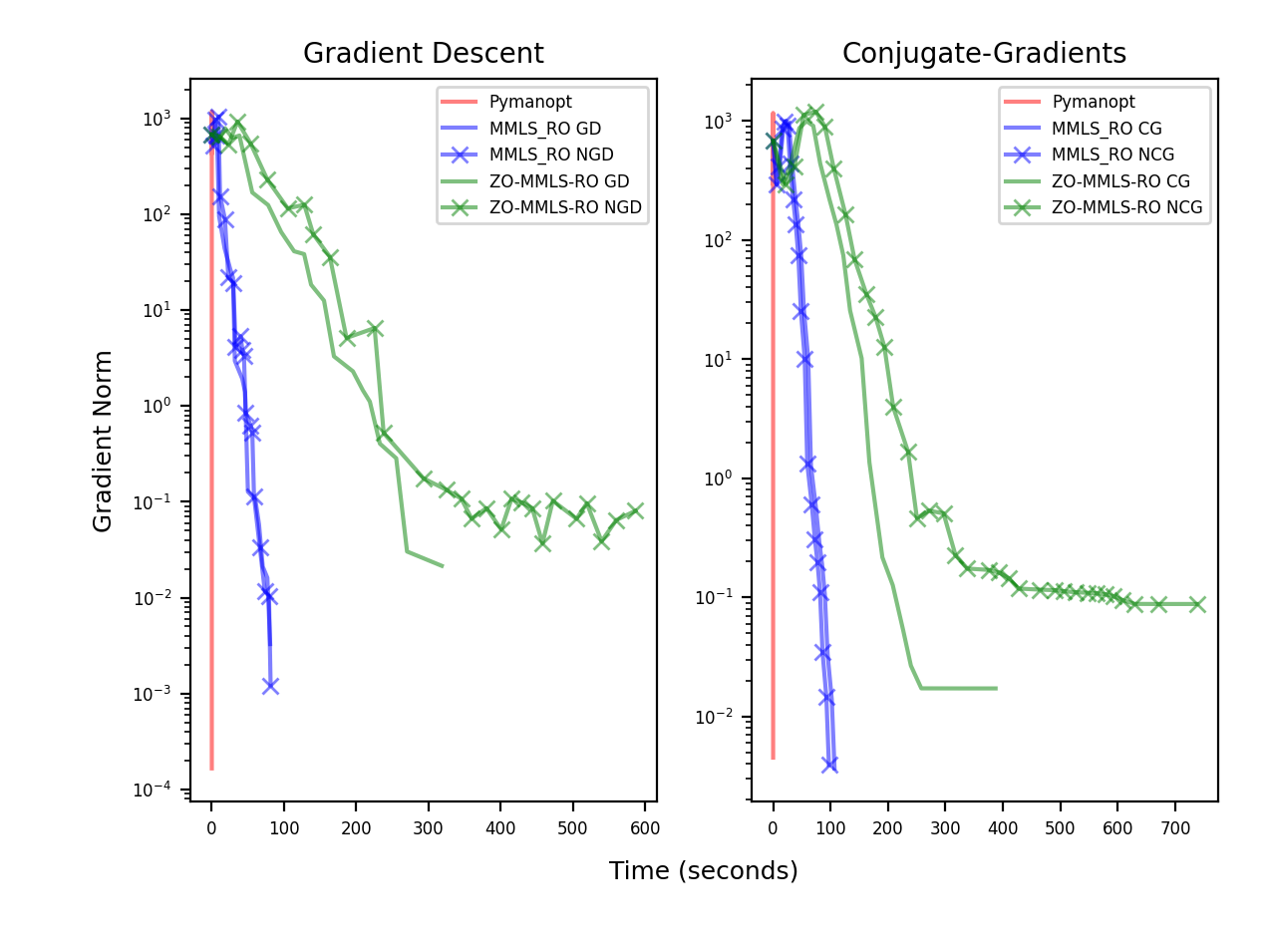}
\end{tabular}
\par\end{centering}
\caption{Relative error in Frobenious norm versus iteration count (top left) and versus time (bottom left), (Approximate-)Riemannian gradient norm versus iteration count (top right) and versus time (bottom right), for Problem \eqref{eq:PCA_Stiefel3D}.\label{fig:Stiefel3D_PCA}}
\end{figure}

The last problem we solve is a low-rank approximation of a randomly generated matrix, $\matA_{4}\in \R^{2 \times 2}$, thus the constraining manifold is a fixed-rank  manifold. Explicitly,
\begin{equation}\label{eq:low_rank_approx}
    \min_{\matX\in \R^{2 \times 2}_{1}} \FNormS{\matX - \matA_{4}},\ \matA_{4} \coloneqq \left(\begin{array}{cc}
-0.13 & -0.24\\
-0.49 & 0.11
\end{array}\right),\ \R^{2 \times 2}_{1} \coloneqq \{\matX\in \R^{2 \times 2}\ |\ \rank{\matX} = 1 \}.
\end{equation}
Similarly to the experiments on the Stiefel manifold, our algorithms received column stack vectors, while the inputs to the cost function were reshaped back into the corresponding matrix form. The ambient dimension in Problem \eqref{eq:low_rank_approx} is $D=4$ and the intrinsic dimension is $d=3$. We use $n=100000$ samples of $\R^{2 \times 2}_{1}$, and polynomial approximations of degree $m=12$. The noise is an additive Gaussian noise $\mathcal{N}(0, 10^{-4})$, added to both the samples of $\R^{2 \times 2}_{1}$, and the cost function samples in Problem \eqref{eq:low_rank_approx}. The results are presented in Fig. \ref{fig:low_rank}. The left plots, present the relative error in Frobenius norm of a rank-$1$ approximation of $\matA_{4}$, i.e., $\FNorm{\matX - \matU_{\matA_{4}}\mat{\Sigma}_{\matA_{4}}\matV_{\matA_{4}}^{\T}} / \FNorm{\matU_{\matA_{4}}\mat{\Sigma}_{\matA_{4}}\matV_{\matA_{4}}^{\T}}$ where $\matU_{\matA_{4}}\mat{\Sigma}_{\matA_{4}}\matV_{\matA_{4}}^{\T}$ is a rank-$1$ truncated SVD decomposition of $\matA_{4}$, versus iteration count (top) and versus time (bottom). The right plots present (approximate-)Riemannian gradient norms versus iteration count (top) and versus time (bottom).

\begin{figure}[tb]
\begin{centering}
\begin{tabular}{ccc}
\centering{}\includegraphics[scale=0.45]{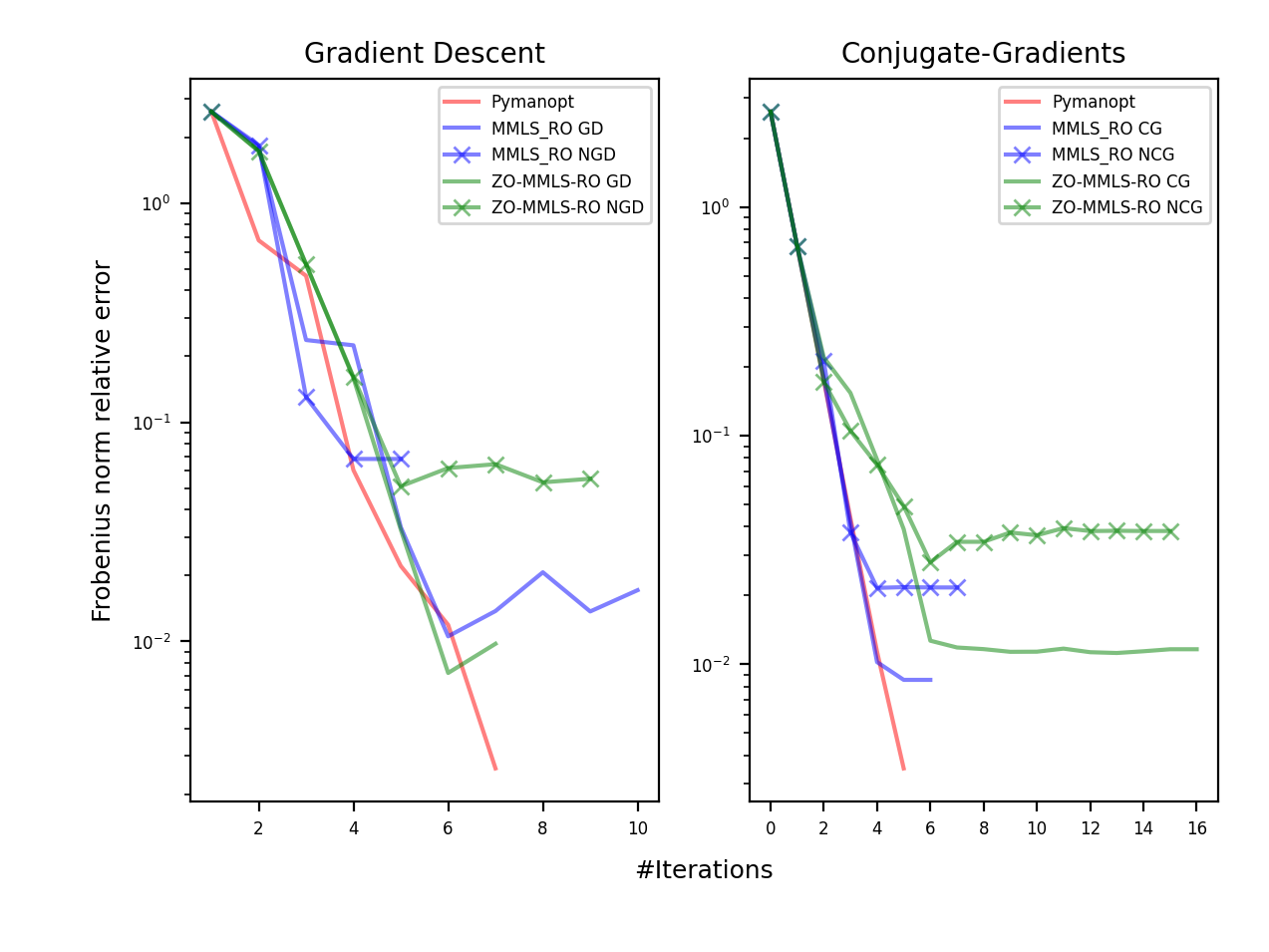} & ~ & \includegraphics[scale=0.45]{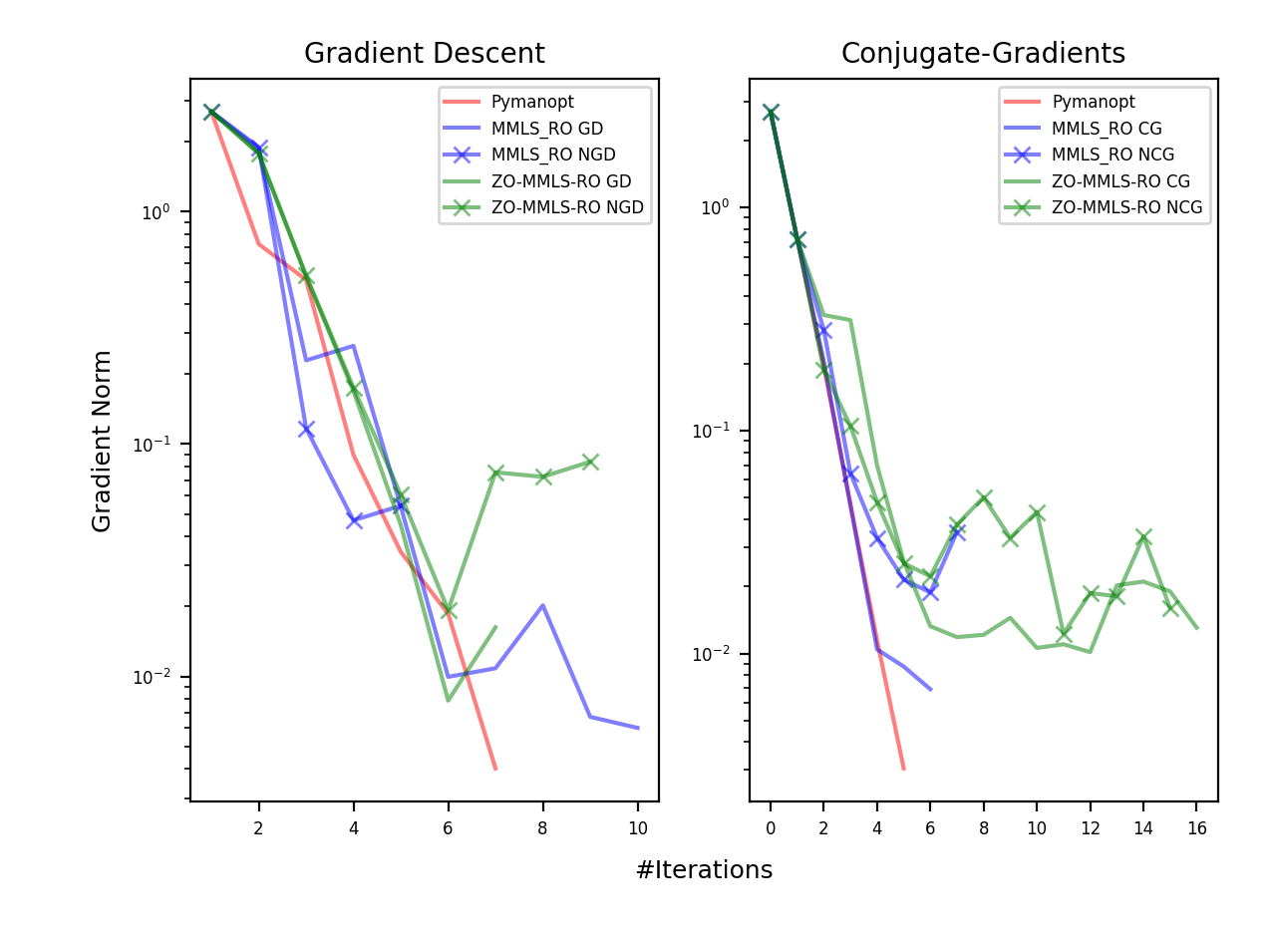} \\ \centering{}\includegraphics[scale=0.45]{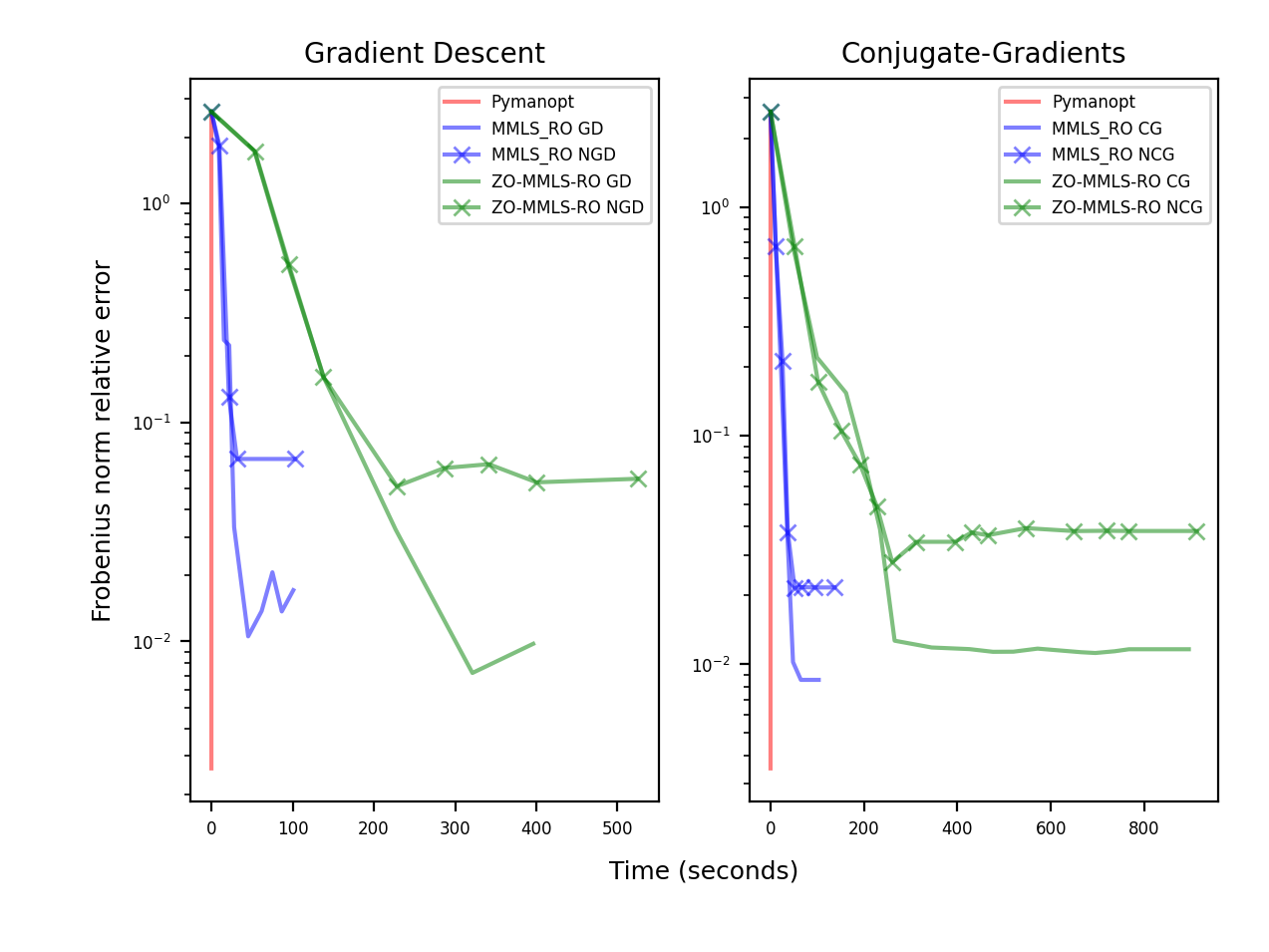} & ~ & \includegraphics[scale=0.45]{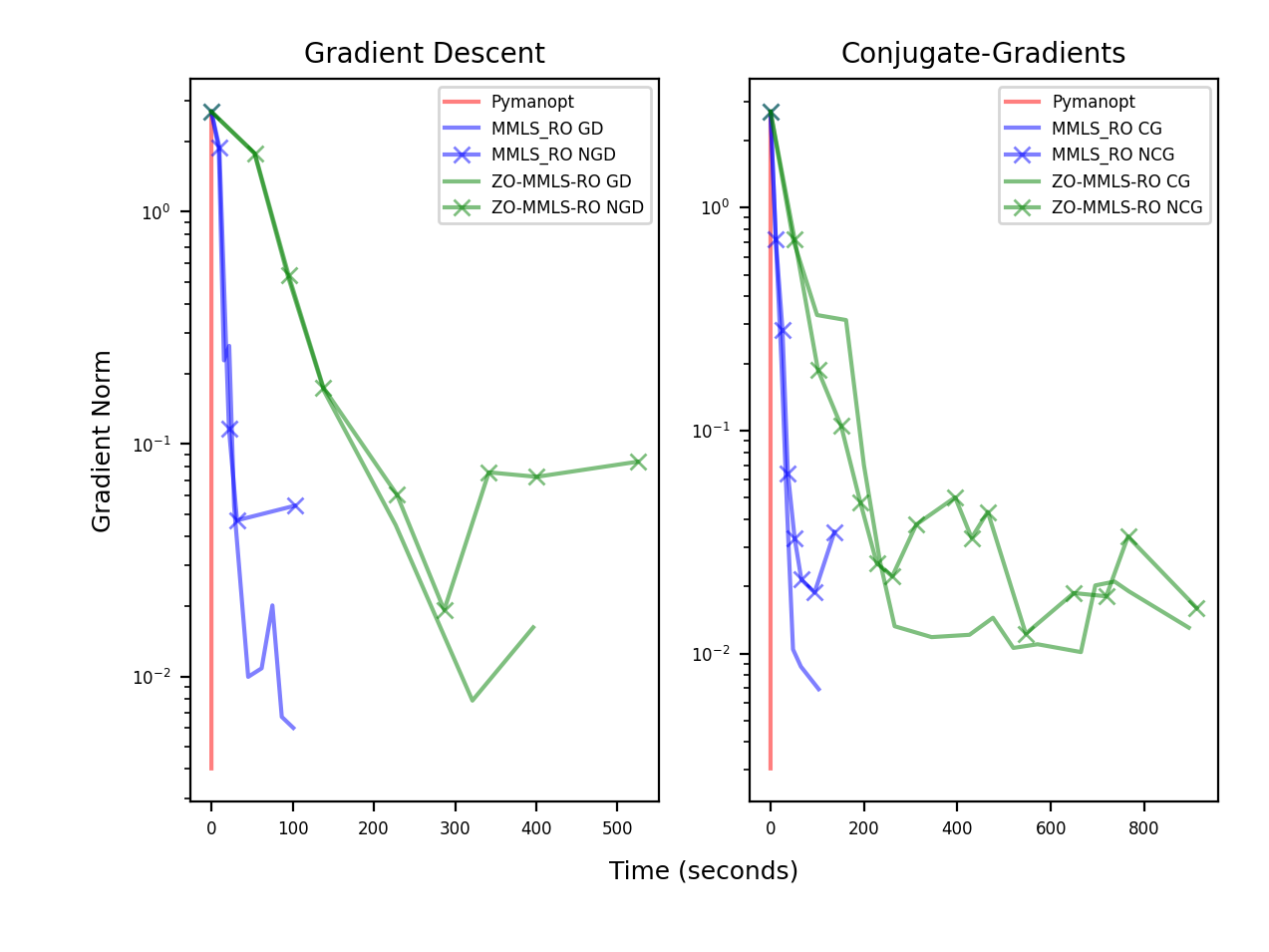}
\end{tabular}
\par\end{centering}
\caption{Relative error in Frobenious norm versus iteration count (top left) and versus time (bottom left), (Approximate-)Riemannian gradient norm versus iteration count (top right) and versus time (bottom right), for Problem \eqref{eq:low_rank_approx}.\label{fig:low_rank}}
\end{figure}

To conclude this subsection, we discuss our results. In all the experiments in this subsection, we demonstrate that our proposed algorithms obtain comparable results with respect to \textsc{PYMANOPT} solvers which have full knowledge of both the constraint and the cost function, even in the presence of noise in the samples of the corresponding constraining manifold, and the cost function. 

There are two phenomena we observe in our experiments which we would like to address. First, Note that in some cases our solvers achieve better errors than the ones obtained by \textsc{PYMANOPT}, but these cases may arise since our algorithms only approximately satisfy the constraints in the optimization, thus might reach a point with a "better" cost function value, but which does not satisfy the constraints. However, one should note that our algorithm does not have explicit access to the constraints, so it is unreasonable to expect it to uphold them exactly. The second phenomena we observe, is that in some of the experiments there are "bumps" in the suboptimality or the error values, especially when the values are small. This phenomena arise from the following reason. The specific implementation in \href{https://github.com/aizeny/manapprox/blob/main/manapprox/ManApprox.py}{MMLS}, begins MMLS algorithm with checking if the point (which is taken as the initial point in solving Problem \eqref{eq:MMLS_step1}, see \cite[Section 3.2]{sober2020manifold}) we wish to project using MMLS, has neighboring data points with respect to the weight function, to approximately make sure that Constraint \ref{item:MMLS_step1_3} from Problem \eqref{eq:MMLS_step1} is met for the initial point. If a point has no neighbors, then the nearest data point with respect to the weight function is taken as the initial point for MMLS algorithm. If such point does not exist, MMLS stops. Changing the point from which our optimization algorithm takes a step leads to unpredictable cost function behaviour. Adding more samples may aid in avoiding such occurrences. 

\section{\label{sec:conclusions}Conclusions}
In this paper, we propose a novel approach for solving optimization problems constrained on manifolds with limited information or accessibility to the constraining manifold and possibly the cost function itself. Our approach is based on approximating the missing geometric components required for Riemannian optimization using a manifold learning technique, MMLS \cite{sober2020manifold}, but can be extended to other techniques as well, e.g., our analysis in Subsection \ref{subsec:Convergence-Analysis-of1} is general and not bound to MMLS. Using these approximated components, we propose a variant of Riemannian gradient-descent algorithm and a variant of Riemannian CG algorithm both with a backtracking procedure. We analyze the guarantees and the costs of the proposed approximations, and also study the global convergence of our proposed Riemannian gradient-descent variants. Finally, we demonstrate numerically the effectiveness of our algorithms (even in the presence of noise), comparing their performance with respect to solvers having the exact components and full access to the cost function.

The aim of this paper is to introduce the potential of solving optimization problems constrained on manifolds where there is only limited information or accessibility to the constraining manifold using our approach. There are many possible future research directions following our work. One important aspect would be to apply this approach for solving real-world problems. For example solving optimizations problems that are constrained to satisfy a differential equation which is too costly to solve, or performing optimization on a manifold unknown explicitly such as a set of images of some object. Additional important aspect from the algorithmic point of view, is to develop an approximation of the Riemannian Hessian, enabling second-order algorithms (e.g., Newton and trust-regions) and their analysis. In addition, extending our proposed approach to other manifold learning techniques other than MMLS would allow to widen the scope of possible approximations and theory. From the theoretical standpoint, analysis for noisy sample sets of the cost function and
the constraining manifold, and analysis of additional variants of Riemannian optimization methods based on our proposed components (e.g., Riemannian CG) is required. Moreover, with a refined theory for the approximate-retraction a more precise analysis of the global convergence of the optimization methods would be possible.

\section*{Acknowledgements}
This research was supported by the Israel Science Foundation (grant no. 1272/17).

\bibliographystyle{amsplain}
\bibliography{ellipsoid}

\appendix
\section{\label{sec:miss_proof}Missing Proofs}
\subsection{\label{subsec:proof_lipschitz}Proof of Lemma \ref{lem:pullback_Lip}}
\begin{proof}
guarantees$\x,\y\in\text{Conv}(\widetilde{\mathcal{M}} \cup {\mathcal{M}})$ we have
\begin{equation*}
    | f(\y)-[ f(\x)+\dotprod{\nabla  f(\x)}{\y-\x}]|\leq\frac{L}{2}\|\y-\x\|^{2}.
\end{equation*}
In particular, take $\x = \widetilde{R}_{\rb}(\0_{\rb}) = \rb\in \widetilde{\mathcal{M}}$, $\y=\widetilde{R}_{\rb}(\xi)$ with $\xi\in \widetilde{T}_{\rb}\widetilde{\mathcal{M}}$, $\|\xi\|\leq Q$ to have
\begin{equation}\label{eq:lip_proof1}
    |{f}(\widetilde{R}_{\rb}(\xi))-[{f}(\rb)+\dotprod{\nabla  f(\rb)}{\widetilde{R}_{\rb}(\xi)-\rb}]|\leq\frac{L}{2}\|\widetilde{R}_{\rb}(\xi)-\rb\|^{2}.
\end{equation}

Next, we write the inner product above using $\v'_{\rb, \xi}(0)$ from Lemma \ref{lem:retraction_2_property}, and the definition of of the Riemannian gradient on $\widetilde{\mathcal{M}}$ as the orthogonal projection of the Euclidean gradient on its tangent space, in the following way
\begin{eqnarray}\label{eq:lip_proof2}
     \dotprod{\nabla  f(\rb)}{\widetilde{R}_{\rb}(\xi)-\rb} &=& \dotprod{\nabla  f(\rb)}{\xi+\v'_{\rb, \xi}(0)+\widetilde{R}_{\rb}(\xi)-\rb-\xi-\v'_{\rb, \xi}(0)} \nonumber\\
     &=& \dotprod{\gradMtil{f(\rb)}}{\xi+\v'_{\rb, \xi}(0)} + \nonumber\\ &+& \dotprod{\nabla  f(\rb)}{\widetilde{R}_{\rb}(\xi)-\rb-\xi-\v'_{\rb, \xi}(0)}.
\end{eqnarray}
Thus, using Eq. \eqref{eq:lip_proof1} and Eq. \eqref{eq:lip_proof2} yields
\begin{eqnarray*}
     |{f}(\widetilde{R}_{\rb}(\xi))-[{f}(\rb)+\dotprod{\gradMtil{f(\rb)}}{\xi+\v'_{\rb, \xi}(0)}]| &\leq& \frac{L}{2}\|\widetilde{R}_{\rb}(\xi)-\rb\|^{2} +\\
     &+& \|\nabla  f(\rb)\|\cdot\|\widetilde{R}_{\rb}(\xi)-\rb-\xi-\v'_{\rb, \xi}(0)\|.
\end{eqnarray*}

Since $\widetilde{\mathcal{M}}$ is a compact manifold and $\nabla  f(\rb)$ is assumed to be Lipschitz
continuous, there exists a finite $G>0$ such that $\|\nabla  f(\rb)\|\leq G$. In addition, $\widetilde{R}_{(\cdot)}(\cdot)$ is defined in a compact set of $K\subset \widetilde{T}\widetilde{\mathcal{M}}$ where $\rb\in\widetilde{\mathcal{M}}$ and $\xi\in \widetilde{T}_{\rb}\widetilde{\mathcal{M}}$ such that $\|\xi\|\leq Q$ (Assumption \ref{assu:ret_well_defined} and Eq. \eqref{eq:ret_def_in_tan_bun}). Thus, using 
$$\widetilde{R}_{\rb}(\xi) = \rb + \xi + \v'_{\rb, \xi}(0) + O(\|\xi\|^{2}),$$
we will show that 
\begin{equation}\label{eq:lip_proof3}
    \|\widetilde{R}_{\rb}(\xi)-\rb\| \leq \alpha \|\xi\|,
\end{equation}
and 
\begin{equation}\label{eq:lip_proof4}
    \|\widetilde{R}_{\rb}(\xi)-\rb-\xi-\v'_{\rb, \xi}(0)\| \leq \beta \|\xi\|^{2},
\end{equation}
where $\alpha, \beta \geq 0$. Eq. \eqref{eq:lip_proof3} and Eq. \eqref{eq:lip_proof4} prove the lemma with
$$\widetilde{L} \coloneqq 2\left(\frac{L}{2} \alpha ^{2} + G \beta\right) . $$

To show Eq. \eqref{eq:lip_proof3}, we have that for all $(\rb,\xi)\in K$
\begin{eqnarray*}
     \left \|\widetilde{R}_{\rb}(\xi)-\rb\right \| \leq \int _{0} ^{1} \left \| \frac{d}{dt} \widetilde{R}_{\rb}(t\xi) \right \| dt &=& \int _{0} ^{1} \left \| \text{D}\widetilde{R}_{\rb}(t\xi) [\xi] \right \| dt \\
     & \leq & \max _{(\rb,\eta) \in K} \left \| \text{D}\widetilde{R}_{\rb}(\eta) \right \| \left \| \xi \right \|,
\end{eqnarray*}
where the maximum above exists since $\text{D}\widetilde{R}_{(\cdot)}(\cdot)$ is smooth in the compact set $K$. Thus, 
$$\alpha \coloneqq  \max _{(\rb,\eta) \in K} \left \| \text{D}\widetilde{R}_{\rb}(\eta) \right \| \left \| \eta \right \|.$$

Eq. \eqref{eq:lip_proof4} is proved in a similar manner,
\begin{eqnarray*}
     \left \|\widetilde{R}_{\rb}(\xi)-\rb-\xi-\v'_{\rb, \xi}(0)\right \| &\leq& \int _{0} ^{1} \left \| \frac{d}{dt} \left( \widetilde{R}_{\rb}(t\xi) -\rb - t\xi -t \v'_{\rb, \xi}(0) \right) \right \| dt \\ &=& \int _{0} ^{1} \left \| \text{D}\widetilde{R}_{\rb}(t\xi) [\xi] - \xi-\v'_{\rb, \xi}(0) \right \| dt \\
     & \leq & \frac{1}{2} \max _{(\rb,\eta) \in K} \left \| \text{D}^{2}\widetilde{R}_{\rb}(\eta) \right \| \left \| \eta \right \|^{2},
\end{eqnarray*}
where the last inequality follows from Lemma \ref{lem:retraction_2_property} and
\begin{eqnarray*}
     \left \| \text{D}\widetilde{R}_{\rb}(t\xi) [\xi] - \xi-\v'_{\rb, \xi}(0) \right \| &\leq& \int _{0} ^{1} \left \| \frac{d}{ds}\text{D}\widetilde{R}_{\rb}(st\xi) [\xi] \right \| ds \\
     &\leq& \int _{0} ^{1} \left \| \text{D}^{2}\widetilde{R}_{\rb}(st\xi) [t\xi] \right \| ds \left \| \xi \right \|\\
     &\leq& \max _{(\rb,\eta) \in K} \left \| \text{D}^{2}\widetilde{R}_{\rb}(\eta) \right \| \left \| t \xi \right \| \left \| \xi \right \|.
\end{eqnarray*}
The maximum above exists since $\text{D}^{2}\widetilde{R}_{(\cdot)}(\cdot)$ is smooth in the compact set $K$. Thus, 
$$\beta \coloneqq  \frac{1}{2} \max _{(\rb,\eta) \in K} \left \| \text{D}^{2}\widetilde{R}_{\rb}(\eta) \right \|.$$

\end{proof}

\section{Background Materials and Additional Claims}
In this section we provide additional background material which is required for some of the proofs in this paper. Our aim is to make the paper  self contained.

\subsection{\label{subsec:background_MMLS}MMLS}
In this subsection we recall some useful claims on MMLS procedure \cite{sober2020manifold}, and in particular properties of its differential \cite[Section 2.1 and Section 3.1]{sober2020approximating}. First, we recall the order of approximation of the derivatives of an MLS approximation \cite[Lemma 1]{sober2020approximating} (we write here a simplified version since we only need first-order derivatives):
\begin{lemma}[Lemma 1 in \cite{sober2020approximating}]\label{lem:Sober_app_lemma1}
Let $f\in C^{k}(\R^{D})$ be a scalar valued function. Let $X=\{\x_{1},...,\x_{n}\}\subset \Omega \subset \R^{D}$ be a quasi-uniform unbounded sample set. Suppose $\theta_{h}(\cdot)\in C^{k}$ in 
\begin{equation}\label{eq:lemma1_prob}
    \pi^{\star}(\x\ |\ \xi) = \arg \min_{\pi \in \Pi_{m}^{d}} \sum_{i=1}^{n} \left|f(\x_i)-\pi(\x_i)\right|^2 \theta_{h} (\left\| \xi-\x_i \right\|),
\end{equation}
is compactly supported on $[0,sh]$ and consistent across scales, i.e., $\theta_{h}(th) = \Phi(t)$. Then, if for a fixed but arbitrary $\widehat{\x}\in \Omega$ the problem in Eq. \eqref{eq:lemma1_prob} has a unique solution (i.e., the least-squares matrix is invertible). We get for all $\x\in B_{sh}(\widehat{\x})$ 
\begin{equation*}
    \forall 1 \leq i \leq d,\ \left|\partial_{i}s_{f,X}(\x) - \partial_{i} f(\x) \right| \leq C \max_{1 \leq j \leq d,\ \x \in B_{sh}(\widehat{\x})} \left| \partial_{j} f(\x) \right| h^{m},
\end{equation*}
where $\partial_{i}$ is the partial derivative with respect to the $i$-th variable, $C$ is some constant independent of $f$ and $h$, and $s_{f,X}(\widehat{\x}) \coloneqq \pi^{\star}(\0\ |\ \widehat{\x})$.
\end{lemma}

Lemma \ref{lem:Sober_app_lemma1} can be used to formulate a similar result as \cite[Lemma 4]{sober2020approximating} for the derivatives of function approximation using MMLS \cite{sober2021approximation}. Explicitly, we formulate Lemma \ref{lem:fderiv_approx}. Before formulating it, we recall the \emph{Injectivity Conditions} required for proving properties of the derivatives of MMLS approximation:
\begin{enumerate}
    \item The functions $\theta_{1}(\cdot)$ and $\theta_{h}(\cdot)$ are monotonically decaying and supported on $[0,c_{1}h]$ where $c_{1}>3$.
    \item Suppose that $\theta_{1}(c_{2}h)>c_{3}>0$ and $\theta_{h}(c_{2}h)>c_{3}>0$ for some constants $c_{2}<c_{1}$.
    \item Set $\mu = \reach{\mathcal{M}}/2$ in the second constraint of the optimization in Eq. \eqref{eq:MMLS_step1}.
\end{enumerate}

\begin{lemma}[Euclidean gradient approximation order]\label{lem:fderiv_approx}
Let Assumption \ref{assu:MMLS_func_Samples} and the Injectivity Conditions from this section hold. Also, let the weight function from Problem \eqref{eq:step2-func_approx} satisfy $\theta_{3}(\cdot)\in C^{\infty}$ and $\lim _{t\to 0} \theta_{3} (t) = \infty$. Then, there exists a constant $h_{0}$ such that for all $h\leq h_{0}$, all $\p \in \mathcal{M}$, and for all $\x\in B_{c_{1}h}(\0)\subset H(\p)$
\begin{equation}\label{eq:fderiv_approx}
    \frac{1}{\sqrt{D}}\|\nabla p_{\p}^{f}(\x)-\nabla\widehat{f}(\x)\|\leq
    \|\nabla p_{\p}^{f}(\x)-\nabla\widehat{f}(\x)\|_{\infty} \leq c_{f} h^{m},
\end{equation}
where $p_{\p}^{f}(\x)$ is an approximation of $f(\p)$, i.e., the solution of Problem \eqref{eq:step2-func_approx}, $\widehat{f}=f \circ \varphi$, and $c_{f}$ is a constant independent of $\p$.
\end{lemma}
\begin{proof}
The proof of this lemma is similar to the proof of \cite[Lemma 4]{sober2020approximating}. Recall that MMLS algorithm and its extension for function approximation differ only in the second step of the algorithms, i.e., Eq. \eqref{eq:step2-func_approx} and Eq. \eqref{eq:MMLS_step2} correspondingly. Explicitly, solving the problem in Eq. \eqref{eq:MMLS_step2} provides a polynomial $p_{\p}^{f}(\x)$ of total degree $m$ which approximates $\widehat{f} = f \circ \varphi$, where $f:\mathcal{M}\to\R$ is the cost function and $\varphi:H(\rb)\to\mathcal{M}$ is a parametrization of $\mathcal{M}$ such that $\varphi(\x_i) = \rb_i$ for $1\leq i \leq n$. 

In addition, note that the polynomial $p_{\p}^{f}(\x)$ coincides with the
polynomial $\pi^{\star}(\x\ |\ \xi)$ which minimizes the problem in Eq. \eqref{eq:lemma1_prob} with respect to the domain $H(\p)$. Moreover, the MLS approximation is exact for polynomials (see \cite[Section 2.1]{sober2020approximating}), and in particular it reproduces the Taylor polynomial of degree $m$ of $\widehat{f}$. Thus, from Lemma \ref{lem:Sober_app_lemma1} and the triangle inequality (using the Taylor expansion), we get that for all $1 \leq i \leq d$ and for all $\x\in B_{c_{1}h}(\0)\subset H(\p)$
\begin{equation*}
    \left|\partial_{i}p_{\p}^{f}(\x) - \partial_{i} \widehat{f}(\x) \right| \leq C \max_{1 \leq j \leq d,\ \x \in B_{sh}(0)} \left| \partial_{j} \widehat{f}(\x) \right| h^{m},
\end{equation*}
leading to the desired bound in Eq. \eqref{eq:fderiv_approx} with
\begin{equation*}
    c_{f} =  C \max_{1 \leq j \leq d,\ \x \in B_{sh}(0)} \left| \partial_{j} \widehat{f}(\x) \right|,\ 
\end{equation*}
where 
$$\max_{1 \leq j \leq d,\ \x \in B_{sh}(0)} \left| \partial_{j} \widehat{f}(\x) \right|,$$
exists since $\widehat{f}$ varies smoothly with $\p$ (since $H(\p)$ depends smoothly on $\p$ \cite[Theorem 4.12]{sober2020manifold}), and since we assume that $\mathcal{M}$ is compact.
\end{proof}

Next, we recall \cite[Lemma 3]{sober2020approximating} and \cite[Lemma 4]{sober2020approximating}:

\begin{lemma}[Lemma 3 from \cite{sober2020approximating}]\label{lem:Sober_app_lemma3}
Let Assumption \ref{assu:MMLS_Samples} and the Injectivity Conditions from this section hold. Then, for all $h\leq h_{0}$ and all $\p\in \mathcal{M}$ we have:
\begin{enumerate}
    \item There exists a unique ${\mathcal P}_{m}^{h}(\p)=\widetilde{\p}\in \widetilde{\mathcal{M}}$ and a unique coordinate domain $(\q(\p),H(\p))$.
    \item There exist neighborhoods $W_{\p}\subset\mathcal{M}$, $U\subset H(\p)$ of $\p$, $\q(\p)$ correspondingly such that 
    $$\varphi:U\to\mathcal{M}\subset\R^{D},\ \widetilde{\varphi}:U\to \widetilde{\mathcal{M}}\subset \R^{D}, $$
    and
    $$\varphi[U]=W_{\p} \subset \mathcal{M},\ \widetilde{\varphi}[U]=W \subset \widetilde{\mathcal{M}}. $$
    \item Furthermore,
    $$\varphi(\0)=\p,\ \widetilde{\varphi}(\0)=g^{\star}(\0\ |\ \p)=\widetilde{\p}. $$
\end{enumerate}
\end{lemma}
\begin{lemma}[Lemma 4 from \cite{sober2020approximating}]\label{lem:Sober_app_lemma4}
Let Assumption \ref{assu:MMLS_Samples} and the Injectivity Conditions from this section hold. Also, let $\lim _{t\to 0} \theta_{h} (t) = \infty$, i.e., $\widetilde{\mathcal{M}}$ interpolates $\mathcal{M}$ at the sample points $\rb_{i},\ 1\leq i\leq n$. Then, there exists a constant $h_{0}$ such that for all $h\leq h_{0}$, all $\p \in \mathcal{M}$, any direction $\v\in\R^{d}$, and all $\x\in B_{c_{1}h}(\0)\subset H(\p)$, we have
\begin{eqnarray*}
     \frac{1}{\sqrt{D}}\|\text{D}\varphi(\x)[\v]-\text{D}\widetilde{\varphi}(\x)[\v]\|\leq
    \|\text{D}\varphi(\x)[\v]-\text{D}\widetilde{\varphi}(\x)[\v]\|_{\infty}&\leq& c_{\mathcal{M}, \widetilde{\mathcal{M}}} h^{m},\nonumber\\
    \frac{1}{\sqrt{D}}\|\text{D}g^{\star}(\x\ |\ \p)[\v]-\text{D}\varphi(\x)[\v]\|\leq
    \|\text{D}g^{\star}(\x\ |\ \p)[\v]-\text{D}\varphi(\x)[\v]\|_{\infty}&\leq& c_{\mathcal{M}} h^{m},\\
    \frac{1}{\sqrt{D}}\|\text{D}g^{\star}(\x\ |\ \p)[\v]-\text{D}\widetilde{\varphi}(\x)[\v]\|\leq
    \|\text{D}g^{\star}(\x\ |\ \p)[\v]-\text{D}\widetilde{\varphi}(\x)[\v]\|_{\infty}&\leq& c_{\widetilde{\mathcal{M}}} h^{m},\nonumber
\end{eqnarray*}
where $c_{\mathcal{M}, \widetilde{\mathcal{M}}}, c_{\mathcal{M}}, c_{\widetilde{\mathcal{M}}}$ are constants independent of $\v$ or $\p$.
\end{lemma}
As explained in Subsection \ref{subsec:MMLS_the_algorithm}, Lemma \ref{lem:Sober_app_lemma3} and Lemma \ref{lem:Sober_app_lemma4} can be extended for any $\rb\in\widetilde{\mathcal{M}}$ such that ${\mathcal P}_{m}^{h}(\p)=\rb$. The same extension works for Lemma \ref{lem:fderiv_approx}. 

A key outcome from Lemma \ref{lem:Sober_app_lemma4} is related to the order of approximation of our proposed Riemannian gradients, Eq. \eqref{eq:f_Rgrad} and Eq. \eqref{eq:approximate_f_Rgrad}, which we prove in Lemma \ref{lem:approx_Rim_grad}. In the following, we show that for $\rb\in\widetilde{\mathcal{M}}$ such that ${\mathcal P}_{m}^{h}(\p)=\rb$, the orthogonal projections with respect to the Riemannian metric (the standard inner product in this paper) on $\Range{\text{D}g^{\star}(\x\ |\ \rb)}$, $\Range{\text{D}\varphi(\x)}$, and $\Range{\text{D}\widetilde{\varphi}(\x)}$ ($\widetilde{T}_{\rb}\widetilde{\mathcal{M}}$, $T_{\p}{\mathcal{M}}$, and $T_{\rb}\widetilde{{\mathcal{M}}}$ when $\x = \0$ correspondingly) where $g^{\star}$, $\varphi$, and $\widetilde{\varphi}$ are the functions from Lemma \ref{lem:Sober_app_lemma4}, approximate each other in the order of $O(\sqrt{D} h^{m})$ in $L_{2}$ norm. To that end, we first recall the following result from \cite[Theorem 2.5]{chen2016perturbation}.
\begin{lemma}[Based on Theorem 2.5 from \cite{chen2016perturbation}]\label{lem:chen2.5}
Let $\matB, \widetilde{\matB} \in \R^{D\times d}$, and denote $\mat{E} \coloneqq \matB - \widetilde{\matB}$. If $\rank{\matB} = \rank{\widetilde{\matB}}$, then
\begin{equation}\label{eq:chen2.21}
    \left\|\Pi_{\Range{\matB}}(\cdot) - \Pi_{\Range{\widetilde{\matB}}}(\cdot) \right\| = \left\| \matB \matB^\pinv - \widetilde{\matB} \widetilde{\matB}^\pinv \right\| \leq \min \left\{ \left\| \mat{E} \matB^\pinv \right\|,\ \left\| \mat{E} \widetilde{\matB}^\pinv \right\| \right\},
\end{equation}
where $\Pi_{\Range{\matB}}(\cdot) = \matB \matB^\pinv$ and $\Pi_{\Range{\widetilde{\matB}}}(\cdot) = \widetilde{\matB} \widetilde{\matB}^\pinv$ are the orthogonal projection matrices on the column spaces of $\matB$ and $\widetilde{\matB}$ correspondingly in their explicit matrix form (see \cite[Chapter 5.5.2]{golub2013matrix}).
\end{lemma}

Now, using Lemma \ref{lem:chen2.5} we can prove our result regarding the orthogonal projections in Lemma \ref{lem:our_proj_prop}.

\begin{lemma}[Orthogonal projection approximation order]\label{lem:our_proj_prop}
Given $\rb\in\widetilde{\mathcal{M}}$ such that ${\mathcal P}_{m}^{h}(\p)=\rb$, let Assumption \ref{assu:MMLS_Samples}, let the Injectivity Conditions from this section hold, and let $\lim _{t\to 0} \theta_{h} (t) = \infty$. Then, for all $\x\in B_{c_{1}h}(\0)\subset H(\rb)$, we have
\begin{eqnarray}\label{eq:our_proj_prop1}
    \left\|\Pi_{\Range{\text{D}\varphi(\x)}}(\cdot) - \Pi_{\Range{\text{D}\widetilde{\varphi}(\x)}}(\cdot) \right\| &\leq& c_{\mathcal{M}, \widetilde{\mathcal{M}}} \sqrt{D} h^{m},\\
    \left\|\Pi_{\Range{\text{D}g^{\star}(\x\ |\ \rb)}}(\cdot) - \Pi_{\Range{\text{D}\varphi(\x)}}(\cdot) \right\| &\leq& c_{\mathcal{M}} \sqrt{D}  h^{m},\nonumber\\
    \left\|\Pi_{\Range{\text{D}g^{\star}(\x\ |\ \rb)}}(\cdot) - \Pi_{\Range{\text{D}\widetilde{\varphi}(\x)}}(\cdot) \right\| &\leq& c_{\widetilde{\mathcal{M}}} \sqrt{D} h^{m},\nonumber
\end{eqnarray}
and
\begin{eqnarray}\label{eq:our_proj_prop2}
    \left\|\Pi_{T_{\p}{\mathcal{M}}}(\cdot) - \Pi_{T_{\rb}\widetilde{{\mathcal{M}}}}(\cdot) \right\| &\leq& c_{\mathcal{M}, \widetilde{\mathcal{M}}} \sqrt{D} h^{m},\\
    \left\|\Pi_{\widetilde{T}_{\rb}\widetilde{\mathcal{M}}}(\cdot) - \Pi_{T_{\p}{\mathcal{M}}}(\cdot) \right\| &\leq& c_{\mathcal{M}} \sqrt{D} h^{m},\nonumber\\
    \left\|\Pi_{\widetilde{T}_{\rb}\widetilde{\mathcal{M}}}(\cdot) - \Pi_{T_{\rb}\widetilde{{\mathcal{M}}}}(\cdot) \right\| &\leq& c_{\widetilde{\mathcal{M}}} \sqrt{D} h^{m},\nonumber
\end{eqnarray}
where $\Pi_{\Range{\text{D}g^{\star}(\x\ |\ \rb)}}(\cdot)$, $\Pi_{\Range{\text{D}\varphi(\x)}}(\cdot)$, $\Pi_{\Range{\text{D}\widetilde{\varphi}(\x)}}(\cdot)$, $\Pi_{\widetilde{T}_{\rb}\widetilde{\mathcal{M}}}(\cdot)$, $\Pi_{T_{\p}{\mathcal{M}}}(\cdot)$, and $\Pi_{T_{\rb}\widetilde{{\mathcal{M}}}}(\cdot)$ are the orthogonal projection operators on $\Range{\text{D}g^{\star}(\x\ |\ \rb)}$, $\Range{\text{D}\varphi(\x)}$, $\Range{\text{D}\widetilde{\varphi}(\x)}$, $\widetilde{T}_{\rb}\widetilde{\mathcal{M}}$, $T_{\p}{\mathcal{M}}$, and $T_{\rb}\widetilde{{\mathcal{M}}}$ correspondingly.
\end{lemma}

\begin{proof}
We want to show Eq. \eqref{eq:our_proj_prop1}, and it will also prove Eq. \eqref{eq:our_proj_prop2}, since the orthogonal projections on $\widetilde{T}_{\rb}\widetilde{\mathcal{M}}$, $T_{\p}{\mathcal{M}}$, and $T_{\rb}\widetilde{{\mathcal{M}}}$ are equivalent to the orthogonal projections on $\Range{\text{D}g^{\star}(\0\ |\ \rb)}$, $\Range{\text{D}\varphi(\0)}$, and $\Range{\text{D}\widetilde{\varphi}(\0)}$ correspondingly. 

For simplicity, we will use the matrix form of the operators, and denote each possible pair of matrices from: $\text{D}g^{\star}(\x\ |\ \rb)$, $\text{D}\varphi(\x)$, and $\text{D}\widetilde{\varphi}(\x)$, by $\matB$ and $\widetilde{\matB}$. Using Lemma \ref{lem:chen2.5}, Eq. \eqref{eq:chen2.21} holds for $\matB$ and $\widetilde{\matB}$. Thus, it is sufficient to bound
\begin{equation}\label{eq:our_chen2.21}
    \min \left\{ \left\| \mat{E} \matB^\pinv \right\|,\ \left\| \mat{E} \widetilde{\matB}^\pinv \right\| \right\}.
\end{equation}
Recall that Lemma \ref{lem:Sober_app_lemma4} ensures that 
\begin{equation}\label{eq:our_proj_prop_given}
    \left\| \matB \v - \widetilde{\matB} \v \right\| \leq c \sqrt{D}h^{m},
\end{equation}
for any $\v\in \R^{d}$ and some constant $c>0$ independent of $\v$. The right-hand side of Eq. \eqref{eq:our_proj_prop_given} does not depend on $\v$, thus we can take the maximum over $\left\| \v \right\| = 1$ and get the spectral matrix norm of $\matB - \widetilde{\matB}$, i.e.,
\begin{equation}\label{eq:our_proj_prop_given2}
   \left\| \mat{E} \right\| \coloneqq \left\| \matB - \widetilde{\matB} \right\| \leq c \sqrt{D}h^{m}.
\end{equation}

Next, recall that the output of the second step of MMLS, i.e., $g^{\star}(\x\ |\ \rb):H(\rb)\to\R^{D}$, can be equivalently viewed as $g^{\star}(\x\ |\ \rb):H(\rb)\to H^{\perp}(\rb)$, i.e., an approximation of $\mathcal{M}$ as a graph of a function (see Subsection \ref{subsec:MMLS_the_algorithm}). Correspondingly, $\varphi:H(\rb)\to H^{\perp}(\rb)$  and $\widetilde{\varphi}:H(\rb)\to H^{\perp}(\rb)$ are representations of $\mathcal{M}$ and $\widetilde{\mathcal{M}}$ as graphs of functions. Now, take a basis of $\R^{D}$ to be a union of some orthogonal bases of $H(\rb)$ and $H^{\perp}(\rb)$, then the differentials of $g^{\star}(\cdot\ |\ \rb)$, $\varphi(\cdot)$, and $\widetilde{\varphi}(\cdot)$ are of the form (see also \cite[Subsection 2.2.2]{aizenbud2021non}):
\begin{equation}\label{eq:our_proj_prop_mat}
    \matB \coloneqq \left [ \begin{array}{c} \matI_{d}\\
\matA \\
\end{array}\right]\in\R^{D\times d}, 
\end{equation}
where $\matA\in \R^{(D-d)\times d}$. Note that $\matB^\pinv = \matG_{\matB}^{-1} \matB^{\T}$ where 
\begin{equation}\label{eq:our_proj_prop_Gram}
    \matG_{\matB} \coloneqq \matB^{\T} \matB = \matI_{d} + \matA^{\T} \matA,
\end{equation}
is the Gram matrix of the matrix $\matB$. 

From Eq. \eqref{eq:our_proj_prop_Gram}, the eigenvalues of $\matG_{\matB}$ (symmetric positive semi-definite matrix by definition) are larger or equal to $1$ making the matrix SPD. Thus, the eigenvalues of $\matG_{\matB}^{-1}$, which is also an SPD matrix, are in the range $(0,1]$. Moreover, we have that
\begin{equation}\label{eq:our_proj_prop_norm}
    \left\|\matG_{\matB}^{-1} \right\|\leq 1,\ \left\|\matG_{\matB}^{-\nicehalf} \right\|\leq 1,
\end{equation}
where $\matG_{\matB}^{-\nicehalf}$ is the unique SPD matrix such that $\matG_{\matB}^{-1} = \matG_{\matB}^{-\nicehalf}\matG_{\matB}^{-\nicehalf}$. Eq. \eqref{eq:our_proj_prop_norm} holds since any SPD matrix $\matM\in\R^{d\times d}$ with eigenvalues in the range $(0,1]$, has a spectral norm $\left\| \matM \right\| \leq 1$, since the spectral norm of any matrix $\matU$ equals its largest singular value \cite[Example 5.6.6]{horn2012matrix}, which is also equal to the largest eigenvalue of $\matU\matU^{\T}$ (or $\matU^{\T}\matU$) \cite[Theorem 2.6.3]{horn2012matrix}.

Next, since $\matB \matG_{\matB}^{-1} \matB^{\T}$ is an orthogonal projection matrix, its eigenvalues are either $0$ or $1$, making the spectral norm of the following matrix be bounded by $1$:
\begin{equation}\label{eq:our_proj_prop_given3}
    \left\| \matG_{\matB}^{-\nicehalf} \matB^{\T} \right\| \leq 1.
\end{equation}

Finally, to conclude the proof we can bound Eq. \eqref{eq:our_chen2.21} using Eqs. \eqref{eq:our_proj_prop_given2}, \eqref{eq:our_proj_prop_norm}, and \eqref{eq:our_proj_prop_given3}, by
\begin{equation}\label{eq:our_final_bound}
\left\| \mat{E} \matB^\pinv \right\| \leq \left\| \mat{E} \right\| \cdot \left\| \matB^\pinv \right\| \leq \left\| \mat{E} \right\| \cdot \left\| \matG_{\matB}^{-\nicehalf} \right\| \cdot \left\| \matG_{\matB}^{-\nicehalf} \matB^{\T} \right\| \leq c \sqrt{D}h^{m},
\end{equation}
where the bound in Eq. \eqref{eq:our_final_bound} is also true for $\left\| \mat{E} \widetilde{\matB}^\pinv \right\|$.
\end{proof}

\subsection{\label{subsec:background_proj-like}Projection-Like Retractions}
In this section we recall some useful claims on projection-like retractions from Section 4 in \cite{absil2012projection}. We begin with Definition 14 from \cite{absil2012projection} of a \emph{retractor}:
\begin{defn}[Definition 14 from \cite{absil2012projection}]\label{def:retractor}
Let $\mathcal{M}$ be a $d$-dimensional submanifold of class $C^{k}$, where $k\geq 2$, of $\R^{D}$. A \textbf{retractor} on $\mathcal{M}$ is a $C^{k-1}$ mapping $A$ from the tangent bundle $T\mathcal{M}$ into the Grassmann manifold $\text{Gr}(D-d)$ of $\R^{D}$,
whose domain contains a neighborhood of the zero section of $T\mathcal{M}$ (submanifold of the bundle that consists of all the zero vectors), and such that, for all $\x\in\mathcal{M}$, the intersection of $A(\x, \0_{\x})$ and $T_{x}\mathcal{M}$ is trivial.
\end{defn}
Next, using the definition of a retractor it is possible to define a retraction following \cite[Theorem 15]{absil2012projection}:
\begin{thm}[Theorem 15 from \cite{absil2012projection}]\label{thm:retractorretraction}
Let D be a retractor (Definition \ref{def:retractor}) and, for all $(\x,\u)\in \text{dom}(A)$, define the affine space $\mathcal{A}(x,\u) = \x+\u + A(\x, \u)$.
Consider the point-to-set function $R:\text{dom}(A)\to \mathcal{M}$ such that $R(\x, \u)$ is the set of points of ${\mathcal{M}}\cap \mathcal{A} (\x, \v) $ nearest to $\x+\u$ (for a small neighborhood of $(\x, \0_{x})$ in $T\mathcal{M}$ the map $R$ maps to a singleton). Then $R$ is a retraction on $\mathcal{M}$.
The retraction $R$ thus defined is called the retraction \textbf{induced} by the retractor $A$.

\end{thm}

\end{document}